\documentclass[smallextended,numbook,review]{svjour3}
\usepackage[utf8]{inputenc} 
\usepackage[english]{babel} 
\usepackage[usenames, dvipsnames]{xcolor}
\usepackage[a4paper, tmargin=1.0in,bmargin=1.0in, lmargin=1in, rmargin=1in]{geometry} 
\usepackage{bm} 
\usepackage{amsmath} 
\usepackage{amssymb}
\usepackage{cases} 
\usepackage{tikz} 
\usepackage{stmaryrd} 
\usepackage{multirow} 
\usepackage{lipsum} 
\usepackage{wrapfig} 

\usepackage{breakcites} 
\usepackage{mathtools}   
\usepackage{cases}
\usepackage{float}
\usepackage{empheq} 
\usepackage{mdframed}

\newenvironment{summarybox}[1]{%
  \begin{mdframed}[
    linewidth=0.6pt,
    backgroundcolor=gray!10,
    skipabove=10pt,
    skipbelow=10pt,
    innerleftmargin=8pt,
    innerrightmargin=8pt,
    innertopmargin=6pt,
    innerbottommargin=6pt
  ]
  \textbf{#1}\par\medskip
}{%
  \end{mdframed}
}

\usepackage{multirow}
\usepackage{makecell}
\usepackage{booktabs}

\usepackage[colorlinks=true,linkcolor=OrangeRed,citecolor=RoyalBlue]{hyperref} 
\usepackage{cleveref} 

\crefname{paragraph}{paragraph}{paragraphs}
\Crefname{paragraph}{Paragraph}{Paragraphs}

\newcommand{\spaceVariable}{x}
\newcommand{\leftBoundary}{\spaceVariable_{\ell}}
\newcommand{\rightBoundary}{\spaceVariable_{r}}
\newcommand{\spaceStep}{\Delta \spaceVariable}
\newcommand{\indexSpace}{j}
\newcommand{\indexTime}{n}
\newcommand{\spaceGridPoint}[1]{\spaceVariable_{#1}}
\newcommand{\spaceGridPointOnY}[1]{y_{#1}}
\newcommand{\definitionEquality}{:=}
\newcommand{\numberGridPoints}{J}
\newcommand{\numberGridPointsOnY}{K}
\newcommand{\integerInterval}[2]{\llbracket #1, #2 \rrbracket}
\newcommand{\conservedMoment}{u}
\newcommand{\nonConservedMoment}{v}
\newcommand{\nonNonConservedMoment}{w}
\newcommand{\collided}{\star}
\newcommand{\strong}[1]{\emph{#1}}
\newcommand{\relaxationParameter}{\omega}
\usepackage{mathrsfs}  
\newcommand{\courantNumber}{\mathscr{C}}
\newcommand{\distributionFunctionLetter}{f}
\newcommand{\fourierShift}{\kappa}
\newcommand{\matricial}[1]{\bm{#1}}
\newcommand{\vectorial}[1]{\bm{#1}}
\newcommand{\schemeMatrixBulk}{\matricial{{E}}}
\newcommand{\schemeMatrixBulkFourier}{\schemeMatrixBulk}
\newcommand{\determinant}{\textnormal{det}}
\newcommand{\timeShiftOperator}{z}
\newcommand{\identityMatrix}[1]{\matricial{I}_{#1}}
\newcommand{\stableMarker}{\textnormal{s}}
\newcommand{\unstableMarker}{\textnormal{u}}
\newcommand{\stableRoot}{\fourierShift_{\stableMarker}}
\newcommand{\unstableRoot}{\fourierShift_{\unstableMarker}}
\newcommand{\eigenvectorLetter}{\varphi}
\newcommand{\zTransformed}[1]{\tilde{#1}}
\newcommand{\coefficientStable}{C_{\stableMarker}}
\newcommand{\coefficientUnstable}{C_{\unstableMarker}}
\newcommand{\relatives}{\mathbb{Z}}
\newcommand{\reals}{\mathbb{R}}
\newcommand{\complex}{\mathbb{C}}
\newcommand{\kernel}{\textnormal{ker}}
\newcommand{\transpose}[1]{#1^{\mathsf{T}}}

\newcommand{\unitCircle}{\mathbb{S}}
\newcommand{\unitDisk}{\mathbb{D}}
\newcommand{\closedUnitDisk}{\overline{\mathbb{D}}}
\newcommand{\neighborhoodInfinity}{\mathbb{U}}
\newcommand{\closedNeighborhoodInfinity}{\overline{\mathbb{U}}}
\newcommand{\productRoots}{\Pi}
\newcommand{\timeShiftOperatorModified}{\tilde{\timeShiftOperator}}
\newcommand{\argumentLegendrePolynomials}{\alpha}
\newcommand{\legendrePolynomial}[1]{P_{#1}}
\newcommand{\coefficientsExpansionSquareRoot}{B}
\newcommand{\coefficientsLaurentStableRoot}{s}
\newcommand{\differential}{\textnormal{d}}
\newcommand{\coefficientNumeratorRight}{\nu}
\newcommand{\coefficientNumeratorLeft}{\overline{\nu}}
\newcommand{\coefficientMultiplierNonConservedMomentRight}{\sigma}
\newcommand{\coefficientMultiplierNonConservedMomentLeft}{\overline{\sigma}}
\newcommand{\bigO}[1]{\mathcal{O}(#1)}
\newcommand{\confer}{{cf.}}
\newcommand{\coefficientsLaurentStableRootAllParities}{\beta}
\newcommand{\coefficientsLaurentUnstableRootReciprocalAllParities}{\upsilon}
\newcommand{\timeVariable}{t}
\newcommand{\advectionVelocity}{a}
\newcommand{\latticeVelocity}{\lambda}
\newcommand{\timeStep}{\Delta \timeVariable}
\newcommand{\lbmScheme}[2]{$\textnormal{D}_{#1}\textnormal{Q}_{#2}$}
\newcommand{\lbmSchemeVectorial}[3]{$\textnormal{D}_{#1}\textnormal{Q}_{#2}^{#3}$}
\newcommand{\timeGridPoint}[1]{\timeVariable^{#1}}
\newcommand{\naturals}{\mathbb{N}}
\newcommand{\labZeroVel}{\times}
\newcommand{\labPosX}{\Yright}
\newcommand{\labPosY}{\Yup}
\newcommand{\labNegX}{\Yleft}
\newcommand{\labNegY}{\Ydown}

\newcommand{\xLabel}{x}
\newcommand{\yLabel}{y}
\newcommand{\antiSymmetricMomentXLetter}{v_{\xLabel}}
\newcommand{\antiSymmetricMomentYLetter}{v_{\yLabel}}
\newcommand{\symmetricMomentXLetter}{w_{\xLabel}}
\newcommand{\symmetricMomentYLetter}{w_{\yLabel}}
\newcommand{\antiSymmetricMomentX}[1]{v_{\xLabel, #1}}
\newcommand{\antiSymmetricMomentY}[1]{v_{\yLabel, #1}}
\newcommand{\symmetricMomentX}[1]{w_{\xLabel, #1}}
\newcommand{\symmetricMomentY}[1]{w_{\yLabel, #1}}
\newcommand{\courantNumberX}{\courantNumber_{\xLabel}}
\newcommand{\courantNumberY}{\courantNumber_{\yLabel}}
\newcommand{\equilibriumCoefficientSymmetricX}{\mathscr{S}_{\xLabel}}
\newcommand{\equilibriumCoefficientSymmetricY}{\mathscr{S}_{\yLabel}}
\newcommand{\momentMatrix}{\matricial{M}}
\newcommand{\canonicalBasisVector}[1]{\vectorial{e}_{#1}}
\newcommand{\frequency}{\xi}
\newcommand{\zeroMatrix}[1]{\matricial{0}_{#1}}
\newcommand{\diagonalMatrix}{\text{\textbf{diag}}}
\newcommand{\numberTimeIterations}{N}

\newcommand{\termAtOrder}[2]{#1^{(#2)}}

\newcommand{\fluxLetter}{F}
\newcommand{\numberConservationLaws}{N}
\newcommand{\referenceState}{\overline{\vectorial{\conservedMoment}}}
\newcommand{\referenceStateMarked}[1]{\overline{#1}}
\newcommand{\height}{h}
\newcommand{\velocity}{u}
\newcommand{\gravity}{g}
\newcommand{\eigenvalueLinearized}{a}
\newcommand{\eigenvectorLinearized}{\vectorial{r}}
\newcommand{\matrixEigenvectorsLinearized}{\matricial{R}}
\newcommand{\linearGroup}[2]{\mathsf{GL}_{#1}(#2)}
\newcommand{\inTheWavesBase}[1]{\check{#1}}
\newcommand{\spectrum}{\textnormal{sp}}
\newcommand{\atEquilibrium}{\textnormal{eq}}
\newcommand{\singularity}{\eta}
\usepackage{dsfont}
\newcommand{\indicatorFunction}[1]{\mathds{1}_{#1}}
\newcommand{\topBoundary}{\yLabel_{t}}
\newcommand{\bottomBoundary}{\yLabel_{b}}
\newcommand{\soundSpeed}{c_{\textnormal{s}}}
\newcommand{\fourierTransformed}[1]{\hat{#1}}

\allowdisplaybreaks
\usepackage{multicol}

\title{Perfectly transparent boundary conditions and wave propagation in lattice Boltzmann schemes}
\author{Thomas Bellotti}
\date{\today}
\institute{Université Paris-Saclay, CNRS, CentraleSupélec, Laboratoire EM2C \& Fédération de Mathématiques de CentraleSupélec, 91190, Gif-sur-Yvette, France}

\providecommand{\keywords}[1]{{\noindent \small \textbf{{Keywords: }} {#1}}}

\begin{document}

\maketitle

\begin{abstract}
    Systems of $\numberConservationLaws=1, 2, \dots$ first-order hyperbolic conservation laws feature $\numberConservationLaws$ undamped waves propagating at finite speeds.  
    On their own hand, multi-step Finite Difference and lattice Boltzmann schemes with $q=\numberConservationLaws+1, \numberConservationLaws+2, \dots$ unknowns involve $\numberConservationLaws$ ``physical'' waves, which are aimed at being as closely-looking as possible to the ones of the PDEs, and $q-\numberConservationLaws$ ``numerical--spurious--parasitic'' waves, which are subject to their own speed of propagation, and either damped or undamped. 
    The whole picture is even more complicated in the discrete setting---as numerical schemes act as dispersive media, thus propagate different harmonics at different phase (and group) velocities.
    For compelling practical reasons, simulations must always be conducted on bounded domains, even when the target problem is unbounded in space. 
    The importance of transparent boundary conditions, preventing artificial boundaries from acting as mirrors producing polluting ricochets, naturally follows.

    This work presents, building on Besse, Coulombel, and Noble [ESAIM: M2AN, 55 (2021)], a systematic way of developing perfectly transparent boundary conditions for lattice Boltzmann schemes tackling linear problems in one and two space dimensions.
    Our boundary conditions are ``perfectly'' transparent, at least for 1D problems, as they absorb both physical and spurious waves regardless of their frequency.
    After presenting, in a simple framework, several approaches to handle the fact that $q>\numberConservationLaws$, we elect the so-called ``scalar'' approach (which despite its name, also works when $\numberConservationLaws>1$) as method of choice for more involved problems. 
    This method solely relies on computing the coefficients of the Laurent series at infinity of the roots of the dispersion relation of the bulk scheme. 
    We insist on asymptotics  for these coefficients in the spirit of analytic combinatorics.
    The reason is two-fold: asymptotics guide truncation of boundary conditions to make them depending on a fixed number of past time-steps, and make it clear---during the process of computing coefficients---whether intermediate quantities can be safely stored using floating-point arithmetic or not.
    Numerous numerical investigations in 1D and 2D with $\numberConservationLaws = 1$ and $2$ are carried out, and show the effectiveness of the proposed boundary conditions.
\end{abstract}




\section{Introduction}

Lattice Boltzmann methods (very often abridged ``LBM'') are numerical schemes for the approximation of solutions of Partial Differential Equations (PDEs) expressing some form of \strong{conservation principle}. 
These methods discretize time and space using uniform Cartesian meshes, and the algorithm evolves unknowns in a time-explicit fashion through two consecutive yet separated steps.
These schemes owe their name to this peculiar structure that mimics the Boltzmann equation.
Indeed, the first step is known as collision (or \strong{relaxation}) and is entirely local to each node of the mesh. It is the only part of the method that can feature non-linearities.
The second step, known as \strong{transport}, limits oneself to move unknowns, without mixing them, from one grid point to another, according to a velocity assigned to each unknown.

The growing success of lattice Boltzmann schemes, especially within applied communities of researchers, owes to their extreme rapidity and ease of implementation, both coming from their peculiar structure. 
Although most of the works use lattice Boltzmann schemes as weakly compressible approximations for the incompressible Navier-Stokes system, the present paper focuses on solving \strong{first-order systems of hyperbolic conservation laws}, an emerging trend using lattice Boltzmann methods \cite{dubois2014simulation, graille2014approximation, rao2015lattice, guillon2024stability, kozhanova2025hybrid, bellotti2025fourth, wissocq2024positive}.

The specific trait of any system of $\numberConservationLaws=1, 2, \dots$ hyperbolic conservation laws is that it features $\numberConservationLaws$ \strong{propagating waves} that are not attenuated in time, each one with its specific speed.
As they are shaped by the previously described structure, lattice Boltzmann methods to tackle these systems of PDEs must be based on $q\geq \numberConservationLaws + 1$ unknowns, and thus support $q>\numberConservationLaws$ (discrete) waves. 
The first $\numberConservationLaws$ waves by the numerical scheme are \strong{physically rooted} and hopefully good approximations, in some sense, of those of the PDEs, although affected by numerical dissipation and dispersion.
The remaining $q-\numberConservationLaws$ waves are purely \strong{artificial} (otherwise said, spurious). They experience propagation, possible time attenuation--damping, and of course dissipation and dispersion. 
They are a rather annoying by-product of the discretization process with which one is compelled to deal with, especially when they are not attenuated.
It is not much of a surprise that the very same plethora of waves at the numerical level exists with \strong{multi-step Finite Difference schemes} (e.g., the leap-frog scheme \cite{charney1950numerical})---considering the universal link between these and lattice Boltzmann schemes \cite{fuvcik2021equivalent, bellotti2022finite}.

Although the system of conservation laws may be posed on the whole space, the limited amount of computer memory dictates bounded computational meshes, and one would therefore like the discrete system behave \strong{as if boundaries were not present}.
This is the aim of \strong{transparent}---sometimes called ``non-reflecting'' or ``absorbing''---boundary conditions: allow waves (both physical and spurious) incident to the boundary to exit as they were carrying on their undisturbed propagation.
For Finite Difference schemes, the construction and analysis of such boundary conditions is a long line of research, and the interested reader can look at \cite{coulombel2019transparent} and references therein on this matter.
Another approach to achieve non-reflection is the so-called PML (\strong{Perfectly Matched Layer}) technique \cite{berenger1994perfectly, duru2012well, duru2014boundary}.
It boils down to surrounding the domain with an absorbing layer of finite thickness in which the governing equations are devised so that  waves decay rapidly.
This way of proceeding has been used with lattice Boltzmann schemes to minimize reflections of acoustic waves, see \cite{tekitek2009towards, xu2013analysis} for instance. 
In these works, absorbing layers are designed relying on formal expansions for small grid steps, and therefore only treat low-frequencies (i.e., smooth solutions) for physical waves, missing higher parts of spectra and spurious waves.
Another approach with lattice Boltzmann methods boils down to setting the amplitude of \strong{incoming waves} deduced from the target PDEs at a given boundary to zero.
However, although different ways of doing so have been proposed, they all rely on boundary conditions designed in the low-frequency limit \cite{izquierdo2008characteristic, wissocq2017regularized}, through Taylor expansions, or using equilibria \cite{heubes2014characteristic}. 
As a result, none of the previous approaches attains \strong{perfectly transparent} boundary conditions for lattice Boltzmann methods, which is the aim of the present contribution.
The existing works closer to our standpoint\footnote{That we discovered when our journey through these topics had already started.} are those of Heubes and collaborators \cite{heubes2014exact, heubes2015concept, heubes2015discrete}, with the first one dealing with a linear scalar \lbmScheme{1}{2} scheme\footnote{The notation \lbmScheme{d}{q} scheme designates a lattice Boltzmann scheme for a $d$-dimensional problem that features $q$ unknowns.}. Their approach yields perfectly transparent boundary conditions under the assumption that initial data are constant outside the domain, using the history of the values of the distribution functions at the boundary node, along with those of the equilibrium on a large set of nodes outside the domain.
As their construction is based on a good understanding of the way information propagates between nodes in a tree-like fashion, applicability is limited to 1D scalar schemes with a modest number of degrees of freedom. 
Generalizations to more involved settings \cite{heubes2015concept, heubes2015discrete} are presented only from the algorithmic standpoint, i.e. no closed-form formul\ae{} for boundary conditions are provided, relying on ``sub-problems'' solved at each time-step.
Finally, as these boundary conditions are, as the ones we presently develop, global-in-time, truncations to achieve higher efficiency are also proposed.

The aim of our work is to construct perfectly transparent boundary conditions for linear lattice Boltzmann schemes, both tackling scalar equations and systems, in 1D and 2D.
For systems of equations, our manuscript considers methods both based on a vectorial distribution function (the so-called ``vectorial'' schemes) and on a scalar distribution function with several conserved moments (that we call ``monolitic'').
To this end, we build on the work of \cite{besse2021discrete}, and propose a far-reaching and highly non-trivial generalization of \cite{heubes2014exact}.
After applying the $\timeShiftOperator$-transform to the numerical scheme to turn discrete time points into a complex parameter, this approach is based on the roots of the so-called ``\strong{characteristic equation}'' of the scheme.
In particular, the coefficients involved in the boundary conditions are obtained from those of \strong{Laurent expansions} of these roots at infinity, which can be computed either by direct computations, or inductively.
This latter approach shall turn out to be particularly flexible and easily deployable.
To make the scalar approach of Besse and collaborators work for lattice Boltzmann schemes, which feature at least two unknowns even for scalar PDEs, we investigate two solutions.
The first one, which we call ``systemic'', is based on the construction of the eigenspace associated with a root of the characteristic equation (i.e., an eigenvalue), see \cite{bellotti2025stability}.
The second one, called ``scalar'', lets the numerical scheme do the job at steering information on different components of the numerical solution, and unknowns are treated in an unentangled fashion at the boundary. This approach is particularly suitable for more involved numerical schemes, as it avoids the study of eigenvectors.
In the paper, we pay particular attention to the \strong{asymptotics} of the coefficients involved in boundary conditions, which we propose to compute through devices germane to \strong{analytic combinatorics} \cite{flajolet2009analytic}.
On the one hand, knowing the way coefficients tend to zero fosters truncation of boundary conditions, which are otherwise global-in-time for they depend on the entire past close to the boundary.
On the other hand, knowing their asymptotics enlightens whether intermediate quantities can be safely stored under floating-point arithmetic or not.
Before presenting the manuscript structure, let us point out that our strategy is intrinsically linear, and thus its extension to genuinely non-linear problems is a significant challenge left for future research.

Rather than pursuing a theoretical exposition, the paper is \strong{example-driven} and each part deals with an additional difficulty in the design of transparent boundary conditions.
Comprehensive numerical experiments are incrementally presented.
Overall, the scalar \lbmScheme{1}{2} scheme serves as initial benchmark to understand the main ideas, possible issues and best/easier practices. 
This example is therefore discussed in deep detail.
The paper starts with \Cref{sec:scalarProblems} on scalar problems, first solved in 1D with a \lbmScheme{1}{2} scheme and a fourth-order \lbmScheme{1}{3} scheme, then adding the difficulties of the 2D setting using a link two-relaxation-times \lbmScheme{2}{5} scheme.
\Cref{sec:systems} explains how the previous ideas are adapted to handle systems\footnote{Only in 1D, for the sake of presentation.}, both approximated with a monolithic \lbmScheme{1}{3} scheme on the example of shallow water equations, and by a vectorial \lbmSchemeVectorial{1}{2}{\numberConservationLaws} that works on any hyperbolic system.

\subsection{Notations and $\timeShiftOperator$-transform}\label{sec:notations}
For future use, we introduce the following notations on the segmentation of the complex plane:
\begin{align*}
	\unitCircle\definitionEquality\{\timeShiftOperator\in\complex\quad\text{such that}\quad |\timeShiftOperator|=1\}, \qquad
	\unitDisk &\definitionEquality \{\timeShiftOperator\in\complex\quad\text{such that}\quad |\timeShiftOperator|<1\}, \qquad \closedUnitDisk\definitionEquality \unitDisk\cup\unitCircle, \\
	\neighborhoodInfinity &\definitionEquality \{\timeShiftOperator\in\complex\quad\text{such that}\quad |\timeShiftOperator|>1\}, \qquad \closedNeighborhoodInfinity\definitionEquality \neighborhoodInfinity\cup\unitCircle.
\end{align*}

We also introduce the $\timeShiftOperator$-transform of a temporal signal, at least from a formal standpoint.
Let $(\conservedMoment^{\indexTime})_{\indexTime\in\naturals}\subset \reals$.
We define its $\timeShiftOperator$-transformed counterpart by 
\begin{equation*}
    \zTransformed{\conservedMoment}(\timeShiftOperator)\definitionEquality\sum_{\indexTime = 0}^{+\infty} \timeShiftOperator^{-\indexTime}\conservedMoment^{\indexTime},
\end{equation*}
which shows that $\conservedMoment^{\indexTime}$ is nothing but the $\indexTime$-th term in the \strong{Laurent series} of $\zTransformed{\conservedMoment}(\timeShiftOperator)$ at infinity.
Also, $\zTransformed{\conservedMoment}(\timeShiftOperator^{-1})$ can be interpreted as the (ordinary) \strong{generating function} of the sequence $(\conservedMoment^{\indexTime})_{\indexTime\in\naturals}$.
The inverse formula is given by 
\begin{equation*}
    \conservedMoment^{\indexTime} = \frac{1}{2\pi i}\oint_{\gamma} \timeShiftOperator^{\indexTime - 1}\zTransformed{\conservedMoment}(\timeShiftOperator)\differential\timeShiftOperator,
\end{equation*}
where $\gamma$ is a positively oriented closed curve enclosing the origin and completely within the region of convergence of $\zTransformed{\conservedMoment}(\timeShiftOperator)$.
Remark that the $\timeShiftOperator$-transform of the convolution of two sequences is simply the product of the $\timeShiftOperator$-transforms of each sequence.

\section{Scalar problems}\label{sec:scalarProblems}

\subsection{1D setting}

\subsubsection{Target partial differential equation}

We approximate the solution of 
\begin{equation}\label{eq:targetTransport1D}
    \begin{dcases}
        \begin{aligned}
            &\partial_{\timeVariable}\conservedMoment(\timeVariable, \spaceVariable) + \advectionVelocity\partial_{\spaceVariable}\conservedMoment(\timeVariable, \spaceVariable)  = 0, \qquad &\timeVariable>0, \quad &\spaceVariable\in\reals, \\
            &\conservedMoment(0, \spaceVariable)  = \conservedMoment^{\circ}(\spaceVariable), \qquad & &\spaceVariable\in\reals,
        \end{aligned}
    \end{dcases}
\end{equation}
with $\advectionVelocity\neq 0$, whose exact solution is given by $\conservedMoment(\timeVariable, \spaceVariable) =  \conservedMoment^{\circ}(\spaceVariable-\advectionVelocity\timeVariable)$.
As numerical simulations must necessarily be conducted on bounded domains, let us say $[\leftBoundary, \rightBoundary]$  with $\rightBoundary>\leftBoundary$, compactly supported initial data $\conservedMoment^{\circ}$ eventually yield zero solutions on $[\leftBoundary, \rightBoundary]$.
We would like this property to be preserved by numerical schemes.

\subsubsection{Space and time discretization}

The interval $[\leftBoundary, \rightBoundary]$ is paved with a mesh made up of $\numberGridPoints + 2$ points, with spacing
\begin{equation*}
    \spaceStep \definitionEquality \frac{\rightBoundary-\leftBoundary}{\numberGridPoints + 1}.
\end{equation*}
In this way, the discrete mesh is made up of $\spaceGridPoint{\indexSpace}\definitionEquality\leftBoundary + \indexSpace\spaceStep$, with $\indexSpace\in\integerInterval{0}{\numberGridPoints + 1}$.
On the points $\spaceGridPoint{0} = \leftBoundary$ and $\spaceGridPoint{\numberGridPoints + 1}= \rightBoundary$, boundary conditions must be enforced\footnote{We analyze schemes with space stencil of one both to the left and to the right.}.
Time is discretized by a mesh $\timeGridPoint{\indexTime}\definitionEquality\indexTime\timeStep$ with $\indexTime\in\naturals$ using a step $\timeStep\definitionEquality\spaceStep/\latticeVelocity$, where $\latticeVelocity>0$ is kept fixed.
For future use, we introduce the Courant number given by $\courantNumber\definitionEquality\advectionVelocity/\latticeVelocity$.

\subsubsection{Numerical schemes and their properties}

\paragraph{\lbmScheme{1}{2} scheme}\label{sec:D1Q2PresentationAndProp}

This scheme is presented, for instance, in \cite{junk2008regular, graille2014approximation}.
At the iteration $\indexTime\in\naturals$, the known input is  $(\conservedMoment_0^{\indexTime}, \nonConservedMoment_0^{\indexTime}), (\conservedMoment_1^{\indexTime}, \nonConservedMoment_1^{\indexTime}), \dots, (\conservedMoment_{\numberGridPoints}^{\indexTime}, \nonConservedMoment_{\numberGridPoints}^{\indexTime}), (\conservedMoment_{\numberGridPoints + 1}^{\indexTime}, \nonConservedMoment_{\numberGridPoints+1}^{\indexTime})$, with $\conservedMoment_{\indexSpace}^{\indexTime}\approx \conservedMoment(\timeGridPoint{\indexTime}, \spaceGridPoint{\indexSpace})$.
Then, the solution undergoes the following steps.
\begin{itemize}
    \item \strong{Relaxation} on the moments:
    \begin{equation}\label{eq:relaxationD1Q2}
        \conservedMoment_{\indexSpace}^{\indexTime, \collided} = \conservedMoment_{\indexSpace}^{\indexTime} \qquad \text{and}
        \qquad \nonConservedMoment_{\indexSpace}^{\indexTime, \collided} = (1-\omega)\nonConservedMoment_{\indexSpace}^{\indexTime} + \omega \underbrace{\courantNumber \conservedMoment_{\indexSpace}^{\indexTime}}_{\nonConservedMoment^{\atEquilibrium}(\conservedMoment_{\indexSpace}^{\indexTime})}, \qquad \indexSpace\in\integerInterval{0}{\numberGridPoints + 1}.
    \end{equation}
    The first equation in \eqref{eq:relaxationD1Q2} expresses local conservation for $\conservedMoment$, while the second one is a relaxation of $\nonConservedMoment$ towards its equilibrium with relaxation parameter $\relaxationParameter\in(0, 2]$.
    When $\relaxationParameter = 1$ is chosen, we obtain a relaxation scheme, see \cite{bouchut2003entropy}.
    \item \strong{Transport} on the distribution functions. After having computed the distribution functions $\distributionFunctionLetter_{\pm, \indexSpace}^{\indexTime, \collided}= \tfrac{1}{2}(\conservedMoment_{\indexSpace}^{\indexTime, \collided} \pm \nonConservedMoment_{\indexSpace}^{\indexTime, \collided} )$ for $\indexSpace\in\integerInterval{0}{\numberGridPoints + 1}$, we have 
    \begin{equation}\label{eq:transportD1Q2}
        \distributionFunctionLetter_{\pm, \indexSpace}^{\indexTime + 1}= \distributionFunctionLetter_{\pm, \indexSpace \mp 1}^{\indexTime, \collided}, \qquad \indexSpace\in\integerInterval{1}{\numberGridPoints}.
    \end{equation}
    Then, we obtain the moments $\conservedMoment_{\indexSpace}^{\indexTime+1} = \distributionFunctionLetter_{+, \indexSpace}^{\indexTime + 1} +  \distributionFunctionLetter_{-, \indexSpace}^{\indexTime + 1}$ and $\nonConservedMoment_{\indexSpace}^{\indexTime+1} = \distributionFunctionLetter_{+, \indexSpace}^{\indexTime + 1} -  \distributionFunctionLetter_{-, \indexSpace}^{\indexTime + 1}$ again, for $ \indexSpace\in\integerInterval{1}{\numberGridPoints}$.
    \item \strong{Boundary conditions}. Compute $(\conservedMoment_0^{\indexTime + 1}, \nonConservedMoment_0^{\indexTime + 1})$ and $(\conservedMoment_{\numberGridPoints + 1}^{\indexTime + 1}, \nonConservedMoment_{\numberGridPoints+1}^{\indexTime + 1})$ as described in \Cref{sec:transparentD1Q2}.
\end{itemize}

\newcommand{\squareTikzTwoColors}[2]{%
	\fill[color=RoyalBlue] (#1-0.2,#2-0.2) -- (#1-0.2,#2+0.2) -- (#1,#2+0.2) -- (#1,#2-0.2) -- cycle;
	\fill[color=ForestGreen] (#1,#2-0.2) -- (#1,#2+0.2) -- (#1+0.2,#2+0.2) -- (#1+0.2,#2-0.2) -- cycle;
  	\draw[black] (#1-0.2,#2-0.2) rectangle (#1+0.2,#2+0.2);
}

\newcommand{\squareTikzFourColors}[2]{%
	\fill[color=RoyalBlue] (#1-0.2,#2-0.2) -- (#1-0.2,#2+0.2) -- (#1,#2+0.2) -- (#1,#2-0.2) -- cycle;
	\fill[color=ForestGreen] (#1,#2-0.2) -- (#1,#2+0.2) -- (#1+0.2,#2+0.2) -- (#1+0.2,#2-0.2) -- cycle;

    \fill[color=YellowOrange] (#1-0.2,#2+0.2) -- (#1-0.2,#2+0.6) -- (#1,#2+0.6) -- (#1,#2+0.2) -- cycle;
	\fill[color=Magenta] (#1,#2+0.2) -- (#1,#2+0.6) -- (#1+0.2,#2+0.6) -- (#1+0.2,#2+0.2) -- cycle;

  	\draw[black] (#1-0.2,#2-0.2) rectangle (#1+0.2,#2+0.2);
  	\draw[black] (#1-0.2,#2+0.2) rectangle (#1+0.2,#2+0.6);
}

\newcommand{\squareTikzFourColorsFlipped}[2]{%
	\fill[color=YellowOrange] (#1-0.2,#2-0.2) -- (#1-0.2,#2+0.2) -- (#1,#2+0.2) -- (#1,#2-0.2) -- cycle;
	\fill[color=Magenta] (#1,#2-0.2) -- (#1,#2+0.2) -- (#1+0.2,#2+0.2) -- (#1+0.2,#2-0.2) -- cycle;
    
    \fill[color=RoyalBlue] (#1-0.2,#2+0.2) -- (#1-0.2,#2+0.6) -- (#1,#2+0.6) -- (#1,#2+0.2) -- cycle;
	\fill[color=ForestGreen] (#1,#2+0.2) -- (#1,#2+0.6) -- (#1+0.2,#2+0.6) -- (#1+0.2,#2+0.2) -- cycle;

  	\draw[black] (#1-0.2,#2-0.2) rectangle (#1+0.2,#2+0.2);
  	\draw[black] (#1-0.2,#2+0.2) rectangle (#1+0.2,#2+0.6);
}

\usetikzlibrary{decorations.pathreplacing} 

\begin{figure}[h]
	\begin{center}
		\begin{tikzpicture}

			\draw[->] (-4,0) -- (6,0) node[right] {\(x\)};
			\draw[->] (-3.8,-0.2) -- (-3.8,6) node[above] {\(t\)};

			\foreach \x in {-2,...,0}
				\fill[color=black](\x*1.4,0) circle (2pt);
            \foreach \x in {2, 3}
                \fill[color=black](\x*1.4,0) circle (2pt);

			\node at (-2*1.4, -0.5) {$\spaceGridPoint{0} = \leftBoundary$};
			\node at (-1*1.4, -0.5) {$\spaceGridPoint{1}$};
			\node at (0*1.4, -0.5) {$\spaceGridPoint{2}$};
			\node at (1*1.4, -0.5) {$\cdots$};
            \node at (1*1.4, 1) {$\cdots$};
            \node at (1*1.4, 3) {$\cdots$};
            \node at (1*1.4, 5) {$\cdots$};
			\node at (2*1.4, -0.5) {$\spaceGridPoint{\numberGridPoints }$};
			\node at (3*1.4, -0.5) {$\spaceGridPoint{\numberGridPoints + 1} = \rightBoundary$};

            \draw (-3.9, 1) node[left] {$\timeGridPoint{\indexTime}$} -- (-3.7, 1);
			\draw (-3.9, 3) node[left] {($\timeGridPoint{\indexTime\collided}$)} -- (-3.7, 3);
			\draw (-3.9, 5) node[left] {$\timeGridPoint{\indexTime + 1}$} -- (-3.7, 5);

            \squareTikzTwoColors{-2.8}{1}
            \squareTikzTwoColors{-1.4}{1}			
			\squareTikzTwoColors{0.}{1}
			\squareTikzTwoColors{2.8}{1}
			\squareTikzTwoColors{4.2}{1}

            \foreach \x in {-2,...,0, 2, 3}
                \draw[thick, ->, color=RoyalBlue] (\x*1.4-0.1, 1.3) to[out=120,in=240] (\x*1.4-0.1, 2.5);
            \foreach \x in {-2,...,0, 2, 3}
                \draw[thick, smooth, ->, color=RoyalBlue] (\x*1.4-0.07, 1.3) -- (\x*1.4+0.07, 2.4);
            \foreach \x in {-2,...,0, 2, 3}
                \draw[thick, ->, color=ForestGreen] (\x*1.4+0.1, 1.3) to[out=60,in=300] (\x*1.4+0.1, 2.5);

            \squareTikzFourColors{-2.8}{2.8}
            \squareTikzFourColors{-1.4}{2.8}			
			\squareTikzFourColors{0.}{2.8}
			\squareTikzFourColors{2.8}{2.8}
			\squareTikzFourColors{4.2}{2.8}

			\foreach \x in {-2,...,0}
				\draw[thick, ->, color=YellowOrange] (\x*1.4-0.1, 3.5) -- (\x*1.4 + 1.3, 4.5);

			\foreach \x in {-1,...,-1, 1, 2}
				\draw[thick, ->, color=Magenta] (\x*1.4+1.5, 3.5) -- (\x*1.4 + 0.1, 4.5);

            \squareTikzTwoColors{-2.8}{5.2}
			\squareTikzFourColorsFlipped{-1.4}{4.8}
			\squareTikzFourColorsFlipped{0.}{4.8}
			\squareTikzFourColorsFlipped{2.8}{4.8}
            \squareTikzTwoColors{4.2}{5.2}

            \draw[decorate, decoration={brace, mirror}, thick] (4*1.4-0.5,1.1) -- (4*1.4-0.5,2.9) node[midway, right=2pt] {Relaxation \eqref{eq:relaxationD1Q2}};
			\draw[decorate, decoration={brace, mirror}, thick] (4*1.4-0.5,3.1) -- (4*1.4-2,4.9) node[midway, right=2pt] {Transport \eqref{eq:transportD1Q2}};

            \draw (-2.8, 5.2) circle[radius=0.5cm, thick];
            \draw (4.2, 5.2) circle[radius=0.5cm, thick];
            \path (-2.8, 5.2) ++(30:0.65cm) coordinate (A);
            \path (4.2, 5.2) ++(150:0.65cm) coordinate (B);

  \draw[-, thick]
    (A) .. controls ++(60:1cm) and ++(120:1cm) .. (B)
    node[midway, above] {Boundary conditions};
		\end{tikzpicture}
	\end{center}\caption{\label{fig:descriptionScheme}Schematic description of the \lbmScheme{1}{2} algorithm. Cold colors indicate moments, with blue for the conserved moment. 
    Warm colors indicate distribution functions.}
\end{figure}

A visual illustration of the algorithm is given in \Cref{fig:descriptionScheme}.
Notice that $\conservedMoment$ and $\nonConservedMoment$ must be initialized to bridge with the datum $\conservedMoment^{\circ}$.
In all the numerical simulations and schemes, we consider 
\begin{equation}\label{eq:initializationEquilibrium}
    \conservedMoment_{\indexSpace}^{0} = \conservedMoment^{\circ}(\spaceGridPoint{\indexSpace})\qquad\text{and} \qquad \nonConservedMoment_{\indexSpace}^{0} = \nonConservedMoment^{\atEquilibrium}(\conservedMoment^{\circ}(\spaceGridPoint{\indexSpace})), \qquad \indexSpace\in\integerInterval{0}{\numberGridPoints+1},
\end{equation}
or analogous expressions for other schemes, assuming the initial datum $\conservedMoment^{\circ}$  be smooth enough to be point-wise sampled.

Introducing the space-shift operator $\fourierShift$, such that $(\fourierShift\distributionFunctionLetter)_{\indexSpace} = \distributionFunctionLetter_{\indexSpace+1}$, we can recast the numerical scheme \eqref{eq:relaxationD1Q2}--\eqref{eq:transportD1Q2} on an unbounded domain under the form $\transpose{(\conservedMoment_{\indexSpace}^{\indexTime + 1}, \nonConservedMoment_{\indexSpace}^{\indexTime + 1})} = \schemeMatrixBulkFourier(\fourierShift) \transpose{(\conservedMoment_{\indexSpace}^{\indexTime}, \nonConservedMoment_{\indexSpace}^{\indexTime})}$  with $\indexSpace\in\relatives$.
We have that 
\begin{equation*}
    \schemeMatrixBulkFourier(\fourierShift) = 
    \frac{1}{2}
    \begin{pmatrix}
       (1+\courantNumber\relaxationParameter)\fourierShift^{-1} + (1-\courantNumber\relaxationParameter)\fourierShift & (1-\relaxationParameter)\fourierShift^{-1} -  (1-\relaxationParameter)\fourierShift\\
       (1+\courantNumber\relaxationParameter)\fourierShift^{-1} - (1-\courantNumber\relaxationParameter)\fourierShift & (1-\relaxationParameter)\fourierShift^{-1} +  (1-\relaxationParameter)\fourierShift
    \end{pmatrix}.
\end{equation*}
The matrix $\schemeMatrixBulkFourier(\fourierShift)$ can be conveniently interpreted as made up of Laurent polynomials in $\fourierShift$, or also considering $\fourierShift$ as a complex number.
We also have  the associated \strong{characteristic equation}
\begin{equation}\label{eq:charEquation}
    \determinant(\timeShiftOperator\identityMatrix{2} - \schemeMatrixBulkFourier(\fourierShift)) = \tfrac{1}{2}((1-\courantNumber)\relaxationParameter-2)\timeShiftOperator \fourierShift^{-1} + (\timeShiftOperator^2 + 1-\relaxationParameter) +\tfrac{1}{2}((1+\courantNumber)\relaxationParameter-2)\timeShiftOperator \fourierShift = 0,
\end{equation}
which can be seen, besides its formal interpretation,
\begin{itemize}
    \item as an equation on $\timeShiftOperator$ parametrized by $\fourierShift\in\unitCircle$ (which yields a discrete Fourier transform in space), for the analysis on unbounded domains and to study the waves supported by the \strong{bulk scheme}. 
    This point is important to understand numerical results with imperfectly transparent boundary conditions.
    \item Or as an equation on $\fourierShift$ parametrized by $\timeShiftOperator\in\closedNeighborhoodInfinity$, which we use to devise boundary conditions.
    This standpoint is tightly linked to the GKS (Gustafsson--Kreiss--Sundstr{\"o}m) stability analysis for boundary conditions \cite{gustafsson1972stability, coulombel2009stability}.
\end{itemize}
Observe that \eqref{eq:charEquation} is the \strong{dispersion relation} of the scheme, which is \strong{non-linear} despite the linear character of the method.
It essentially rules what modes $(\timeShiftOperator, \fourierShift)$ are admitted by the numerical scheme.

\begin{remark}[How \eqref{eq:charEquation} could be already known]\label{rem:D1Q2AlreadyKnown}
    A careful inspection of \eqref{eq:charEquation} reveals that it is the dispersion relation of the leap-frog scheme for the equation $\partial_{\timeVariable\timeVariable}\conservedMoment+\latticeVelocity^2\partial_{\spaceVariable\spaceVariable}\conservedMoment = 0$ when $\relaxationParameter = 0$, that of the Lax-Friedrichs scheme for \eqref{eq:targetTransport1D} when $\relaxationParameter = 1$, and the one for the leap-frog scheme (\confer{}, \cite{besse2021discrete}) tackling \eqref{eq:targetTransport1D} when $\relaxationParameter = 2$.
    Outside these particular values of $\omega$,  \eqref{eq:charEquation} can be seen as a convex combination between two of the three previously mentioned dispersions relations.
\end{remark}

The study of the roots $\timeShiftOperator$ of  \eqref{eq:charEquation} as functions of $\fourierShift$ allows a characterization of both consistency and stability\footnote{At least concerning necessary conditions based solely on the modulus of eigenvalues, i.e. von Neumann criterion.}  as follows (proof in \cite{bellotti2025stability} and references therein). 
We also indulge on a long description of other properties of the \lbmScheme{1}{2}, such as dissipation, damping, and propagation, which---although elementary---are scattered or nowhere to be found in previous works.
\begin{proposition}[Consistency and stability of the \lbmScheme{1}{2} scheme on $\relatives$]\label{prop:consistencyD1Q2}
    The \lbmScheme{1}{2} scheme is such that there exists a unique eigenvalue $\timeShiftOperator_{\textnormal{phy}}(\frequency\spaceStep)$ of $\schemeMatrixBulkFourier(e^{i\frequency\spaceStep})$ with $\frequency\spaceStep\in[-\pi, \pi]$ (equivalently, root of \eqref{eq:charEquation} with $\fourierShift=e^{i\frequency\spaceStep}$) such that 
    \begin{equation}\label{eq:expansionConsistencyRoot}
        \timeShiftOperator_{\textnormal{phy}}(\frequency\spaceStep) = 1 - i \advectionVelocity\xi\timeStep + \Bigl ( -\frac{\courantNumber^2}{2} + (1-\courantNumber^2)\Bigl (\frac{1}{2}-\frac{1}{\relaxationParameter}\Bigr )\Bigr ) \latticeVelocity^2(\xi\timeStep)^2 + \bigO{(\xi\timeStep)^3}.
    \end{equation}
    This entails that the scheme is first-order accurate when $\relaxationParameter\in(0, 2)$ and second-order accurate when $\relaxationParameter = 2$.

    The scheme is $\ell^2$--stable, namely there exists $C>0$ such that $\sup_{\frequency\spaceStep\in[-\pi, \pi]} |\schemeMatrixBulkFourier(e^{i\frequency\spaceStep})^{\indexTime}|\leq C$ for all $\indexTime\in\naturals$, if and only if 
    \begin{equation}\label{eq:stabConditionD1Q2}
        \text{when }\relaxationParameter\in (0, 2), \quad \text{then}\quad |\courantNumber|\leq 1 \qquad\qquad\text{or}\qquad\qquad
        \text{when }\relaxationParameter = 2,\quad \text{then}\quad |\courantNumber|< 1.
    \end{equation}
\end{proposition}

Let $\fourierShift = e^{i\frequency\spaceStep}$ with $\frequency\spaceStep\in[-\pi, \pi]$ in \eqref{eq:charEquation} and denote $\timeShiftOperator_{\textnormal{phy}}(\frequency\spaceStep)$ and $\timeShiftOperator_{\textnormal{spu}}(\frequency\spaceStep)$ the roots of \eqref{eq:charEquation}, which have been continuously and globally (for the moment, we just assume this fact) parametrized. These roots are symbols (and except for $\relaxationParameter = 1$, not trigonometric polynomials) or---analogously to pseudo-differential operators---pseudo-Finite Difference operators. 
The symbol $\timeShiftOperator_{\textnormal{phy}}(\frequency\spaceStep)$ is the PDE-equivalent of what \cite[Chapter XV]{3-540-30663-3} call ``underlying one-step method'' in the context of multi-step methods for ODEs.
We call $\timeShiftOperator_{\textnormal{phy}}(\frequency\spaceStep)$ ``physical symbol'' as we assume that fulfills $\timeShiftOperator_{\textnormal{phy}}(0) = 1$.
Conversely, $\timeShiftOperator_{\textnormal{spu}}(\frequency\spaceStep)$ is called ``spurious symbol''.
If these symbols are distinct, fact that we study in a moment, the numerical solution in the Fourier space, see \cite[Chapter 4]{strikwerda2004finite}, \cite[Chapter 1]{gustafsson2013time}, \cite[Chapter XV]{3-540-30663-3}, reads 
\begin{equation*}
    \fourierTransformed{\conservedMoment}^{\indexTime}(\frequency) = \timeShiftOperator_{\textnormal{phy}}(\frequency\spaceStep)^{\indexTime} A_{\textnormal{phy}}(\frequency) + \timeShiftOperator_{\textnormal{spu}}(\frequency\spaceStep)^{\indexTime} A_{\textnormal{spu}}(\frequency),
\end{equation*}
where $A_{\textnormal{phy}}(\frequency)$ and $A_{\textnormal{spu}}(\frequency)$ are determined by the initialization. 
This means that the numerical solution is the superposition of iterative applications of two pseudo-schemes, one physical that approximates the solution of the PDE as good as its feature allow, and one spurious which is merely a byproduct of the numerical approximation. 

As it makes an important bifurcation in the structure of the spectrum, we study the fact that the two symbols coincide for some frequencies.
To do so, we study the system
\begin{empheq}[left=\empheqlbrace]{align}
    \determinant(\tilde{\timeShiftOperator}\identityMatrix{2}-\schemeMatrixBulkFourier(e^{i\frequency\spaceStep}))|_{\tilde{\timeShiftOperator} = \timeShiftOperator(\frequency\spaceStep)} &= 0 \label{eq:systemDoubleRootFirst} \\
    \tfrac{\differential{}}{\differential{\tilde{\timeShiftOperator}}}\determinant(\tilde{\timeShiftOperator}\identityMatrix{2}-\schemeMatrixBulkFourier(e^{i\frequency\spaceStep}))|_{\tilde{\timeShiftOperator} = \timeShiftOperator(\frequency\spaceStep)} &= 0. \label{eq:systemDoubleRootSecond}
  \end{empheq}
Through straightforward computations presented in \Cref{app:D1Q2coincidingsymbols}, we obtain what follows. 
\begin{itemize}
    \item For $\relaxationParameter<1$, the roots $\timeShiftOperator_{\textnormal{phy}}(\frequency\spaceStep)$ and $\timeShiftOperator_{\textnormal{spu}}(\frequency\spaceStep)$  are distinct for $\frequency\spaceStep\in[-\pi, \pi]$, and thus smoothly parametrized (they are continuous for free as the eigenvalues are always continuous functions), see \cite[Section 5.2]{serre}, although it might be difficult to construct this global parametrization explicitly.

    \item For $\relaxationParameter\geq 1$, the roots $\timeShiftOperator_{\textnormal{phy}}(\frequency\spaceStep)$ and $\timeShiftOperator_{\textnormal{spu}}(\frequency\spaceStep)$, otherwise distinct and smoothly parametrized, coincide for the relaxation parameter
    \begin{equation}\label{eq:relaxationParameterDoubleTimeEig}
        \relaxationParameter = \frac{2}{\courantNumber^2} (1-\sqrt{1-\courantNumber^2})
    \end{equation}
    at $\frequency\spaceStep = \pm \pi/2$, the glancing frequencies of the leap-frog scheme, with the coinciding eigenvalues being $\timeShiftOperator = \pm i ({\tfrac{2}{\courantNumber^2} (1-\sqrt{1-\courantNumber^2})-1})^{1/2}$.
    The right-hand side of \eqref{eq:relaxationParameterDoubleTimeEig} describes an even convex function of $\courantNumber$, with image in $[1, 2]$, and such that $\lim_{\courantNumber\to 0}\frac{2}{\courantNumber^2} (1-\sqrt{1-\courantNumber^2}) = 1$ and $\lim_{|\courantNumber|\to 1}\frac{2}{\courantNumber^2} (1-\sqrt{1-\courantNumber^2}) = 2$.
    The latter case describes the celebrated reason why the leap-frog scheme cannot be employed with $|\courantNumber| = 1$, see \cite[Chapter 4]{strikwerda2004finite}.
\end{itemize}
One can still conjecture, looking at \Cref{fig:D1Q2Spectrum}, that even when \eqref{eq:relaxationParameterDoubleTimeEig} holds (in the setting of \Cref{fig:D1Q2Spectrum}, this happens for $\relaxationParameter = \tfrac{72}{25}(1-\sqrt{{11}/{36}})\approx 1.28802$), one can still find global smooth parametrizations for $\timeShiftOperator_{\textnormal{phy}}(\frequency\spaceStep)$ and $\timeShiftOperator_{\textnormal{spu}}(\frequency\spaceStep)$.
They are however very different from the ones that one naturally considers when  \eqref{eq:relaxationParameterDoubleTimeEig} does not hold.
Indeed, it is not possible to find parametrizations being smooth both  with respect to $\frequency\spaceStep$ and $\relaxationParameter$, i.e. spectra are not continuous with respect to more than one parameter.

\begin{figure}
    \begin{center}
        \includegraphics[width=1\textwidth]{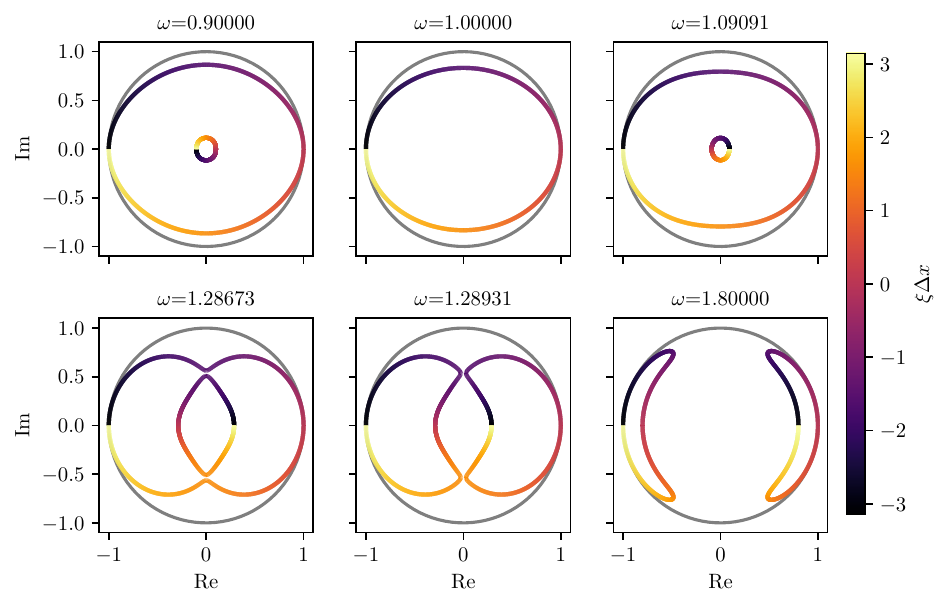}
    \end{center}\caption{\label{fig:D1Q2Spectrum}Symbols $\timeShiftOperator_{\textnormal{phy}}(\frequency\spaceStep)$ and $\timeShiftOperator_{\textnormal{spu}}(\frequency\spaceStep)$ parametrized globally and continuously for the \lbmScheme{1}{2} with $\courantNumber = \tfrac{5}{6}$.}
\end{figure}

We also use the dispersion relation to study \strong{damping} harmonics-by-harmonics (i.e., behaviors of kind $\propto \timeShiftOperator^{\indexTime}$ for $|\timeShiftOperator|<1$) and the \strong{propagation} of spurious modes. We shall observe these features on actual numerical simulations.
\begin{remark}[Damping \strong{versus} dissipation]
    Damping and dissipation (defined by \cite[Definition 5.1.1]{strikwerda2004finite}--\cite[Definition 4.2.4]{gustafsson2013time}) are related but distinct concepts.
    Damping describes the tendency of packets to be (geometrically) \strong{crushed} in time, without any description, contrarily to dissipation, of packets' \strong{spread} in space as time advances.
    This comes from the point-wise nature of damping in the frequency space. 
    A symbol that does not damp all non-zero frequencies cannot be dissipative (e.g., the leap-frog scheme).

    Also notice that dissipation is not generally easy to describe for lattice Boltzmann schemes, as symbols are not trigonometric polynomials.
    For the \lbmScheme{1}{2} when $\relaxationParameter\in(0, 2)$, it is not difficult to see that dissipation is very similar to the one of the Lax-Friedrichs scheme, which is ``nondissipative but not strictly so'' \cite{strikwerda2004finite} due to $\frequency\spaceStep = \pm \pi$.
    What can be done in general is a localized-in-frequency description of dissipation by Taylor expansions, see remark below.
\end{remark}

\begin{remark}[Dissipation at low frequencies]\label{rem:dissipationLowFrequencies}
    Taylor expansions yield 
    \begin{align*}
        |\timeShiftOperator_{\textnormal{phy}}(\frequency\spaceStep)| = &1-(1-\courantNumber^2)\Bigl (\frac{1}{\relaxationParameter} - \frac{1}{2}\Bigr )(\frequency\spaceStep)^2 + \bigO{(\frequency\spaceStep)^4}, \\
        |\timeShiftOperator_{\textnormal{spu}}(\frequency\spaceStep)| = |\relaxationParameter-1|\Bigl ( &1-(1-\courantNumber^2)\Bigl (\frac{1}{2} - \frac{1}{\relaxationParameter}\Bigr )(\frequency\spaceStep)^2 + \bigO{(\frequency\spaceStep)^4}\Bigr ),
    \end{align*}
    which is coherent with \Cref{fig:D1Q2Spectrum}, and shows that for $\relaxationParameter\in(0, 2)$ and smooth solutions, the physical symbol is \strong{dissipative of order 2} (this is often called ``numerical diffusion''), whereas the renormalized spurious symbol is \strong{anti-dissipative}.
    This means that harmonics around $\frequency\spaceStep = 0$ are damped slower than constant modes.
    This is visible on \Cref{fig:D1Q2Spectrum}: the modulus increases around $\frequency\spaceStep = 0$.
    This is not a problem as the spurious symbol experiences \strong{damping}.
\end{remark}

\begin{lemma}[Damping in the \lbmScheme{1}{2} scheme]\label{lemma:dampingD1Q2}
    Consider the strict CFL condition $|\courantNumber|<1$.    
    Let $\fourierShift = e^{i\frequency\spaceStep}$ with $\frequency\spaceStep\in[-\pi, \pi]$ in \eqref{eq:charEquation}.
    Then the roots $\timeShiftOperator = \timeShiftOperator(\frequency\spaceStep)$ of \eqref{eq:charEquation} are as follows.
    \begin{itemize}
        \item For $\frequency\spaceStep\in (-\pi, \pi)\smallsetminus\{0\}$, they belong to $\unitDisk$ when $\relaxationParameter\in (0, 2)$ and to $\unitCircle$ when $\relaxationParameter = 2$.
        \item For  $\frequency\spaceStep=0$, they are $\{1, 1-\relaxationParameter\}$, so the second one belongs to $\unitDisk$ when $\relaxationParameter\in (0, 2)$ and to $\unitCircle$ when $\relaxationParameter = 2$.
        \item For $\frequency\spaceStep=\pm\pi$, they are $\{-1, \relaxationParameter-1\}$, so the second  one belongs to $\unitDisk$ when $\relaxationParameter\in (0, 2)$ and to $\unitCircle$ when $\relaxationParameter = 2$.
    \end{itemize}
\end{lemma}
\begin{proof}
    The proof follows the procedures described in \cite[Chapter 4]{strikwerda2004finite} to locate roots of complex-coefficients polynomials with respect to $\unitDisk$ and $\unitCircle$.
\end{proof}
\Cref{lemma:dampingD1Q2} tells the following concerning damping.
\begin{itemize}
    \item For $\relaxationParameter\in (0, \frac{2}{\courantNumber^2} (1-\sqrt{1-\courantNumber^2}))$, the physical mode is damped except for the harmonics $\frequency\spaceStep=0, \pm\pi$, which is a peculiar feature of the Lax-Friedrichs scheme ($\relaxationParameter = 1$), \confer{} \cite{breuss2004correct, li2009local}. The spurious mode is damped for every frequency.
    \item For $\relaxationParameter\in[\frac{2}{\courantNumber^2} (1-\sqrt{1-\courantNumber^2}), 2)$. The physical mode is damped except at $\frequency\spaceStep=0$.
    The spurious mode is damped except at $\frequency\spaceStep = \pm \pi$, in such a way that the Lax-Friedrichs' ``signature'' is transferred to the spurious symbol.
    \item For $\relaxationParameter = 2$. No harmonics of either the physical or the spurious mode is attenuated.
\end{itemize}
If the spurious symbol is damped, see Remark 2.3 in \cite[Chapter XV]{3-540-30663-3}, the long-time behavior is \strong{essentially} the one of a one-step method with symbol $\timeShiftOperator_{\textnormal{phy}}(\frequency\spaceStep)$.

\begin{remark}[A taste of $\ell^{\infty}$ stability]
    In light of \Cref{lemma:dampingD1Q2}, there is hope of adapting---in a future work---the results on $\ell^{\infty}$-stability by Thomée to the \lbmScheme{1}{2} when $\relaxationParameter\in (0, 2)$.
    In this case, under the condition $|\courantNumber|<1$, the only points where the symbols touch $\unitCircle$ correspond to $\frequency\spaceStep = 0, \pm \pi$ and Taylor expansions show that 
    \begin{align*}
        \timeShiftOperator_{\textnormal{phy}}(\frequency\spaceStep) &= \text{exp}\Bigl ( i(-\courantNumber)\frequency\spaceStep - \overbrace{(1-\courantNumber^2)\Bigl (\frac{1}{\relaxationParameter} - \frac{1}{2}\Bigr )}^{\textnormal{positive}}(\frequency\spaceStep)^2(1+o(1))\Bigr ), \\
        \timeShiftOperator_{\textnormal{phy/spu}}(\frequency\spaceStep) &= -\text{exp}\Bigl ( i(-\courantNumber)(\frequency\spaceStep-\pi) - (1-\courantNumber^2)\Bigl (\frac{1}{\relaxationParameter} - \frac{1}{2}\Bigr )(\frequency\spaceStep - \pi)^2(1+o(1))\Bigr ),
    \end{align*}
    so there is hope to adapt \cite[Theorem 1]{thomee1965stability}.
\end{remark}

We now analyze \strong{propagation} of the spatially-smooth (not necessarily time-smooth) part of the numerical solution.
This is achieved by eventual Taylor expansions close to $\frequency\spaceStep = 0$.
While the previous analysis shows that the spurious mode is damped approximately at ratio $1-\relaxationParameter$, and time-checkerboard for $\relaxationParameter>1$, the question is whether it propagates before disappearing.
This can be understood by analyzing the \strong{group velocity} \cite{trefethen1982group, trefethen1984instability} in the low-frequency limit, which we extend in the case of  absorptive--lossy ``medium''.
Let $\timeShiftOperator = e^{i\eta\timeStep}$ and $\fourierShift = e^{i\frequency\spaceStep}$, with $\frequency\spaceStep\in[-\pi, \pi]$. When $\relaxationParameter\in(0, 2)$, due to the damping exerted on the spurious symbol, we cannot request $\eta\timeStep\in\reals$, but we allow complex pulsations $\eta\in\complex$.
Into \eqref{eq:charEquation} and differentiating in $\eta$ and $\frequency$, we have 
\begin{equation*}
    \spaceStep \bigl ( (2-\relaxationParameter)\sin(\frequency\spaceStep)+ i \relaxationParameter\courantNumber\cos(\frequency\spaceStep)\bigr ) \differential{\frequency}
    + \timeStep\bigl ( (\relaxationParameter-2)\sin(\eta\timeStep) + i\relaxationParameter \cos(\eta\timeStep) \bigr ) \differential{\eta}= 0.
\end{equation*}
According to \cite[Equation (3.17)]{trefethen1984instability}, we define the group velocity by 
\begin{equation*}
    \textnormal{V}_{\textnormal{g}}\definitionEquality -\frac{\differential{\eta}}{\differential{\frequency}} = \latticeVelocity \frac{ (2-\relaxationParameter)\sin(\frequency\spaceStep)+ i \relaxationParameter\courantNumber\cos(\frequency\spaceStep)}{(\relaxationParameter-2)\sin(\eta\timeStep) + i\relaxationParameter \cos(\eta\timeStep) },
\end{equation*}
where it is understood that $\timeShiftOperator = e^{i\eta\timeStep}$ and $\fourierShift = e^{i\frequency\spaceStep}$ fulfill \eqref{eq:charEquation}.
To evaluate the group velocity for space-smooth data, we consider $\frequency\spaceStep = 0$:
\begin{align}
    \textnormal{(Physical mode)}\qquad &\textnormal{V}_{\textnormal{g}}(\eta\timeStep = 0, \frequency\spaceStep = 0) = \latticeVelocity\courantNumber = \advectionVelocity, \label{eq:gvPhysicalLow}\\
    \textnormal{(Spurious mode)}\qquad &\textnormal{V}_{\textnormal{g}}(\eta\timeStep = -i\log(1-\relaxationParameter), \frequency\spaceStep = 0) = -\latticeVelocity\courantNumber = -\advectionVelocity.\label{eq:gvSpuriousLow}
\end{align}
This result can also be obtained by expansions analogous to \eqref{eq:expansionConsistencyRoot} for the spurious root\footnote{Notice that since we consider low spatial frequencies, group and phase velocities coincide.}, which read
\begin{equation*}
    \timeShiftOperator_{\textnormal{spu}}(\frequency\spaceStep) = \underbrace{(1-\relaxationParameter)}_{\textnormal{damping}} (\underbrace{1 + i\advectionVelocity\frequency\timeStep}_{\textnormal{propagation}} + \bigO{(\frequency\timeStep)^2}).
\end{equation*}
The fact that the smooth part of the spurious mode propagates at speed with opposite sign with respect to the physical mode can be understood by looking at the color scale of \Cref{fig:D1Q2Spectrum} and the way it evolves along the curves close to $\frequency\spaceStep = 0$.

To sum up, we can say that the scheme inherits dissipation (essentially) from the Lax-Friedrichs scheme and propagation for the spurious mode from the leap-frog scheme, \confer{} \Cref{rem:D1Q2AlreadyKnown}.
These phenomena are interweaved by damping.

\begin{summarybox}{Summary of the main properties of the \lbmScheme{1}{2} scheme}
    \begin{small}
        \begin{center}
            \begin{tabular}{|ll!{\vrule width 1pt}c|c!{\vrule width 1pt}}
                \cline{3-4}
                \multicolumn{2}{l|}{} & $\relaxationParameter\in(0, 2)$ & \multicolumn{1}{c|}{$\relaxationParameter = 2$}\\
                \cline{1-2}\noalign{\global\arrayrulewidth=1pt}
                \cline{3-4}
                \noalign{\global\arrayrulewidth=.4pt}
                \multicolumn{2}{|l!{\vrule width 1pt}}{Consistency (see \Cref{prop:consistencyD1Q2})} & 1st order & 2nd order \\
                \hline
                \multicolumn{2}{|l!{\vrule width 1pt}}{Stability condition (see \Cref{prop:consistencyD1Q2})} & $|\courantNumber|\leq 1$ & $|\courantNumber|< 1$ \\
                \hline
                \multirow{2}{*}{\makecell{Dissipation at low freq. ($\frequency\spaceStep\approx 0$) \\ (see \Cref{rem:dissipationLowFrequencies})}} & Physical symb. &2nd order dissipative & None\\\cline{2-4}
                                             & \makecell{Spurious symb. \\ (renormalized)} &2nd order anti-dissipative & None\\
                \hline
                \multicolumn{2}{|l!{\vrule width 1pt}}{Damping (see \Cref{lemma:dampingD1Q2})} & All modes except $\frequency\spaceStep=0, \pm\pi$ & None \\
                \hline
                \multirow{2}{*}{Propagation speed at low freq. ($\frequency\spaceStep\approx 0$)} & Physical symb. \eqref{eq:gvPhysicalLow} &$\advectionVelocity$ &$\advectionVelocity$\\\cline{2-4}
                                             & Spurious symb. \eqref{eq:gvSpuriousLow} &$-\advectionVelocity$ &$-\advectionVelocity$\\
                \cline{1-2}\noalign{\global\arrayrulewidth=1pt}
                \cline{3-4}\noalign{\global\arrayrulewidth=.4pt}
                \end{tabular}
        \end{center}        
    \end{small}
\end{summarybox}

\begin{remark}[The ``role'' of boundary conditions]
    Boundary conditions transform part of the entering waves (either physical or spurious) into reflected ones (either spurious or physical). 
    The role of transparent boundary conditions is to make the amplitude of the reflected waves equal to zero.
    Interestingly, even if the bulk scheme in linear and thus obeys the superposition principle, which means that modes $(\timeShiftOperator, \fourierShift)$ satisfying the dispersion relation do not mix, boundary conditions turn modes into another. 
    In the following, we indicate $(\timeShiftOperator_{\text{in}}, \fourierShift_{\text{in}}) \rightsquigarrow (\timeShiftOperator_{\text{ref}}, \fourierShift_{\text{ref}})$ when a boundary transforms the mode $(\timeShiftOperator_{\text{in}}, \fourierShift_{\text{in}})$ incident to it into the reflected mode $(\timeShiftOperator_{\text{ref}}, \fourierShift_{\text{ref}})$.
\end{remark}

\paragraph{Fourth-order \lbmScheme{1}{3} scheme}

Sharing its essential structure with the \lbmScheme{1}{2} scheme, we have the scheme introduced in \cite{bellotti2024influence, bellotti2025stability}, where the following \Cref{prop:consistencyD1Q34th} and \Cref{lemma:noDampingD1Q3Fourth} are proved.
\begin{itemize}
    \item \strong{Relaxation}.
    \begin{equation*}
        \conservedMoment_{\indexSpace}^{\indexTime, \collided} = \conservedMoment_{\indexSpace}^{\indexTime}, \qquad \nonConservedMoment_{\indexSpace}^{\indexTime, \collided} = -\nonConservedMoment_{\indexSpace}^{\indexTime} + 2 \underbrace{\courantNumber \conservedMoment_{\indexSpace}^{\indexTime}}_{\nonConservedMoment^{\atEquilibrium}(\conservedMoment_{\indexSpace}^{\indexTime})}, \qquad \nonNonConservedMoment_{\indexSpace}^{\indexTime, \collided} = -\nonNonConservedMoment_{\indexSpace}^{\indexTime} + 2 \underbrace{\tfrac{1}{3}(1+2\courantNumber^2) \conservedMoment_{\indexSpace}^{\indexTime}}_{\nonNonConservedMoment^{\atEquilibrium}(\conservedMoment_{\indexSpace}^{\indexTime})}, \qquad \indexSpace\in\integerInterval{0}{\numberGridPoints + 1}.
    \end{equation*}
    This relaxation parallels the one of the \lbmScheme{1}{2} scheme \eqref{eq:relaxationD1Q2} where the relaxation parameters have been taken equal to two.
    \item \strong{Transport}. After having computed $\distributionFunctionLetter_{0, \indexSpace}^{\indexTime, \collided}= \conservedMoment_{\indexSpace}^{\indexTime, \collided} - \nonNonConservedMoment_{\indexSpace}^{\indexTime, \collided}$ and $\distributionFunctionLetter_{\pm, \indexSpace}^{\indexTime, \collided}= \tfrac{1}{2}(\pm\nonConservedMoment_{\indexSpace}^{\indexTime, \collided} + \nonNonConservedMoment_{\indexSpace}^{\indexTime, \collided} )$ for $\indexSpace\in\integerInterval{0}{\numberGridPoints + 1}$, we have 
    \begin{equation*}
        \distributionFunctionLetter_{0, \indexSpace}^{\indexTime + 1}= \distributionFunctionLetter_{0, \indexSpace}^{\indexTime, \collided} \qquad \text{and}\qquad 
        \distributionFunctionLetter_{\pm, \indexSpace}^{\indexTime + 1}= \distributionFunctionLetter_{\pm, \indexSpace \mp 1}^{\indexTime, \collided}, \qquad \indexSpace\in\integerInterval{1}{\numberGridPoints}.
    \end{equation*}
    Then, we obtain $\conservedMoment_{\indexSpace}^{\indexTime+1} = \distributionFunctionLetter_{0, \indexSpace}^{\indexTime + 1} + \distributionFunctionLetter_{+, \indexSpace}^{\indexTime + 1} +  \distributionFunctionLetter_{-, \indexSpace}^{\indexTime + 1}$, $\nonConservedMoment_{\indexSpace}^{\indexTime+1} = \distributionFunctionLetter_{+, \indexSpace}^{\indexTime + 1} -  \distributionFunctionLetter_{-, \indexSpace}^{\indexTime + 1}$, and $\nonNonConservedMoment_{\indexSpace}^{\indexTime+1} = \distributionFunctionLetter_{+, \indexSpace}^{\indexTime + 1} +  \distributionFunctionLetter_{-, \indexSpace}^{\indexTime + 1}$ for $ \indexSpace\in\integerInterval{1}{\numberGridPoints}$.
    \item \strong{Boundary conditions}. Compute $(\conservedMoment_0^{\indexTime + 1}, \nonConservedMoment_0^{\indexTime + 1}, \nonNonConservedMoment_0^{\indexTime + 1})$ and $(\conservedMoment_{\numberGridPoints + 1}^{\indexTime + 1}, \nonConservedMoment_{\numberGridPoints+1}^{\indexTime + 1}, \nonNonConservedMoment_{\numberGridPoints+1}^{\indexTime + 1})$ as described in \Cref{sec:BCFourthOrderD1Q3}.
\end{itemize}

This scheme has a computable matrix $\schemeMatrixBulkFourier(\fourierShift)$ such that $\transpose{(\conservedMoment_{\indexSpace}^{\indexTime+1}, \nonConservedMoment_{\indexSpace}^{\indexTime+1}, \nonNonConservedMoment_{\indexSpace}^{\indexTime+1})} = \schemeMatrixBulkFourier(\fourierShift) \transpose{(\conservedMoment_{\indexSpace}^{\indexTime}, \nonConservedMoment_{\indexSpace}^{\indexTime}, \nonNonConservedMoment_{\indexSpace}^{\indexTime})}$ for $\indexSpace\in\relatives$ with characteristic equation
\begin{align}
    \determinant(&\timeShiftOperator\identityMatrix{3} - \schemeMatrixBulkFourier(\fourierShift)) \nonumber\\
    =
    &\bigl( -\tfrac{1}{3}(2\courantNumber - 1)(\courantNumber + 2)\timeShiftOperator^2 + \tfrac{1}{3}(2\courantNumber + 1)(\courantNumber - 2)\timeShiftOperator\bigr )\fourierShift^{-1} 
    +\bigl ( \timeShiftOperator^3 +\tfrac{1}{3}(4\courantNumber^2 - 1)\timeShiftOperator^2-\tfrac{1}{3}(4\courantNumber^2 - 1)\timeShiftOperator - 1\bigr )
    \nonumber \\
    +&\bigl( -\tfrac{1}{3}(2\courantNumber + 1)(\courantNumber - 2)\timeShiftOperator^2 + \tfrac{1}{3}(2\courantNumber - 1)(\courantNumber + 2)\timeShiftOperator\bigr )\fourierShift  = 0.\label{eq:charEquationD1Q3FourthOrder}
\end{align}

\begin{proposition}[Consistency and stability of the fourth-order \lbmScheme{1}{3} scheme on $\relatives$]\label{prop:consistencyD1Q34th}
    The fourth-order \lbmScheme{1}{3} scheme is such that there exists a unique eigenvalue $\timeShiftOperator_{\textnormal{phy}}(\frequency\spaceStep)$ of $\schemeMatrixBulkFourier(e^{i\frequency\spaceStep})$ such that 
    \begin{equation}\label{eq:taylorConsistencyFourthOrder}
        \timeShiftOperator_{\textnormal{phy}}(\frequency\spaceStep) = 1 - i \advectionVelocity\frequency\timeStep - \frac{1}{2} \advectionVelocity^2(\frequency\timeStep)^2 + \frac{i}{6} \advectionVelocity^3(\frequency\timeStep)^3 +  \frac{1}{24} \advectionVelocity^4(\frequency\timeStep)^4+\frac{i\latticeVelocity^4\advectionVelocity}{360}(5\courantNumber^4 - 10\courantNumber^2+2)(\frequency\timeStep)^5 +\bigO{(\frequency\timeStep)^6},
    \end{equation}
    so the scheme is fourth-order accurate.

    The scheme is $\ell^2$--stable, namely there exists $C>0$ such that $\sup_{\frequency\spaceStep\in[-\pi, \pi]} |\schemeMatrixBulkFourier(e^{i\frequency\spaceStep})^{\indexTime}|\leq C$ for all $\indexTime\in\naturals$, if and only if 
    \begin{equation}\label{eq:stabilityConditionFourthOrder}
        |\courantNumber|<\tfrac{1}{2}.
    \end{equation}
\end{proposition}
\begin{lemma}[Absence of damping in the fourth-order \lbmScheme{1}{3} scheme]\label{lemma:noDampingD1Q3Fourth}
    Consider the strict CFL condition $|\courantNumber|<\tfrac{1}{2}$.    
    Let $\fourierShift = e^{i\frequency\spaceStep}$ with $\frequency\spaceStep\in[-\pi, \pi]$ in \eqref{eq:charEquationD1Q3FourthOrder}.
    Then the roots $\timeShiftOperator = \timeShiftOperator(\frequency\spaceStep)$ of \eqref{eq:charEquationD1Q3FourthOrder} are on $\unitCircle$ and distinct for $\frequency\spaceStep\neq 0$.
    For $\frequency\spaceStep = 0$, they are $\{1, -1, -1\}$.
\end{lemma}
This shows that the scheme does not damp any wave (whether physical or spurious) for every harmonics.
Moreover, the scalar Finite Difference scheme having \eqref{eq:charEquationD1Q3FourthOrder} as characteristic equation is $\ell^2$-unstable.
\begin{remark}[On the initialization at equilibrium]
    We warn the reader that the initialization at equilibrium (i.e., similar to \eqref{eq:initializationEquilibrium}) that we have proposed does not ensure that the overall scheme preserve fourth-order accuracy for smooth data.
    With this initialization, the non-attenuated spurious symbols carry $\bigO{\spaceStep}$ amplitudes, so the method is only first-order accurate.
    This issue can be fixed by different choices for the initial data, see \cite{bellotti2024influence}.
\end{remark}
We finish on the propagation of spurious time-checkerboard modes for low spatial frequencies:
\begin{equation}\label{eq:groupVelocitySmoothParasiticD1Q3}
    \timeShiftOperator_{\textnormal{spu}, \pm}(\frequency\spaceStep) = -\bigl ( 1+i \tfrac{\latticeVelocity}{6}(3\courantNumber\pm\sqrt{24-15\courantNumber^2}) \frequency\timeStep + \bigO{(\frequency\timeStep)^2}\bigr ),
\end{equation}
whence spurious smooth-in-space packets travel, one to the left and one to the right, faster than physical waves, as $|\tfrac{\latticeVelocity}{6}(3\courantNumber\pm\sqrt{24-15\courantNumber^2})|>\advectionVelocity$ under \eqref{eq:stabilityConditionFourthOrder}.

\pagebreak

\begin{summarybox}{Summary of the main properties of the fourth-order \lbmScheme{1}{4} scheme}
    \begin{small}
        \begin{center}
            \begin{tabular}{|ll!{\vrule width 1pt}c!{\vrule width 1pt}}
                \cline{1-2}\noalign{\global\arrayrulewidth=1pt}
                \cline{3-3}
                \noalign{\global\arrayrulewidth=.4pt}
                \multicolumn{2}{|l!{\vrule width 1pt}}{Consistency (see \Cref{prop:consistencyD1Q34th})} & 4th order \\
                \hline
                \multicolumn{2}{|l!{\vrule width 1pt}}{Stability condition (see \Cref{prop:consistencyD1Q34th})} & $|\courantNumber|< 1/2$  \\
                \hline
                \multicolumn{2}{|l!{\vrule width 1pt}}{Dissipation and damping (see \Cref{lemma:noDampingD1Q3Fourth})} & None  \\
                \hline
                \multirow{2}{*}{Propagation speed at low freq. ($\frequency\spaceStep\approx 0$)} & Physical symb. \eqref{eq:taylorConsistencyFourthOrder} &$\advectionVelocity$ \\\cline{2-3}
                                             & Spurious symb. \eqref{eq:groupVelocitySmoothParasiticD1Q3} &$-\tfrac{\latticeVelocity}{6}(3\courantNumber\pm\sqrt{24-15\courantNumber^2})$\\
                                             \cline{1-2}\noalign{\global\arrayrulewidth=1pt}
                                             \cline{3-3}
                                             \noalign{\global\arrayrulewidth=.4pt}
            \end{tabular}
        \end{center}        
    \end{small}
\end{summarybox}

\subsubsection{Transparent boundary conditions for the \lbmScheme{1}{2} scheme}\label{sec:transparentD1Q2}

We devise boundary conditions building on the work of \cite{besse2021discrete}, adapted to several numerical unknowns in two fashions.
The first one follows \cite{bellotti2025stability}, and is based on the detailed constructions of eigenvectors.
The second (simpler) one lets the numerical scheme do the job.
We present the case $\courantNumber>0$.

\paragraph{Spatial roots of the characteristic equation}

As previously announced, the construction of transparent boundary conditions relies on seeing \eqref{eq:charEquation} as an equation on $\fourierShift$ parametrized by $\timeShiftOperator\in\closedNeighborhoodInfinity$.
Under the stability condition \eqref{eq:stabConditionD1Q2}, only the coefficient in front of $\fourierShift$ in \eqref{eq:charEquation} vanishes when $\relaxationParameter = \tfrac{2}{1+\courantNumber}$. 
Except in this case, \eqref{eq:charEquation} is a second-order equation in $\fourierShift$.
{Von Neumann} stability ({a fortiori}, $\ell^2$ stability) ensures the following splitting of the roots (see \cite{bellotti2025stability} and references therein for the proofs of Lemmata \ref{lemma:spatialRootsD1Q2} and \ref{lemma:continuousExtension}).
\begin{lemma}[Spatial roots of the characteristic equation of the \lbmScheme{1}{2} scheme]\label{lemma:spatialRootsD1Q2}
    Let $\courantNumber>0$ and the stability condition \eqref{eq:stabConditionD1Q2} be fulfilled.
    \begin{itemize}
        \item  If $\relaxationParameter \neq \tfrac{2}{1+\courantNumber}$, \eqref{eq:charEquation} has two roots, one stable $\stableRoot(\timeShiftOperator)\in\unitDisk$ for $\timeShiftOperator\in\neighborhoodInfinity$ and one unstable $\unstableRoot(\timeShiftOperator)\in\neighborhoodInfinity$ for $\timeShiftOperator\in\neighborhoodInfinity$.
        Their product $\stableRoot(\timeShiftOperator)\unstableRoot(\timeShiftOperator)$ given by $\productRoots= \frac{(1-\courantNumber)\relaxationParameter-2}{(1+\courantNumber)\relaxationParameter-2}$, thus independent of $\timeShiftOperator$.
        The continuous extension of the roots to the unit circle $\unitCircle$ fulfills $\stableRoot(\pm 1) = \pm 1$ and $\unstableRoot(\pm 1) = \pm \productRoots$.

        \item If $\relaxationParameter = \tfrac{2}{1+\courantNumber}$, \eqref{eq:charEquation} has one root $\stableRoot(\timeShiftOperator) = \frac{2\courantNumber\timeShiftOperator}{(\courantNumber+1)\timeShiftOperator^2+\courantNumber-1}\in\unitDisk$ for $\timeShiftOperator\in\neighborhoodInfinity$, with its continuous extension to $\unitCircle$ satisfying $\stableRoot(\pm 1) = \pm 1$ as $\courantNumber>0$.
    \end{itemize}
\end{lemma}

Only for this simple scheme, we shall rely on explicit formul\ae{} for $\stableRoot(\timeShiftOperator)$ ($\unstableRoot(\timeShiftOperator)$ can be computed from $\stableRoot(\timeShiftOperator)$ through $\productRoots$): when $\relaxationParameter \neq \tfrac{2}{1+\courantNumber}$, we obtain that  
\begin{equation}
    \stableRoot(\timeShiftOperator) = \frac{(\relaxationParameter-1) \timeShiftOperator^{-1} - \timeShiftOperator + \sqrt{\timeShiftOperator^2 + ((\courantNumber^2-1)\relaxationParameter^2 + 2\relaxationParameter-2) + (\relaxationParameter-1)^2\timeShiftOperator^{-2}}}{(\courantNumber + 1)\relaxationParameter-2}, \label{eq:paramStableRoot}
\end{equation}
which becomes Equation (3.8) in the work of Besse \emph{et al.} in the case $\relaxationParameter = 2$.
The general solution of the $\timeShiftOperator$-transformed numerical scheme on $\relatives$ is
\begin{equation}\label{eq:generalSolution}
    \begin{pmatrix}
        \zTransformed{\conservedMoment}_{\indexSpace}(\timeShiftOperator)\\
        \zTransformed{\nonConservedMoment}_{\indexSpace}(\timeShiftOperator)
    \end{pmatrix} = 
    \coefficientStable(\timeShiftOperator)
    \begin{pmatrix}
        1\\
        \eigenvectorLetter_{\stableMarker}(\timeShiftOperator)
    \end{pmatrix}
    \stableRoot(\timeShiftOperator)^{\indexSpace}
    +
    \coefficientUnstable(\timeShiftOperator)
    \begin{pmatrix}
        1\\
        \eigenvectorLetter_{\unstableMarker}(\timeShiftOperator)
    \end{pmatrix}
    \unstableRoot(\timeShiftOperator)^{\indexSpace}, \qquad \indexSpace\in\relatives,
\end{equation}
for $\timeShiftOperator\in\neighborhoodInfinity$.
Here, we have chosen a normalization for the eigenvectors $\transpose{(1, \eigenvectorLetter_{\stableMarker}(\timeShiftOperator))}\in\kernel(\timeShiftOperator\identityMatrix{2} - \schemeMatrixBulkFourier(\stableRoot(\timeShiftOperator)))$ and $\transpose{(1, \eigenvectorLetter_{\unstableMarker}(\timeShiftOperator))}\in\kernel(\timeShiftOperator\identityMatrix{2} - \schemeMatrixBulkFourier(\unstableRoot(\timeShiftOperator)))$, which is a convenient one and the same as in \cite{bellotti2025stability}.
This entails that, with $\fourierShift(\timeShiftOperator)$ being either $\stableRoot(\timeShiftOperator)$ or $\unstableRoot(\timeShiftOperator)$:
\begin{equation*}
    \eigenvectorLetter(\timeShiftOperator) = \frac{2\timeShiftOperator - (1+\courantNumber\relaxationParameter)\fourierShift(\timeShiftOperator)^{-1} - (1-\courantNumber\relaxationParameter)\fourierShift(\timeShiftOperator)}{(1-\relaxationParameter)(\fourierShift(\timeShiftOperator)^{-1} -  \fourierShift(\timeShiftOperator))}.
\end{equation*}
The coefficients $\coefficientStable(\timeShiftOperator)$ and $\coefficientUnstable(\timeShiftOperator)$ in \eqref{eq:generalSolution} are determined by the initial data. 

\begin{remark}[Relaxation scheme]
    The previous expression for $\eigenvectorLetter(\timeShiftOperator)$ seems ill-defined when $\relaxationParameter = 1$. 
    This is actually not true as, injecting the expressions of $\fourierShift(\timeShiftOperator)$, a factor $1-\relaxationParameter$ appears also in the numerator, hence simplifying the expression.

    In this case, $\nonConservedMoment$ relaxes on the equilibrium, thus is slave of $\conservedMoment$.
    Anything can be put on the non-conserved moment on the boundary (i.e., taken as $\nonConservedMoment_0^{\indexTime}$ and $\nonConservedMoment_{\numberGridPoints + 1}^{\indexTime}$), and considering eigenvectors is superfluous.
\end{remark}

\begin{lemma}[Continuous extension of the eigenvectors]\label{lemma:continuousExtension}
    Let $\courantNumber>0$.
    The stable eigenvector ({i.e.}, $\eigenvectorLetter_{\stableMarker}(\timeShiftOperator)$) can be continuously extended to $\unitCircle$ for every $\relaxationParameter \in (0, 2]$.
    Conversely, the unstable eigenvector ({i.e.}, $\eigenvectorLetter_{\unstableMarker}(\timeShiftOperator)$) can be continuously extended to $\unitCircle$ for every $\relaxationParameter \in (0, 2)$, but cannot be so when $\relaxationParameter = 2$, in which case 
    \begin{equation*}
        \eigenvectorLetter_{\unstableMarker}(\timeShiftOperator) = \mp 2\courantNumber (\timeShiftOperator \mp 1)^{-1}+\bigO{1}\qquad \text{as}\qquad \timeShiftOperator\to \pm 1.
    \end{equation*}
\end{lemma}
This result can have consequences on the effect of computing the coefficients of transparent boundary condition using floating-point arithmetic, in a way that we analyze with numerical experiments.

\paragraph{Systemic approach}\label{sec:systemicRight}

The first approach we utilize, called ``\strong{systemic}'', builds on an eigenvector of the system which remains fixed whether we deal with $\conservedMoment$ or $\nonConservedMoment$.
Probably the most natural approach in view of \eqref{eq:generalSolution}, it is not the most flexible one when dealing with more involved schemes, since it requires a study of the eigenvectors. 

\begin{itemize}

\item \strong{Right boundary}.
We follow \cite{besse2021discrete} and consider \eqref{eq:generalSolution} for $\indexSpace\geq \numberGridPoints + 1$.
Since we want that the solution decays as $\indexSpace\to+\infty$ to remain $\ell^2$-integrable, we take $\coefficientUnstable(\timeShiftOperator)\equiv 0$. 
We employ $\conservedMoment_{\numberGridPoints}$ as a boundary datum to the left (this choice is arbitrary, as we could have used $\nonConservedMoment_{\numberGridPoints}$, see the discussion in \Cref{sec:onceAgainChoiceEigenvalue}).
This gives $\zTransformed{\conservedMoment}_{\numberGridPoints}(\timeShiftOperator) = \coefficientStable(\timeShiftOperator) \stableRoot(\timeShiftOperator)^{\numberGridPoints}$, so $\coefficientStable(\timeShiftOperator) = \stableRoot(\timeShiftOperator)^{-\numberGridPoints} \zTransformed{\conservedMoment}_{\numberGridPoints}(\timeShiftOperator)$, and \eqref{eq:generalSolution} becomes 
\begin{equation}\label{eq:generalSolutionRightWithoutFreeCoefficient}
    \begin{pmatrix}
        \zTransformed{\conservedMoment}_{\indexSpace}(\timeShiftOperator)\\
        \zTransformed{\nonConservedMoment}_{\indexSpace}(\timeShiftOperator)
    \end{pmatrix} = 
    \begin{pmatrix}
        1\\
        \eigenvectorLetter_{\stableMarker}(\timeShiftOperator)
    \end{pmatrix}
    \stableRoot(\timeShiftOperator)^{\indexSpace-\numberGridPoints} \zTransformed{\conservedMoment}_{\numberGridPoints}(\timeShiftOperator), \qquad \indexSpace\geq\numberGridPoints + 1.
\end{equation}
Evaluating \eqref{eq:generalSolutionRightWithoutFreeCoefficient} for $\indexSpace = \numberGridPoints + 1$ provides
\begin{equation}\label{eq:tmp2}
    \zTransformed{\conservedMoment}_{\numberGridPoints + 1}(\timeShiftOperator)
    =
    \stableRoot(\timeShiftOperator) \zTransformed{\conservedMoment}_{\numberGridPoints}(\timeShiftOperator)
    \qquad \text{and}\qquad 
    \zTransformed{\nonConservedMoment}_{\numberGridPoints + 1}(\timeShiftOperator)
    =
    \eigenvectorLetter_{\stableMarker}(\timeShiftOperator) \stableRoot(\timeShiftOperator)\zTransformed{\conservedMoment}_{\numberGridPoints}(\timeShiftOperator).
\end{equation}
Now we have to invert the $\timeShiftOperator$-transform as presented in \Cref{sec:notations}: in particular, we find the Laurent series of $\stableRoot(\timeShiftOperator)$ and $\eigenvectorLetter_{\stableMarker}(\timeShiftOperator) \stableRoot(\timeShiftOperator)$ at infinity.
We claim---and prove in the following pages---that 
\begin{equation}\label{eq:laurentExpansions}
    \stableRoot(\timeShiftOperator) = \sum_{\indexTime = 0}^{+\infty}\coefficientsLaurentStableRoot_{\indexTime}\timeShiftOperator^{-2\indexTime-1} \qquad \text{and}\qquad 
    \eigenvectorLetter_{\stableMarker}(\timeShiftOperator)\stableRoot(\timeShiftOperator) = \sum_{\indexTime = 0}^{+\infty}\coefficientMultiplierNonConservedMomentRight_{\indexTime}\timeShiftOperator^{-2\indexTime-1}.
\end{equation}
From \eqref{eq:tmp2} and using \eqref{eq:laurentExpansions} in the inverse $\timeShiftOperator$-transformation, we obtain for $\indexTime\geq 1$
\begin{align}
    \conservedMoment_{\numberGridPoints + 1}^{\indexTime} &= \sum_{k = 0}^{\lfloor (\indexTime - 1)/ 2 \rfloor} \coefficientsLaurentStableRoot_{k} \conservedMoment_{\numberGridPoints}^{\indexTime-2k - 1},\label{eq:rightBCU}\\
    \nonConservedMoment_{\numberGridPoints + 1}^{\indexTime} &= \sum_{k = 0}^{\lfloor (\indexTime - 1)/ 2 \rfloor} \coefficientMultiplierNonConservedMomentRight_{k} \conservedMoment_{\numberGridPoints}^{\indexTime-2k - 1}. \label{eq:bcRightNonConserved}
\end{align}
\begin{remark}[A first comparison with \cite{heubes2014exact}]
    Comparing \eqref{eq:rightBCU}--\eqref{eq:bcRightNonConserved} to the work of \cite{heubes2014exact}, we see that their boundary conditions populate the distribution function $\distributionFunctionLetter_{-, \numberGridPoints}^{\indexTime + 1}$ at $\spaceGridPoint{\numberGridPoints}$ rather than at $\spaceGridPoint{\numberGridPoints + 1}$, as they already take transport and collision into account. Their computation of $\distributionFunctionLetter_{-, \numberGridPoints}^{\indexTime + 1}$ uses the values of distribution functions $\distributionFunctionLetter_{\pm, \numberGridPoints}$ at previous time steps with the same parity as $\indexTime + 1$. In our case, the parity is altered, since we populate at $\spaceGridPoint{\numberGridPoints + 1}$ and let the scheme perform the collision and transport. 
    Still, the two approaches are equivalent.
    A specific trait of the work by Heubes \strong{et al.} is that initial data are compactly supported up to an additive constant. 
    In our work, we assume data to be compactly supported in the domain, so the two settings are analogous because of linearity combined with the fact that constants discrete solutions are invariant by the numerical schemes.
\end{remark}
We now discuss ways of finding the explicit expressions for $\coefficientsLaurentStableRoot_{\indexTime}$ and $\coefficientMultiplierNonConservedMomentRight_{\indexTime}$ in \eqref{eq:laurentExpansions}--\eqref{eq:rightBCU}--\eqref{eq:bcRightNonConserved}. 
We only consider the case $\relaxationParameter\geq 1$ for the sake of room in the manuscript.
\begin{itemize}
    \item \strong{Relaxation (Lax-Friedrichs) scheme: $\relaxationParameter = 1$}.
    Using \eqref{eq:paramStableRoot}, we have that 
    \begin{equation}\label{eq:tmp4}
        \stableRoot(\timeShiftOperator) = \frac{-\timeShiftOperator + \sqrt{\courantNumber^2-1+\timeShiftOperator^2}}{\courantNumber- 1} = \frac{1}{\courantNumber-1}\sum_{\indexTime = 0}^{+\infty} \binom{1/2}{\indexTime+1}(\courantNumber^2-1)^{\indexTime + 1}\timeShiftOperator^{-2\indexTime - 1},
    \end{equation}
    thanks to the generalized binomial theorem. We can thus claim  that $\coefficientsLaurentStableRoot_{\indexTime} = \binom{1/2}{\indexTime+1}(\courantNumber+1)(\courantNumber^2-1)^{\indexTime}$.
    By the well-known identity $\binom{n}{k} = \binom{n}{k-1}(n-k+1)/k$, we obtain the recurrence relation
    \begin{equation*}
        \coefficientsLaurentStableRoot_{0} = \tfrac{1}{2}(\courantNumber+1) \qquad \text{and} \qquad  \coefficientsLaurentStableRoot_{\indexTime} = \frac{\tfrac{1}{2}-\indexTime}{1+\indexTime}(\courantNumber^2 - 1)\coefficientsLaurentStableRoot_{\indexTime - 1}, \qquad \indexTime\geq 1.
    \end{equation*} 

    We can estimate the speed of convergence of $\coefficientsLaurentStableRoot_{\indexTime}$ by generalizing the Stirling formula: 
    \begin{equation*}
        \binom{1/2}{\indexTime+1} = \frac{\sqrt{\pi}}{2\Gamma(\indexTime + 2)\Gamma(1/2-\indexTime)} = \frac{(-1)^{\indexTime} \Gamma(\indexTime + 1/2)}{2\sqrt{\pi}\Gamma(\indexTime + 2)} = \frac{(-1)^{\indexTime} }{2\sqrt{\pi}} \indexTime^{-3/2} (1 + \bigO{\indexTime^{-1}}),
    \end{equation*}
    with the second equality from the Euler's reflection formula for the $\Gamma$ function, $\Gamma(1/2-\indexTime) = \tfrac{\pi}{\cos(\indexTime\pi)}\Gamma(\indexTime + 1/2)^{-1} = (-1)^{\indexTime}\pi \Gamma(\indexTime + 1/2)^{-1}$, and the last identity from \cite[Formula 6.1.47]{abramowitz1948handbook} for the asymptotics of the Gamma function.
    This entails 
    \begin{equation}\label{eq:asymptoticRelaxationScheme}
        \coefficientsLaurentStableRoot_{\indexTime} = \frac{1+\courantNumber}{2\sqrt{\pi}} (1-\courantNumber^2)^{\indexTime} \indexTime^{-3/2} (1 + \bigO{\indexTime^{-1}}) , 
    \end{equation}
    so the speed of convergence of $\coefficientsLaurentStableRoot_{\indexTime}$ to zero is geometric with a modulation by $\indexTime^{-3/2}$.
    We shall come back to this estimate in \Cref{sec:backAsympt}.
    \item \strong{Monotone over-relaxing scheme: $1< \relaxationParameter<\tfrac{2}{1+\courantNumber}$}.
    This regime is called like so because it ensures that the scheme is monotone according to the definition by \cite{aregba2024convergence}.
    From \eqref{eq:paramStableRoot}, we are brought to the study of 
    \begin{equation}\label{eq:tmp5}
        \sqrt{\timeShiftOperator^2 + ((\courantNumber^2-1)\relaxationParameter^2 + 2\relaxationParameter-2) + (\relaxationParameter-1)^2\timeShiftOperator^{-2}} = \timeShiftOperator  ( 1 - 2\argumentLegendrePolynomials \timeShiftOperatorModified^{-2}+ \timeShiftOperatorModified^{-4} )\sum_{\indexTime = 0}^{+\infty} \legendrePolynomial{\indexTime}(\argumentLegendrePolynomials)\timeShiftOperatorModified^{-2\indexTime},
    \end{equation}
    where we introduced $\timeShiftOperatorModified^{-2}\definitionEquality (\relaxationParameter - 1)\timeShiftOperator^{-2}$ and $\argumentLegendrePolynomials\definitionEquality \frac{(\courantNumber^2-1)\relaxationParameter^2 + 2\relaxationParameter-2}{2(1-\relaxationParameter)}$, with $\legendrePolynomial{\indexTime}$ being the $\indexTime$-th Legendre polynomial. 
    This way of proceeding follows the computations of Besse \strong{et al.}
    The previous formal identity can be written for every $\argumentLegendrePolynomials$.
    However, here $\argumentLegendrePolynomials$ lays outside $[-1, 1]$, thus $\legendrePolynomial{\indexTime}(\argumentLegendrePolynomials)$ quickly tends to infinity with $\indexTime$ (geometrically with a modulation by $\indexTime^{-1/2}$): by the Laplace-Heine formula \cite[Theorem 8.21.1]{szeg1939orthogonal}
    \begin{equation}\label{eq:LaplaceHeineFormula}
        \legendrePolynomial{\indexTime}(\argumentLegendrePolynomials)
        \cong
        \frac{1}{\sqrt{2\pi}}  (\argumentLegendrePolynomials^2-1)^{-1/4} (\argumentLegendrePolynomials+\sqrt{\argumentLegendrePolynomials^2-1})^{\indexTime+1/2} \indexTime^{-1/2}\qquad \text{as}\qquad \indexTime\to+\infty.
    \end{equation}
    Although this is not a problem on $\reals$, it causes problems when computations are deployed on floating-point arithmetic, and one must not store the values of $\legendrePolynomial{\indexTime}(\argumentLegendrePolynomials)$.
    We continue by introducing
    \begin{equation*}
        \coefficientsExpansionSquareRoot_0 = \legendrePolynomial{1}(\argumentLegendrePolynomials) - 2\argumentLegendrePolynomials\legendrePolynomial{0}(\argumentLegendrePolynomials) \qquad\text{and}\qquad \coefficientsExpansionSquareRoot_{\indexTime} = \legendrePolynomial{\indexTime + 1}(\argumentLegendrePolynomials) - 2\argumentLegendrePolynomials\legendrePolynomial{\indexTime}(\argumentLegendrePolynomials) + \legendrePolynomial{\indexTime - 1}(\argumentLegendrePolynomials), \quad \indexTime\geq 1.
    \end{equation*}
    One can see---using the Bonnet's recursion formula for Legendre polynomials---that they can be obtained by the recurrence 
    \begin{equation}\label{eq:recurrenceOne}
        \coefficientsExpansionSquareRoot_0 = -\argumentLegendrePolynomials, \qquad
        \coefficientsExpansionSquareRoot_{1} = \legendrePolynomial{2}(\argumentLegendrePolynomials) - 2\argumentLegendrePolynomials^2 + 1, \qquad\text{and}\qquad 
        \coefficientsExpansionSquareRoot_{\indexTime} = \frac{2\indexTime - 1}{\indexTime + 1}\argumentLegendrePolynomials \coefficientsExpansionSquareRoot_{\indexTime-1} - \frac{\indexTime-2}{\indexTime + 1}\coefficientsExpansionSquareRoot_{\indexTime-2}, \quad \indexTime\geq 2. 
    \end{equation}
    Thanks to the coefficients $\coefficientsExpansionSquareRoot_{\indexTime}$, we rewrite 
    \begin{equation}\label{eq:legendreAndBn}
        ( 1 - 2\argumentLegendrePolynomials\timeShiftOperatorModified^{-2} + \timeShiftOperatorModified^{-4} )\sum_{\indexTime = 0}^{+\infty} \legendrePolynomial{\indexTime}(\argumentLegendrePolynomials)\timeShiftOperatorModified^{-2\indexTime} = \sum_{\indexTime = 0}^{+\infty}\coefficientsExpansionSquareRoot_{\indexTime}\timeShiftOperatorModified^{-2(\indexTime+1)} + 1 
        = \sum_{\indexTime = 0}^{+\infty}(\relaxationParameter-1)^{\indexTime +1}\coefficientsExpansionSquareRoot_{\indexTime}\timeShiftOperator^{-2(\indexTime+1)} + 1 
    \end{equation}
    so that we obtain 
    \begin{equation}\label{eq:laurentSeriesStableRoot}
        \stableRoot(\timeShiftOperator) =\frac{\relaxationParameter - 1}{(\courantNumber + 1)\relaxationParameter-2} \Bigl (\timeShiftOperator^{-1}  +  \sum_{\indexTime = 0}^{+\infty}(\relaxationParameter - 1)^{\indexTime} \coefficientsExpansionSquareRoot_{\indexTime}\timeShiftOperator^{-2\indexTime-1}  \Bigr ) = \sum_{\indexTime = 0}^{+\infty}\coefficientsLaurentStableRoot_{\indexTime}\timeShiftOperator^{-2\indexTime - 1}, 
    \end{equation}
    where 
    \begin{equation}\label{eq:definitionSAsProduct}
        \coefficientsLaurentStableRoot_{0} = \frac{\relaxationParameter - 1}{(\courantNumber + 1)\relaxationParameter-2} (1+\coefficientsExpansionSquareRoot_0) = 1+\tfrac{1}{2}(\courantNumber-1)\relaxationParameter \qquad \text{and}\qquad \coefficientsLaurentStableRoot_{\indexTime} = \frac{(\relaxationParameter-1)^{\indexTime+1}}{(\courantNumber + 1)\relaxationParameter-2} \coefficientsExpansionSquareRoot_{\indexTime}, \quad \indexTime\geq 1.
    \end{equation}

    Recalling that we also have $\coefficientsExpansionSquareRoot_{\indexTime}=(\legendrePolynomial{\indexTime - 1}(\argumentLegendrePolynomials)-\legendrePolynomial{\indexTime + 1}(\argumentLegendrePolynomials))/(2\indexTime + 1)$ for $\indexTime\geq 1$, \eqref{eq:LaplaceHeineFormula} entails that:
    \begin{equation}\label{eq:expansionBBelowZeroCoefficient}
        \coefficientsExpansionSquareRoot_{\indexTime} = \frac{1}{\sqrt{2\pi}} (\argumentLegendrePolynomials^2-1)^{-1/4} \frac{1-\argumentLegendrePolynomials^2 - \argumentLegendrePolynomials\sqrt{\argumentLegendrePolynomials^2-1}}{\sqrt{\argumentLegendrePolynomials+\sqrt{\argumentLegendrePolynomials^2-1}}} \bigl (\argumentLegendrePolynomials+\sqrt{\argumentLegendrePolynomials^2-1} \bigr )^{\indexTime} (\indexTime^{-3/2} + \bigO{\indexTime^{-5/2} }).
    \end{equation}

    $\coefficientsExpansionSquareRoot_{\indexTime} \sim  (\argumentLegendrePolynomials+\sqrt{\argumentLegendrePolynomials^2-1}  )^{\indexTime} \indexTime^{-3/2}$ tends (essentially) geometrically to infinity. 
    It is therefore catastrophic to use \eqref{eq:definitionSAsProduct} with floating-point arithmetic, as we experience overflows in the computations of $\coefficientsExpansionSquareRoot_{\indexTime}$.
    Then
    \begin{equation*}
        \coefficientsLaurentStableRoot_{\indexTime} = \frac{1}{\sqrt{2\pi}} (\argumentLegendrePolynomials^2-1)^{-1/4} \frac{(1-\argumentLegendrePolynomials^2 - \argumentLegendrePolynomials\sqrt{\argumentLegendrePolynomials^2-1})(\relaxationParameter - 1)}{((\courantNumber + 1)\relaxationParameter-2)\sqrt{\argumentLegendrePolynomials+\sqrt{\argumentLegendrePolynomials^2-1}}} \bigl ((\relaxationParameter - 1)(\argumentLegendrePolynomials+\sqrt{\argumentLegendrePolynomials^2-1}) \bigr )^{\indexTime} (\indexTime^{-3/2} + \bigO{\indexTime^{-5/2} }),
    \end{equation*}
    so tends to zero geometrically (since $(\relaxationParameter - 1)(\argumentLegendrePolynomials+\sqrt{\argumentLegendrePolynomials^2-1} ) \in [0, 1)$) with a modulation by $\indexTime^{-3/2}$, which is fine.
    Overall, we need an alternative to \eqref{eq:definitionSAsProduct} to be deployed for actual computations, contrarily to the setting of Besse \strong{at al.}
    It is quite interesting that the regime $\relaxationParameter< \tfrac{2}{1+\courantNumber}$ is one of the two cases where \cite{heubes2014exact} (\confer{}, their Lemma 13) were able to prove that their coefficients, analogous to $\coefficientsLaurentStableRoot_{\indexTime}$, tend to zero.
    Their proof relies on an upper bound (and not an asymptotic estimate) with control by a geometric term accompanied by a loose estimate on the modulation $\sim \indexTime^{-1}$ ($\indexTime^{-1}\gg \indexTime^{-3/2}$ for large $\indexTime$).
    Bringing the digression to an end, the formula we need is recurrent and  obtained by multiplying both sides of \eqref{eq:recurrenceOne} by $\frac{(\relaxationParameter-1)^{\indexTime+1}}{(\courantNumber + 1)\relaxationParameter-2}$:
    \begin{equation}\label{eq:recurrenceS}
        \coefficientsLaurentStableRoot_{\indexTime} = \frac{2\indexTime - 1}{\indexTime + 1} (1-\relaxationParameter+\tfrac{1}{2}(1-\courantNumber^2)\relaxationParameter^2)\coefficientsLaurentStableRoot_{\indexTime-1} - \frac{\indexTime - 2}{\indexTime + 1}(\relaxationParameter-1)^2 \coefficientsLaurentStableRoot_{\indexTime-2}, \qquad \indexTime\geq 3.
    \end{equation}

    We now consider $\eigenvectorLetter_{\stableMarker}(\timeShiftOperator) \stableRoot(\timeShiftOperator)$: formal computations yield 
    \begin{equation*}
        \eigenvectorLetter_{\stableMarker}(\timeShiftOperator) \stableRoot(\timeShiftOperator) = \frac{A(\timeShiftOperator) + B(\timeShiftOperator) \sqrt{\timeShiftOperator^4 + ((\courantNumber^2-1)\relaxationParameter^2 + 2\relaxationParameter-2)\timeShiftOperator^2 + (\relaxationParameter-1)^2}}{C(\timeShiftOperator) + D(\timeShiftOperator) \sqrt{\timeShiftOperator^4 + ((\courantNumber^2-1)\relaxationParameter^2 + 2\relaxationParameter-2)\timeShiftOperator^2 + (\relaxationParameter-1)^2}},
    \end{equation*}
    where $A(\timeShiftOperator), B(\timeShiftOperator), C(\timeShiftOperator)$, and $D(\timeShiftOperator)$ are suitable polynomials in $\timeShiftOperator$, depending on $\courantNumber$ and $\relaxationParameter$.
    We multiply numerator and denominator by $C(\timeShiftOperator) - D(\timeShiftOperator) \sqrt{\timeShiftOperator^4 + ((\courantNumber^2-1)\relaxationParameter^2 + 2\relaxationParameter-2)\timeShiftOperator^2 + (\relaxationParameter-1)^2}$, and simplify:
    \begin{equation}\label{eq:fraction}
        \eigenvectorLetter_{\stableMarker}(\timeShiftOperator) \stableRoot(\timeShiftOperator) = \frac{\mathscr{N}(\timeShiftOperator)}{\mathscr{D}(\timeShiftOperator)},
    \end{equation}
    where 
    \begin{multline*}
        \mathscr{N}(\timeShiftOperator) = 
        {\timeShiftOperator}^{3}  + {\left({\left({\courantNumber}^{2} + {\courantNumber} - 1\right)} {\relaxationParameter}^{2} - {\left({\courantNumber} - 2\right)} {\relaxationParameter} - 2\right)} {\timeShiftOperator} + \bigl ( {\left(2  {\courantNumber} + 1\right)} {\relaxationParameter}^{2} - {\courantNumber} {\relaxationParameter}^{3} - {\left({\courantNumber} + 2\right)} {\relaxationParameter}  + 1 \bigr ) \timeShiftOperator^{-1} \\
        - {\left({\courantNumber} {\relaxationParameter}^{2} - {\left({\courantNumber} + 1\right)} {\relaxationParameter} + {\timeShiftOperator}^{2} + 1\right)} \sqrt{\timeShiftOperator^2 + ((\courantNumber^2-1)\relaxationParameter^2 + 2\relaxationParameter-2) + (\relaxationParameter-1)^2\timeShiftOperator^{-2}}, \\
        \mathscr{D}(\timeShiftOperator) = - ((\courantNumber+1)\relaxationParameter - 2) ((\relaxationParameter-1)^2 - \timeShiftOperator^2).
    \end{multline*}
    We start by discussing the Laurent series relative to $1/\mathscr{D}(\timeShiftOperator)$:
    \begin{equation*}
        \frac{1}{\mathscr{D}(\timeShiftOperator)} = -\frac{1}{((\courantNumber+1)\relaxationParameter - 2)(\relaxationParameter-1)^2}\frac{1}{1 -  ( \frac{\timeShiftOperator}{\relaxationParameter-1} )^2} = \frac{1}{((\courantNumber+1)\relaxationParameter - 2)(\relaxationParameter-1)^2} \sum_{\indexTime = 1}^{+\infty}\Bigl ( \frac{\timeShiftOperator}{\relaxationParameter-1} \Bigr )^{-2\indexTime}.
    \end{equation*}
    For the numerator, using \eqref{eq:legendreAndBn}
    \begin{multline*}
        \mathscr{N}(\timeShiftOperator) = 
        \bigl (\tfrac{1}{2}(\courantNumber^2-1)\relaxationParameter^2  + 2(\relaxationParameter-1)\bigr ) {\timeShiftOperator} \\
        + \bigl ( {\left(2  {\courantNumber} + 1\right)} {\relaxationParameter}^{2} - {\courantNumber} {\relaxationParameter}^{3} - {\left({\courantNumber} + 2\right)} {\relaxationParameter}  + 1 - (\relaxationParameter-1)^2\coefficientsExpansionSquareRoot_1 - {\left({\courantNumber} {\relaxationParameter}^{2} - {\left({\courantNumber} + 1\right)} {\relaxationParameter} +  1\right)} (\relaxationParameter-1)\coefficientsExpansionSquareRoot_0\bigr ) \timeShiftOperator^{-1} \\
        -  \sum_{\indexTime = 1}^{+\infty} (\relaxationParameter-1)^{\indexTime +1} \Bigl ( (\relaxationParameter-1)\coefficientsExpansionSquareRoot_{\indexTime + 1}+ {\left({\courantNumber} {\relaxationParameter}^{2} - {\left({\courantNumber} + 1\right)} {\relaxationParameter} +  1\right)} \coefficientsExpansionSquareRoot_{\indexTime} \Bigr ) \timeShiftOperator^{-2\indexTime-1} = \sum_{\indexTime = 0}^{+\infty} \coefficientNumeratorRight_{\indexTime} \timeShiftOperator^{-2\indexTime + 1},
    \end{multline*}
    where we set 
    \begin{align}
        \coefficientNumeratorRight_0&= \tfrac{1}{2}(\courantNumber^2-1)\relaxationParameter^2  + 2(\relaxationParameter-1),\nonumber \\ 
        \coefficientNumeratorRight_1&=  {\left(2  {\courantNumber} + 1\right)} {\relaxationParameter}^{2} - {\courantNumber} {\relaxationParameter}^{3} - {\left({\courantNumber} + 2\right)} {\relaxationParameter}  + 1 - (\relaxationParameter-1)((\relaxationParameter-1)\coefficientsExpansionSquareRoot_1 + {\left({\courantNumber} {\relaxationParameter}^{2} - {\left({\courantNumber} + 1\right)} {\relaxationParameter} +  1\right)}\coefficientsExpansionSquareRoot_0),\nonumber \\
        \coefficientNumeratorRight_{\indexTime}&=  -(\relaxationParameter-1)^{\indexTime}  ( (\relaxationParameter-1)\coefficientsExpansionSquareRoot_{\indexTime}+ {\left({\courantNumber} {\relaxationParameter}^{2} - {\left({\courantNumber} + 1\right)} {\relaxationParameter} +  1\right)} \coefficientsExpansionSquareRoot_{\indexTime-1}  ), \qquad \indexTime\geq 2.\label{eq:tmp1}
    \end{align}
    Analogously to $\coefficientsLaurentStableRoot_{\indexTime}$, one has to avoid computing $\coefficientNumeratorRight_{\indexTime}$ from $\coefficientsExpansionSquareRoot_{\indexTime}$ using \eqref{eq:tmp1} in floating-point arithmetic.
    We therefore find the recurrence relation satisfied by $\coefficientNumeratorRight_{\indexTime}$, which reads 
    \begin{equation}\label{eq:recurrecenceRight}
        \coefficientNumeratorRight_{\indexTime} = \mathscr{V}_{-1}(\indexTime)\coefficientNumeratorRight_{\indexTime-1}+ \mathscr{V}_{-2}(\indexTime)\coefficientNumeratorRight_{\indexTime-2}, \qquad \indexTime\geq 3, 
    \end{equation}
    where the weights $\mathscr{V}_{-1}(\indexTime)$ and $\mathscr{V}_{-1}(\indexTime)$ have explicit yet involved expressions given in \Cref{app:expressionWeights}. 
    Expanding the Cauchy product of formal series:
    \begin{align*}
        \eigenvectorLetter_{\stableMarker}(\timeShiftOperator) \stableRoot(\timeShiftOperator) &=\frac{1}{(\courantNumber+1)\relaxationParameter - 2}  \Biggl ( \sum_{\indexTime = 0}^{+\infty} \coefficientNumeratorRight_{\indexTime} \timeShiftOperator^{-2\indexTime + 1} \Biggr ) \Biggl ( \sum_{\indexTime = 1}^{+\infty}(\relaxationParameter-1)^{2(\indexTime-1)}\timeShiftOperator^{-2\indexTime} \Biggr )\\
        &=\frac{1}{(\courantNumber+1)\relaxationParameter - 2} \bigl ( \coefficientNumeratorRight_0\timeShiftOperator^{-1} + (\coefficientNumeratorRight_0 (\relaxationParameter - 1)^2 + \coefficientNumeratorRight_1)\timeShiftOperator^{-3} + (\coefficientNumeratorRight_0(\relaxationParameter-1)^4 + \coefficientNumeratorRight_1(\relaxationParameter-1)^2 + \coefficientNumeratorRight_2)\timeShiftOperator^{-5}+\dots \bigr )\\
        &=\sum_{\indexTime = 0}^{+\infty}\coefficientMultiplierNonConservedMomentRight_{\indexTime}\timeShiftOperator^{-2\indexTime - 1}, \qquad \text{where}\quad \coefficientMultiplierNonConservedMomentRight_{\indexTime}=\frac{1}{(\courantNumber+1)\relaxationParameter - 2}\sum_{k=0}^{\indexTime}(\relaxationParameter-1)^{2(\indexTime-k)}\coefficientNumeratorRight_{k}.
    \end{align*}
    Notice that $\coefficientMultiplierNonConservedMomentRight_{\indexTime}$ can be computed recursively, hence in a cheaper way, using 
    \begin{equation}\label{eq:recurrenceEigv}
        \coefficientMultiplierNonConservedMomentRight_{\indexTime + 1} = (\relaxationParameter-1)^2\coefficientMultiplierNonConservedMomentRight_{\indexTime} + \frac{1}{(\courantNumber+1)\relaxationParameter - 2}\coefficientNumeratorRight_{\indexTime + 1}.
    \end{equation}
    \item \strong{Scheme with no unstable spatial root: $\relaxationParameter=\tfrac{2}{1+\courantNumber}$}. 
    In this setting, the expression $\stableRoot(\timeShiftOperator)$ is particularly simple, \confer{} \Cref{lemma:spatialRootsD1Q2}.
    Assume for the moment that $0<\courantNumber<1$, then
    \begin{equation}\label{eq:rootNoSecondOrder}
        \stableRoot(\timeShiftOperator) = 
    \frac{2\courantNumber}{\courantNumber - 1} \frac{\timeShiftOperator}{1-\Bigl (\sqrt{\frac{1+\courantNumber}{1-\courantNumber}}\timeShiftOperator \Bigr )^2} = 2\courantNumber \sum_{\indexTime = 0}^{+\infty} \frac{(1-\courantNumber)^{\indexTime}}{(1+\courantNumber)^{\indexTime+1}}\timeShiftOperator^{-2\indexTime - 1},
    \end{equation}
    which allows considering $\coefficientsLaurentStableRoot_{\indexTime} = 2\courantNumber \frac{(1-\courantNumber)^{\indexTime}}{(1+\courantNumber)^{\indexTime+1}}$, tending geometrically to zero.
    Notice that this is also valid when $\courantNumber = 1$, with $\coefficientsLaurentStableRoot_{\indexTime}=\delta_{0\indexTime}$. 
    For the eigenvector
    \begin{equation*}
        \eigenvectorLetter_{\stableMarker}(\timeShiftOperator) = \frac{(3\courantNumber+1)(\courantNumber-1) + (\courantNumber+1)^2 \timeShiftOperator^2}{-(\courantNumber-1)^2 + (\courantNumber+1)^2 \timeShiftOperator^2} = 1-4\courantNumber\sum_{\indexTime  = 1}^{+\infty} \frac{(1-\courantNumber)^{2\indexTime - 1}}{(1+\courantNumber)^{2\indexTime}}\timeShiftOperator^{-2\indexTime}.
    \end{equation*}
    Taking the Cauchy product with $\stableRoot(\timeShiftOperator)$, one clearly sees that 
    \begin{multline*}
        \eigenvectorLetter_{\stableMarker}(\timeShiftOperator)\stableRoot(\timeShiftOperator) = \sum_{\indexTime = 0}^{+\infty}\coefficientMultiplierNonConservedMomentRight_{\indexTime}\timeShiftOperator^{-2\indexTime - 1}, \qquad \text{where} \qquad \coefficientMultiplierNonConservedMomentRight_0 = \frac{2\courantNumber}{1+\courantNumber} \\
        \text{and}\qquad  \coefficientMultiplierNonConservedMomentRight_{\indexTime} = \frac{1-\courantNumber}{1+\courantNumber} \coefficientMultiplierNonConservedMomentRight_{\indexTime - 1}
        -\frac{8\courantNumber^2 (1-\courantNumber)^{2\indexTime - 1}}{(1+\courantNumber)^{2\indexTime + 1}}, \qquad \indexTime\geq 1.
    \end{multline*}
    \item \strong{Non-monotone over-relaxing scheme: $\tfrac{2}{1+\courantNumber}<\relaxationParameter\leq 2$}.
    All the discussions done for $1<\relaxationParameter<\tfrac{2}{1+\courantNumber}$ still hold, except for the estimates of the growth of the coefficients.
    It is now safe to manipulate $\legendrePolynomial{\indexTime}(\argumentLegendrePolynomials)$ and $\coefficientsExpansionSquareRoot_{\indexTime}$ in floating-point arithmetics.
    In this case, we have that $\argumentLegendrePolynomials\in (-1, 1)$, hence $\legendrePolynomial{\indexTime}(\argumentLegendrePolynomials)$ tends to zero with $\indexTime$: Legendre polynomials satisfy the Laplace formula \cite[Theorem 8.21.2]{szeg1939orthogonal}
    \begin{equation}\label{eq:LaplaceFormula}
        \legendrePolynomial{\indexTime}(\argumentLegendrePolynomials) = \sqrt{\frac{2}{\pi}} (1-\argumentLegendrePolynomials^2)^{-1/4} \cos\bigl ( \bigl ( \indexTime+\tfrac{1}{2}\bigr )\vartheta - \tfrac{\pi}{4}\bigr ) \indexTime^{-1/2} + \bigO{\indexTime^{-3/2}},
    \end{equation}
    where $\vartheta$ is such that $\cos(\vartheta) = \argumentLegendrePolynomials$.
    This gives 
    \begin{equation}\label{eq:expansionBCloseTwo}
        \coefficientsExpansionSquareRoot_{\indexTime} = \sqrt{\frac{2}{\pi}} (1-\argumentLegendrePolynomials^2)^{1/4} \sin\bigl ( \bigl ( \indexTime+\tfrac{1}{2}\bigr )\vartheta - \tfrac{\pi}{4}\bigr ) \indexTime^{-3/2} + \bigO{\indexTime^{-5/2}}.
    \end{equation}
    Looking at \eqref{eq:definitionSAsProduct}, we see that $(\relaxationParameter-1)^{\indexTime + 1}$ tends to zero geometrically (except for $\relaxationParameter = 2$, the case investigated by Besse \strong{et al.}), and $\coefficientsExpansionSquareRoot_{\indexTime} \sim \indexTime^{-3/2}$ (\confer{} \eqref{eq:expansionBCloseTwo}) tends to zero---although slowly---as well. 
    It is therefore safe to employ \eqref{eq:definitionSAsProduct} with floating-point arithmetic, namely compute and store $\coefficientsExpansionSquareRoot_{\indexTime}$, and eventually multiply it by $(\relaxationParameter-1)^{\indexTime+1}$.

    The expressions for the weights in the expansion of $\eigenvectorLetter_{\stableMarker}(\timeShiftOperator)\stableRoot(\timeShiftOperator)$ are the same as for the case $1<\relaxationParameter<\tfrac{2}{1+\courantNumber}$.
\end{itemize}

\begin{summarybox}{Summary of the trends of the coefficients $\coefficientsExpansionSquareRoot_{\indexTime}$ and $\coefficientsLaurentStableRoot_{\indexTime}$ in \eqref{eq:rightBCU}}
    \begin{small}
        \begin{center}
            \begin{tabular}{|c!{\vrule width 1pt}c|c!{\vrule width 1pt}}
                \cline{2-3}
                \multicolumn{1}{l|}{} & $\coefficientsExpansionSquareRoot_{\indexTime}$ & $\coefficientsLaurentStableRoot_{\indexTime}$ \\
                \cline{1-2}\noalign{\global\arrayrulewidth=1pt}
                \cline{2-3}
                \noalign{\global\arrayrulewidth=.4pt}
                \multicolumn{1}{|c!{\vrule width 1pt}}{$\relaxationParameter=1$} & Does not apply & $\sim \rho^{\indexTime}\indexTime^{-3/2}$ with $|\rho|<1$\\
                \multicolumn{1}{|c!{\vrule width 1pt}}{$1<\relaxationParameter< \frac{2}{1+\courantNumber}$} & $\sim \rho^{\indexTime}\indexTime^{-3/2}$ with $|\rho|>1$ & $\sim \rho^{\indexTime}\indexTime^{-3/2}$ with $|\rho|<1$\\
                \multicolumn{1}{|c!{\vrule width 1pt}}{$\relaxationParameter = \frac{2}{1+\courantNumber}$} & Does not apply & $\sim \rho^{\indexTime}$ with $|\rho|<1$\\
                \multicolumn{1}{|c!{\vrule width 1pt}}{$\frac{2}{1+\courantNumber}<\relaxationParameter<2$} & $\sim \rho^{\indexTime}\indexTime^{-3/2}$ with $|\rho|<1$ & $\sim \rho^{\indexTime}\indexTime^{-3/2}$ with $|\rho|<1$\\
                \multicolumn{1}{|c!{\vrule width 1pt}}{$\relaxationParameter = 2$} & $\sim \indexTime^{-3/2}$ & $\sim \indexTime^{-3/2}$ \\
                                \cline{1-2}\noalign{\global\arrayrulewidth=1pt}
                \cline{2-3}\noalign{\global\arrayrulewidth=.4pt}
                \end{tabular}
        \end{center}        
    \end{small}
\end{summarybox}

\item \strong{Left boundary}.
Consider \eqref{eq:generalSolution} for $\indexSpace\leq 0$.
Since we want the solution decay as $\indexSpace\to-\infty$, \eqref{eq:generalSolution} becomes 
\begin{equation}\label{eq:solD1Q2Left}
    \begin{pmatrix}
        \zTransformed{\conservedMoment}_{\indexSpace}(\timeShiftOperator)\\
        \zTransformed{\nonConservedMoment}_{\indexSpace}(\timeShiftOperator)
    \end{pmatrix} = 
    \coefficientUnstable(\timeShiftOperator)
    \begin{pmatrix}
        1\\
        \eigenvectorLetter_{\unstableMarker}(\timeShiftOperator)
    \end{pmatrix}
    \unstableRoot(\timeShiftOperator)^{\indexSpace} = 
    \coefficientUnstable(\timeShiftOperator)
    \begin{pmatrix}
        1\\
        \eigenvectorLetter_{\unstableMarker}(\timeShiftOperator)
    \end{pmatrix}
    \productRoots^{\indexSpace}\stableRoot(\timeShiftOperator)^{-\indexSpace}, \qquad \indexSpace \leq 0.
\end{equation}
Analogously to the right boundary, we obtain $ \coefficientUnstable(\timeShiftOperator) = \productRoots^{-1}\stableRoot(\timeShiftOperator) \zTransformed{\conservedMoment}_{1}(\timeShiftOperator) = \zTransformed{\conservedMoment}_{0}(\timeShiftOperator)$, hence inverting the $\timeShiftOperator$-transformation:
\begin{equation}\label{eq:bcConservedLeft}
    \conservedMoment_0^{\indexTime} = \frac{1}{\productRoots} \sum_{k = 0}^{\lfloor (\indexTime - 1)/ 2 \rfloor} \coefficientsLaurentStableRoot_{k} \conservedMoment_{1}^{\indexTime-2k - 1}.
\end{equation}
Remark that when $\relaxationParameter=\tfrac{2}{1+\courantNumber}$, we have $\productRoots^{-1} = 0$, hence \eqref{eq:bcConservedLeft} becomes $\conservedMoment_0^{\indexTime} = 0$.
For the non-conserved moment, if we consider that 
\begin{equation*}
    \eigenvectorLetter_{\unstableMarker}(\timeShiftOperator)\stableRoot(\timeShiftOperator)=\frac{1}{\productRoots}\sum_{\indexTime = 0}^{+\infty}\coefficientMultiplierNonConservedMomentLeft_{\indexTime}\timeShiftOperator^{-2\indexTime-1},
    \qquad \text{we obtain}
    \qquad 
    \nonConservedMoment_{0}^{\indexTime} = \frac{1}{\productRoots} \sum_{k = 0}^{\lfloor (\indexTime - 1)/ 2 \rfloor} \coefficientMultiplierNonConservedMomentLeft_{k} \conservedMoment_{1}^{\indexTime-2k - 1}.
\end{equation*}
Let us illustrate how to compute $\coefficientMultiplierNonConservedMomentLeft_{k}$ in the case $\relaxationParameter\in (1, 2]\smallsetminus\{\tfrac{2}{1+\courantNumber}\}$.
Like for the right boundary, we have to compute the Laurent expansion of $\eigenvectorLetter_{\unstableMarker}(\timeShiftOperator)\stableRoot(\timeShiftOperator) = \mathscr{N}(\timeShiftOperator)/\mathscr{D}(\timeShiftOperator)$, where
\begin{multline*}
    \mathscr{N}(\timeShiftOperator) = 
     {\timeShiftOperator}^{3}  + {\left({\left({\courantNumber}^{2} - {\courantNumber} - 1\right)} {\relaxationParameter}^{2} + {\left({\courantNumber} + 2\right)} {\relaxationParameter} - 2\right)} {\timeShiftOperator} + \bigl ( {-\left(2  {\courantNumber} - 1\right)} {\relaxationParameter}^{2} + {\courantNumber} {\relaxationParameter}^{3} + {\left({\courantNumber} - 2\right)} {\relaxationParameter}  + 1 \bigr ) \timeShiftOperator^{-1} \\
    - {\left(-{\courantNumber} {\relaxationParameter}^{2} + {\left({\courantNumber} - 1\right)} {\relaxationParameter} + {\timeShiftOperator}^{2} + 1\right)} \sqrt{\timeShiftOperator^2 + ((\courantNumber^2-1)\relaxationParameter^2 + 2\relaxationParameter-2) + (\relaxationParameter-1)^2\timeShiftOperator^{-2}}, \\
    \mathscr{D}(\timeShiftOperator) = ((\courantNumber+1)\relaxationParameter - 2) ((\relaxationParameter-1)^2 - \timeShiftOperator^2).
\end{multline*}
For the numerator: $\mathscr{N}(\timeShiftOperator) = \sum_{\indexTime = 0}^{+\infty} \coefficientNumeratorLeft_{\indexTime} \timeShiftOperator^{-2\indexTime + 1}$ with
\begin{align*}
    \coefficientNumeratorLeft_0&= \tfrac{1}{2}(\courantNumber^2-1)\relaxationParameter^2  + 2(\relaxationParameter-1), \\ 
    \coefficientNumeratorLeft_1&=  {\left(-2  {\courantNumber} + 1\right)} {\relaxationParameter}^{2} + {\courantNumber} {\relaxationParameter}^{3} - {\left(-{\courantNumber} + 2\right)} {\relaxationParameter}  + 1 - (\relaxationParameter-1)((\relaxationParameter-1)\coefficientsExpansionSquareRoot_1 + {\left({-\courantNumber} {\relaxationParameter}^{2} - {\left(-{\courantNumber} + 1\right)} {\relaxationParameter} +  1\right)}\coefficientsExpansionSquareRoot_0), \\
    \coefficientNumeratorLeft_{\indexTime}&=  -(\relaxationParameter-1)^{\indexTime}  ( (\relaxationParameter-1)\coefficientsExpansionSquareRoot_{\indexTime}+ {\left(-{\courantNumber} {\relaxationParameter}^{2} - {\left(-{\courantNumber} + 1\right)} {\relaxationParameter} +  1\right)} \coefficientsExpansionSquareRoot_{\indexTime-1}  ), \qquad \indexTime\geq 2.
\end{align*}
The recurrence for $\coefficientNumeratorLeft_{\indexTime}$ is the one for $\coefficientNumeratorRight_{\indexTime}$ where the Courant number $\courantNumber$ is replaced by $-\courantNumber$.
We also have 
\begin{equation}\label{eq:recurrenceLeft}
    \coefficientNumeratorLeft_{\indexTime} = \overline{\mathscr{V}}_{-1}(\indexTime)\coefficientNumeratorLeft_{\indexTime-1}+ \overline{\mathscr{V}}_{-2}(\indexTime)\coefficientNumeratorLeft_{\indexTime-2}, \qquad \indexTime\geq 3, 
\end{equation}
where the weights $\overline{\mathscr{V}}_{-1}(\indexTime)$ and $\overline{\mathscr{V}}_{-1}(\indexTime)$ are given in \Cref{app:expressionWeights}. 
We thus obtain 
\begin{equation*}
    \eigenvectorLetter_{\unstableMarker}(\timeShiftOperator)\stableRoot(\timeShiftOperator) =\sum_{\indexTime = 0}^{+\infty}\coefficientMultiplierNonConservedMomentLeft_{\indexTime}\timeShiftOperator^{-2\indexTime - 1}, \qquad \text{where}\quad \coefficientMultiplierNonConservedMomentLeft_{\indexTime}= - \frac{1}{(\courantNumber+1)\relaxationParameter - 2}\sum_{k=0}^{\indexTime}(\relaxationParameter-1)^{2(\indexTime-k)}\coefficientNumeratorLeft_{k},
\end{equation*}
with furthermore
\begin{equation}\label{eq:recurrenceEigvleft}
    \coefficientMultiplierNonConservedMomentLeft_{\indexTime + 1} = (\relaxationParameter-1)^2\coefficientMultiplierNonConservedMomentLeft_{\indexTime} - \frac{1}{(\courantNumber+1)\relaxationParameter - 2}\coefficientNumeratorLeft_{\indexTime + 1}.
\end{equation}

\end{itemize}

\paragraph{Scalar approach}\label{sec:onceAgainChoiceEigenvalue}

The second approach we present, called ``\strong{scalar}'', changes the expression (normalization) of the eigenvector of the system according to $\conservedMoment$ and $\nonConservedMoment$.
Therefore, for the mere sake of enforcing transparent boundary condition, these two unknowns are treated as totally independent ones.
This approach is easier for more involved schemes (e.g., the fourth-order \lbmScheme{1}{3} scheme) and has---on 1D problems---a similar storage cost compared to the previous one.

\begin{itemize}

\item \strong{Right boundary}.
The expression for $\conservedMoment$ is again given by \eqref{eq:rightBCU}.
For $\nonConservedMoment$, we consider
\begin{equation}\label{eq:tmp3}
    \begin{pmatrix}
        \zTransformed{\conservedMoment}_{\indexSpace}(\timeShiftOperator)\\
        \zTransformed{\nonConservedMoment}_{\indexSpace}(\timeShiftOperator)
    \end{pmatrix} = 
    \tilde{\coefficientStable}(\timeShiftOperator)
    \begin{pmatrix}
        \tilde{\eigenvectorLetter}_{\stableMarker}(\timeShiftOperator)\\
        1
    \end{pmatrix}
    \stableRoot(\timeShiftOperator)^{\indexSpace}, \qquad \indexSpace\geq\numberGridPoints + 1,
\end{equation}
where $\transpose{(\tilde{\eigenvectorLetter}_{\stableMarker}(\timeShiftOperator), 1)}\in\kernel(\timeShiftOperator\identityMatrix{2} - \schemeMatrixBulkFourier(\stableRoot(\timeShiftOperator)))$, and boils down to changing the normalization of the eigenvector while treating $\nonConservedMoment$.
This analogously yields \eqref{eq:rightBCU} where $\conservedMoment$ is replaced by $\nonConservedMoment$ on both sides of the equation:
\begin{equation}\label{eq:rightBCVScalar}
    \nonConservedMoment_{\numberGridPoints + 1}^{\indexTime} = \sum_{k = 0}^{\lfloor (\indexTime - 1)/ 2 \rfloor} \coefficientsLaurentStableRoot_{k} \nonConservedMoment_{\numberGridPoints}^{\indexTime-2k - 1}.
\end{equation}
With this approach, we need to store the ``history'' of $\nonConservedMoment_{\numberGridPoints}$ instead of the weights $\coefficientMultiplierNonConservedMomentRight_{\indexTime}$ featured by \eqref{eq:bcRightNonConserved}.

We come back to the discussion that we hinted to at the beginning of \Cref{sec:systemicRight} about the arbitrariness of using $\conservedMoment_{\numberGridPoints}$ as a boundary datum to the left.
If we use $\nonConservedMoment_{\numberGridPoints}$ instead, \eqref{eq:generalSolutionRightWithoutFreeCoefficient} reads 
\begin{equation*}
    \begin{pmatrix}
        \zTransformed{\conservedMoment}_{\indexSpace}(\timeShiftOperator)\\
        \zTransformed{\nonConservedMoment}_{\indexSpace}(\timeShiftOperator)
    \end{pmatrix} = 
    \begin{pmatrix}
        1/\eigenvectorLetter_{\stableMarker}(\timeShiftOperator)\\
        1
    \end{pmatrix}
    \stableRoot(\timeShiftOperator)^{\indexSpace-\numberGridPoints} \zTransformed{\nonConservedMoment}_{\numberGridPoints}(\timeShiftOperator), \qquad \indexSpace\geq\numberGridPoints + 1,
\end{equation*}
whereas the expression corresponding to the normalization of the eigenvector in \eqref{eq:tmp3} is 
\begin{equation*}
    \begin{pmatrix}
        \zTransformed{\conservedMoment}_{\indexSpace}(\timeShiftOperator)\\
        \zTransformed{\nonConservedMoment}_{\indexSpace}(\timeShiftOperator)
    \end{pmatrix} = 
    \begin{pmatrix}
        \tilde{\eigenvectorLetter}_{\stableMarker}(\timeShiftOperator)\\
        1
    \end{pmatrix}
    \stableRoot(\timeShiftOperator)^{\indexSpace-\numberGridPoints} \zTransformed{\nonConservedMoment}_{\numberGridPoints}(\timeShiftOperator), \qquad \indexSpace\geq\numberGridPoints + 1.
\end{equation*}
As a consequence, ${\eigenvectorLetter}_{\stableMarker}(\timeShiftOperator)\tilde{\eigenvectorLetter}_{\stableMarker}(\timeShiftOperator)\equiv 1$, and we could have readily devised the scalar approach while treating the systemic one by just switching the moment used as boundary datum from $\conservedMoment$ to $\nonConservedMoment$ when finding the boundary condition for $\nonConservedMoment$.
\begin{remark}[A second comparison with \cite{heubes2014exact}]
    When initial data are compactly supported, \cite[Equation (15)]{heubes2014exact} reads for $\indexTime\geq 1$
    \begin{equation*}
        \distributionFunctionLetter_{-, \numberGridPoints}^{\indexTime + 1} = \sum_{k = 1}^{\lfloor (\indexTime + 1)/2\rfloor} (A_+(k) \distributionFunctionLetter_{+, \numberGridPoints}^{\indexTime-2k+1} + A_-(k) \distributionFunctionLetter_{-, \numberGridPoints}^{\indexTime-2k+1})
    \end{equation*}
    with our way of indexing discrete velocities, where $A_{\pm}(k)$ are provided in the paper by Heubes \emph{et al}.
    Rewriting \eqref{eq:rightBCU} and \eqref{eq:rightBCVScalar} on the distribution functions, and advancing of one time-step applying relaxation and transport, where the latter step moves $\distributionFunctionLetter_{-, \numberGridPoints+1}^{\indexTime, \collided}$ to $\distributionFunctionLetter_{-, \numberGridPoints}^{\indexTime+1}$, we have
    \begin{equation*}
        \distributionFunctionLetter_{-, \numberGridPoints}^{\indexTime + 1} = \sum_{k = 1}^{\lfloor (\indexTime + 1)/2\rfloor} (\tfrac{1}{2}\relaxationParameter (1-\courantNumber)\coefficientsLaurentStableRoot_{k-1} \distributionFunctionLetter_{+, \numberGridPoints}^{\indexTime-2k+1} +  (1-\tfrac{1}{2}\relaxationParameter(1+\courantNumber))\coefficientsLaurentStableRoot_{k-1} \distributionFunctionLetter_{-, \numberGridPoints}^{\indexTime-2k+1}).
    \end{equation*} 
    One can show that $\tfrac{1}{2}\relaxationParameter (1-\courantNumber)\coefficientsLaurentStableRoot_{k-1} = A_+(k)$ and $(1-\tfrac{1}{2}\relaxationParameter(1+\courantNumber))\coefficientsLaurentStableRoot_{k-1}=A_-(k)$, making our and their boundary conditions the same for compactly supported data.
\end{remark}

\item \strong{Left boundary}.
For the conserved moment $\conservedMoment$, we still use \eqref{eq:bcConservedLeft}.
For the non-conserved moment $\nonConservedMoment$, the equation becomes 
\begin{equation}
    \nonConservedMoment_0^{\indexTime} = \frac{1}{\productRoots} \sum_{k = 0}^{\lfloor (\indexTime - 1)/ 2 \rfloor} \coefficientsLaurentStableRoot_{k} \nonConservedMoment_{1}^{\indexTime-2k - 1}.
\end{equation}

\end{itemize}

We can now comment on the slight difference of storage cost between the systemic and scalar approach.
Let $\numberTimeIterations\gg 1$ be the number of time iteration to reach the final time of the simulation.
For the systemic approach, we have to store $\sim \numberTimeIterations$ floating-point numbers for $\conservedMoment_{\numberGridPoints}^{\indexTime}$, $\sim \numberTimeIterations/2$ for $\coefficientsLaurentStableRoot_{\indexTime}$, and $\sim \numberTimeIterations/2$ for $\coefficientMultiplierNonConservedMomentRight_{\indexTime}$---for the right boundary; and $\sim \numberTimeIterations$ floating-point numbers for $\conservedMoment_{1}^{\indexTime}$ and $\sim \numberTimeIterations/2$ for $\coefficientMultiplierNonConservedMomentLeft_{\indexTime}$---for the right boundary.
Overall, the cost is $\sim \tfrac{7}{2}\numberTimeIterations$.
For the scalar approach, we have to store $\sim \numberTimeIterations$ floating-point numbers for $\conservedMoment_{\numberGridPoints}^{\indexTime}$ and $\nonConservedMoment_{\numberGridPoints}^{\indexTime}$, and $\sim \numberTimeIterations/2$ for $\coefficientsLaurentStableRoot_{\indexTime}$---for the right boundary; $\sim \numberTimeIterations$ floating-point numbers for $\conservedMoment_{1}^{\indexTime}$ and $\nonConservedMoment_{1}^{\indexTime}$---for the right boundary.
Overall, the cost is $\sim \tfrac{9}{2}\numberTimeIterations$.

\begin{itemize}

\item \strong{Recurrent computation of $\coefficientsLaurentStableRoot_{\indexTime}$}.
The coefficients $\coefficientsLaurentStableRoot_{\indexTime}$, the only ones needed in the scalar approach, can also be computed inductively as presented by Besse \strong{at al}.\footnote{This could also be done with the systemic approach, yet be quite cumbersome.
Indeed, the coefficients $\coefficientsLaurentStableRoot_{\indexTime}$ would be used to deduce and expansion for $(\timeShiftOperator^2 + ((\courantNumber^2-1)\relaxationParameter^2 + 2\relaxationParameter-2) + (\relaxationParameter-1)^2\timeShiftOperator^{-2})^{1/2}$, which enters into \eqref{eq:fraction}.} 
Knowing that $\lim_{\timeShiftOperator\to \infty}\stableRoot(\timeShiftOperator) = 0$, we have $\stableRoot(\timeShiftOperator) = \sum_{\indexTime = 1}^{+\infty}\coefficientsLaurentStableRootAllParities_{\indexTime}\timeShiftOperator^{-\indexTime}$.
Into the characteristic equation \eqref{eq:charEquation} and developing the Cauchy product to compute the square term, we obtain 
\begin{equation}\label{eq:tmp6}
    2\coefficientsLaurentStableRootAllParities_1 + ((1-\courantNumber)\relaxationParameter-2) + 2\coefficientsLaurentStableRootAllParities_2 \timeShiftOperator^{-1} + \sum_{\indexTime = 2}^{+\infty}\Bigl (
    ((1+\courantNumber)\relaxationParameter-2)\sum_{k = 1}^{\indexTime - 1} \coefficientsLaurentStableRootAllParities_{k}\coefficientsLaurentStableRootAllParities_{\indexTime - k} + 2 \coefficientsLaurentStableRootAllParities_{\indexTime + 1}+ 2(1-\relaxationParameter)\coefficientsLaurentStableRootAllParities_{\indexTime - 1}\Bigr )  \timeShiftOperator^{-\indexTime} = 0.
\end{equation}
Equating power-by-power, we have that $\coefficientsLaurentStableRootAllParities_1 = 1 + \tfrac{1}{2}(\courantNumber-1)\relaxationParameter$ and $\coefficientsLaurentStableRootAllParities_2 = 0$, along with 
\begin{equation*}
    \coefficientsLaurentStableRootAllParities_{\indexTime} = (\relaxationParameter-1)\coefficientsLaurentStableRootAllParities_{\indexTime - 2} + \bigl ( 1 - \tfrac{1}{2}(1+\courantNumber)\relaxationParameter\bigr )\sum_{k = 1}^{\indexTime - 2} \coefficientsLaurentStableRootAllParities_{k}\coefficientsLaurentStableRootAllParities_{\indexTime - 1 - k}, \qquad \indexTime \geq 3.
\end{equation*}
One can easily show  that $\coefficientsLaurentStableRootAllParities_{2\indexTime} = 0$ for $\indexTime\geq 1$.
Therefore
\begin{multline}\label{eq:recurrentDefinitionSD1Q2Scalar}
    \stableRoot(\timeShiftOperator) = \sum_{\indexTime = 0}^{+\infty}\coefficientsLaurentStableRoot_{\indexTime}\timeShiftOperator^{-2\indexTime-1}, \qquad \text{with}\qquad \coefficientsLaurentStableRoot_0 = 1 + \tfrac{1}{2}(\courantNumber-1)\relaxationParameter \\
    \text{and}\quad 
    \coefficientsLaurentStableRoot_{\indexTime} =  (\relaxationParameter-1)\coefficientsLaurentStableRoot_{\indexTime - 1} + \bigl ( 1 - \tfrac{1}{2}(1+\courantNumber)\relaxationParameter\bigr )\sum_{k = 0}^{\indexTime - 1} \coefficientsLaurentStableRoot_{k}\coefficientsLaurentStableRoot_{\indexTime - 1 - k}, \quad \indexTime\geq 1.
\end{multline}
The previous recurrence simplifies when $\relaxationParameter = \tfrac{2}{1+\courantNumber}$ (the convolution term in the recurrence disappears), where we obtain the expression from \eqref{eq:rootNoSecondOrder}.
The significant advantage of this approach is that it works regardless of the values of the parameter $\relaxationParameter$ and possible links that we could not be able to establish with Legendre polynomials or well-known generating functions.
\end{itemize}

\begin{summarybox}{Summary on the construction of transparent boundary conditions in 1D}
    \begin{small}
        For a given boundary:
        \begin{enumerate}
            \item Consider the $\timeShiftOperator$-transformed scheme and the spatial root(s) of the characteristic equation (e.g., \eqref{eq:charEquation}) associated with stable geometric solutions for the half-space problem extending outward from the boundary (e.g., $\stableRoot(\timeShiftOperator)$ for the right boundary, $\unstableRoot(\timeShiftOperator)$ for the left boundary).
            \item Then, two approaches are possible.
            \begin{itemize}
                \item \strong{Systemic approach}: construct the solution for the half-space problem through an eigenvector with fixed normalization regardless of the component of interest, e.g. \eqref{eq:solD1Q2Left}.
                \item \strong{Scalar approach}: construct the solution for the half-space problem allowing convenient changes in the normalization of the eigenvector according to the component of interest, e.g. first use \eqref{eq:solD1Q2Left} and then \eqref{eq:tmp3}.
            \end{itemize}
            \item Use the previous solution to construct ghost values in the $\timeShiftOperator$-transformed space, e.g. \eqref{eq:tmp2}.
            \item Invert the $\timeShiftOperator$-transform to deduce actual boundary conditions on ghost cells.
            This can be done by
            \begin{itemize}
                \item \strong{Explicitly} solving the characteristic equation in $\fourierShift(\timeShiftOperator)$ (e.g., \eqref{eq:paramStableRoot}) and use well-known results when available (e.g., \eqref{eq:tmp4} and \eqref{eq:tmp5}).
                \item \strong{Implicitly} solving the characteristic equation by inserting the Laurent series of $\fourierShift(\timeShiftOperator)$ and formally equate each power, e.g. \eqref{eq:tmp6}.
            \end{itemize}
        \end{enumerate}
    \end{small}
\end{summarybox}

\paragraph{Back to asymptotic expansions using analytic combinatorics}\label{sec:backAsympt}

The question is: why do we insist on asymptotics?
The answer is that they help us decide when truncating the boundary conditions to obtain cheaper approximated ones.
In the previous discussion, asymptotics were obtained by known results on Legendre polynomials.
We now wish to provide a different proof of some of them, which can be deployed in cases where a link with Legendre polynomials in not established.

Recall from \Cref{sec:notations} that the coefficients $\coefficientsLaurentStableRootAllParities_{\indexTime}$ in the Laurent expansion of $\stableRoot(\timeShiftOperator)$ at $\timeShiftOperator=\infty$ form a sequence whose generating function is $\stableRoot(\timeShiftOperator^{-1})$.
Asymptotics for sequences with given generating functions are well-known in \strong{analytic combinatorics}, see for instance \cite{flajolet2009analytic} and \cite[Chapter 5]{wilf2005generatingfunctionology}.
The determination of such asymptotics boils down to a local analysis of the function $\stableRoot(\timeShiftOperator^{-1})$ at its singularities.
Each singularity $\singularity\in\complex$ gives a contribution to the asymptotics of $\coefficientsLaurentStableRootAllParities_{\indexTime}$ of the form 
\begin{equation*}
    \singularity^{-\indexTime} \tau(\indexTime)
\end{equation*}
where $\tau(\indexTime)$ is a sub-exponential tame factor depending on the nature of the singularity (pole, branch point, \emph{etc.}).
When all singularities lay on $\unitCircle$, one can also employ \cite[Theorem 8.4]{szeg1939orthogonal}.

\begin{equation}\label{eq:stableRootGenerating}
    \stableRoot(\timeShiftOperator^{-1}) = 
    \begin{cases}
        \frac{2\courantNumber\timeShiftOperator}{(\courantNumber+1)+(\courantNumber-1)\timeShiftOperator^2}, \qquad&\text{if}\quad  \relaxationParameter= \tfrac{2}{1+\courantNumber}, \\
        \frac{(\relaxationParameter-1) \timeShiftOperator^{2} - 1 + \sqrt{1 + ((\courantNumber^2-1)\relaxationParameter^2 + 2\relaxationParameter-2)\timeShiftOperator^{2} + (\relaxationParameter-1)^2\timeShiftOperator^{4}}}{((\courantNumber + 1)\relaxationParameter-2)\timeShiftOperator}, \qquad&\text{if}\quad  \relaxationParameter\neq \tfrac{2}{1+\courantNumber}.
    \end{cases}
\end{equation}

\begin{itemize}
    \item When  $\relaxationParameter= \tfrac{2}{1+\courantNumber}$, we observe two singularities $\pm({\frac{1+\courantNumber}{1-\courantNumber}})^{1/2}\in\reals$, first-order poles.
    This indicates \cite[Figure VI.5]{flajolet2009analytic} that tame factors $\tau(\indexTime) = \bigO{1}$ for large $\indexTime$.
    \item When $\relaxationParameter\neq \tfrac{2}{1+\courantNumber}$, we have four (if $\relaxationParameter \neq 1$) or two (if $\relaxationParameter = 1$) singularities under the form of branch points for the square root.  This indicates that the tame factors behave, at leading order, as $\indexTime^{-3/2}$.
    If we show that the branch points belong to $\closedNeighborhoodInfinity$, then this proves that $\coefficientsLaurentStableRootAllParities_{\indexTime}$ is bounded, with periodic structures if they belong to $\unitCircle$ and geometrically damped behavior if they are in $\neighborhoodInfinity$.
\end{itemize}
These facts are coherent with the asymptotics in \Cref{sec:systemicRight} and made precise by the following result, whose proof is in \Cref{app:rootsRadicand} and based on the results of \cite[Chapter 4]{strikwerda2004finite}.
\begin{lemma}[Branch points in $\stableRoot$]\label{lemma:rootsRadicand}
    Assume the stability conditions \eqref{eq:stabConditionD1Q2} be met and $\courantNumber\neq 0$.
    Then, the roots of $1 + ((\courantNumber^2-1)\relaxationParameter^2 + 2\relaxationParameter-2)\timeShiftOperator^{2} + (\relaxationParameter-1)^2\timeShiftOperator^{4}= 0$, whose left-hand side appears in \eqref{eq:stableRootGenerating}, are in $\neighborhoodInfinity$ for $0<\relaxationParameter<2$ and on $\unitCircle$ for $\relaxationParameter = 2$.
\end{lemma}

We now illustrate computations based on a local analysis of the singularities of $\stableRoot(\timeShiftOperator^{-1})$ in the case $\relaxationParameter = 1$ and $\relaxationParameter = 2$.

\begin{itemize}

\item \strong{Relaxation scheme: $\relaxationParameter = 1$}.
We have that $\singularity = \pm (1-\courantNumber^2)^{-1/2}\in\neighborhoodInfinity$ are the singular points, and
\begin{equation*}
    \stableRoot(\timeShiftOperator^{-1}) =
    \frac{\sqrt{(1 - \frac{\timeShiftOperator}{(1-\courantNumber^2)^{-1/2}})(1 - \frac{\timeShiftOperator}{-(1-\courantNumber^2)^{-1/2}}) }-1}{(\courantNumber-1)\timeShiftOperator}.
\end{equation*}
We shall see that to analyze the behavior of $\coefficientsLaurentStableRootAllParities_{\indexTime}$ for large $\indexTime$, all terms except the square root do not contribute.
In the vicinity of $\singularity = \pm (1-\courantNumber^2)^{-1/2}\in\neighborhoodInfinity$, we have the Puiseux expansion
\begin{equation*}
    \stableRoot(\timeShiftOperator^{-1}) = \pm \frac{(1-\courantNumber^2)^{1/2}}{\courantNumber-1}\bigl ( -1 + \sqrt{2} (1-\timeShiftOperator/\singularity)^{1/2} - (1-\timeShiftOperator/\singularity) + \bigO{(1-\timeShiftOperator/\singularity)^{3/2}}\bigr ).
\end{equation*}
The terms with integer powers of $1-\timeShiftOperator/\singularity$ indeed come from a regular part and hence give no contribution to the asymptotics, \confer{} \cite[Figure VI.5]{flajolet2009analytic}.
By this and Theorem VI.5 in the same reference:
\begin{align*}
    \coefficientsLaurentStableRootAllParities_{\indexTime} = &(1-\courantNumber^2)^{\indexTime/2}\Bigl [\frac{(1-\courantNumber^2)^{1/2}}{\courantNumber-1}\Bigl ( 0 + \sqrt{2} \Bigl ( -\frac{1}{2\sqrt{\pi}}\indexTime^{-3/2} + \bigO{\indexTime^{-5/2}}\Bigr ) + 0\Bigr ) \Bigr ]\\
    +(-1)^{\indexTime}&(1-\courantNumber^2)^{\indexTime/2}\underbrace{\Bigl [-\frac{(1-\courantNumber^2)^{1/2}}{\courantNumber-1}\Bigl ( 0 + \sqrt{2} \Bigl ( -\frac{1}{2\sqrt{\pi}}\indexTime^{-3/2} + \bigO{\indexTime^{-5/2}}\Bigr ) + 0\Bigr )\Bigr ]}_{\text{sub-exponential tame}} + \bigO{(1-\courantNumber^2)^{\indexTime/2}\indexTime^{-5/2}}.
\end{align*}
This yields 
\begin{equation*}
    \coefficientsLaurentStableRootAllParities_{\indexTime} = \frac{\sqrt{2}(1-\courantNumber^2)^{1/2}}{(1-\courantNumber)\sqrt{\pi}} (1-\courantNumber^2)^{\indexTime/2}\indexTime^{-3/2} \frac{1 - (-1)^{\indexTime}}{2} +\bigO{(1-\courantNumber^2)^{\indexTime/2}\indexTime^{-5/2}}
\end{equation*}
and gives an idea why, as previously shown, $\coefficientsLaurentStableRootAllParities_{\indexTime}=0$ when $\indexTime$ is even.
For $\indexTime = 2k + 1$, we have 
\begin{equation*}
    \coefficientsLaurentStableRootAllParities_{2k + 1} = \frac{1+\courantNumber}{2\sqrt{\pi}}(1-\courantNumber^2)^{k} k^{-3/2}  +\bigO{(1-\courantNumber^2)^{k}k^{-5/2}}, \qquad \text{which is \eqref{eq:asymptoticRelaxationScheme}.}
\end{equation*}

\item \strong{Leap-frog scheme: $\relaxationParameter = 2$}.
We solve $\singularity^{4} + 2(2\courantNumber^2-1)\singularity^2 + 1=0$ in $\singularity^2$, obtaining 
\begin{equation*}
    \singularity^2 = 1-2\courantNumber^2\pm 2 i \courantNumber\sqrt{1-\courantNumber^2} =
    e^{\pm i\arccos(1-2\courantNumber^2)}.
\end{equation*}
Solving this equation in $\singularity$ and setting $\vartheta\definitionEquality\arccos(1-2\courantNumber^2)$, the four singular points read 
\begin{equation*}
    \singularity_1 = e^{\frac{i}{2}\vartheta}, \quad \singularity_2 = e^{\frac{i}{2}\vartheta+ i\pi}, \qquad \text{and}\qquad 
    \singularity_3 = e^{-\frac{i}{2}\vartheta}, \quad \singularity_4 = e^{-\frac{i}{2}\vartheta+ i\pi}.
\end{equation*}
We have 
\begin{align*}
    &\prod_{\substack{k=1\\k\neq 1}}^{4}(\singularity_1 - \singularity_k) = 4 i e^{i\vartheta/2} \sin (\vartheta)=4 e^{i/2(\vartheta + \pi)} \sin (\vartheta), \qquad \prod_{\substack{k=1\\k\neq 2}}^{4}(\singularity_2 - \singularity_k) = - 4 i e^{i\vartheta/2} \sin (\vartheta) = 4 e^{i/2(\vartheta - \pi)} \sin (\vartheta), \\
    &\prod_{\substack{k=1\\k\neq 3}}^{4}(\singularity_3 - \singularity_k) = -4 i e^{-i\vartheta/2} \sin (\vartheta)= 4 e^{i/2(-\vartheta-\pi)}\sin (\vartheta), \qquad \prod_{\substack{k=1\\k\neq 4}}^{4}(\singularity_4 - \singularity_k) = 4 i e^{-i\vartheta/2} \sin (\vartheta)= 4 e^{i/2(-\vartheta+\pi)}\sin (\vartheta).
\end{align*}
Here, the way we have decided to write $\pm i$ in polar form on the right-hand sides is dictated by the fact that we shall take the square root of these quantities, whose chosen determination must be used consistently.
Moreover, $\sin(\vartheta) = 2\courantNumber\sqrt{1-\courantNumber^2}$.
Next, we feature
\begin{equation*}
    \stableRoot(\timeShiftOperator^{-1}) =\frac{\timeShiftOperator^{2} - 1}{2\courantNumber\timeShiftOperator} + \underbrace{\frac{\sqrt{1 + 2(2\courantNumber^2-1)\timeShiftOperator^{2} + \timeShiftOperator^{4}}}{2\courantNumber\timeShiftOperator}}_{=\psi(\timeShiftOperator)},\qquad \text{where} \qquad \psi(\timeShiftOperator) = \frac{\Bigl ( \prod\limits_{k = 1}^{4}(\timeShiftOperator - \singularity_k) \Bigr )^{1/2}}{2\courantNumber\timeShiftOperator}.
\end{equation*}
We limit the study to $\psi$, as the remainder of $\stableRoot(\timeShiftOperator^{-1})$ is regular.
In the neighborhood of $\singularity = \pm e^{i\vartheta/2}$
\begin{equation*}
    \psi(\timeShiftOperator)= \pm\frac{\sqrt{2}(1-\courantNumber^2)^{1/4}}{\sqrt{\courantNumber} }e^{-i\pi/4}  \sqrt{1-\frac{\timeShiftOperator}{\singularity} } + \bigO{(1-\timeShiftOperator/\singularity)^{3/2}}.
\end{equation*}
For $\timeShiftOperator\to \singularity = \pm e^{-i\vartheta/2}$
\begin{equation*}
    \psi(\timeShiftOperator) =\pm \frac{\sqrt{2}  (1-\courantNumber^2)^{1/4}}{\sqrt{\courantNumber}} e^{i\pi /4} \sqrt{1-\frac{\timeShiftOperator}{\singularity} } + \bigO{(1-\timeShiftOperator/\singularity)^{3/2}}.
\end{equation*}
By the results in \cite{flajolet2009analytic}, we obtain
\begin{align*}
    \coefficientsLaurentStableRootAllParities_{\indexTime} &= \frac{\sqrt{2}  (1-\courantNumber^2)^{1/4}}{\sqrt{\courantNumber}} \overbrace{\Bigl (-\frac{1}{2\sqrt{\pi}}\indexTime^{-3/2} + \bigO{\indexTime^{-5/2}}\Bigr )}^{\text{sub-exponential tames}} \Bigl ( e^{i\pi /4} (\singularity_3^{-\indexTime}-\singularity_4^{-\indexTime}) + e^{-i\pi /4} (\singularity_1^{-\indexTime}-\singularity_2^{-\indexTime}) \Bigr ) + \bigO{\indexTime^{-5/2}}\\
    &=-\frac{2\sqrt{2}  (1-\courantNumber^2)^{1/4}}{\sqrt{\pi\courantNumber}} \indexTime^{-3/2}\frac{1-(-1)^{\indexTime}}{2} \cos( \tfrac{1}{2} \indexTime\vartheta +\pi/4)  + \bigO{\indexTime^{-5/2}}
\end{align*}
for large $\indexTime$.
Writing $\cos( \tfrac{1}{2} \indexTime\vartheta +\pi/4) = \cos( \tfrac{1}{2} \indexTime\vartheta - \pi/4 + \pi/2) = - \sin( \tfrac{1}{2} \indexTime\vartheta - \pi/4)$, we have
\begin{equation*}
    \coefficientsLaurentStableRootAllParities_{2k +1} = \frac{(1-\courantNumber^2)^{1/4}}{\sqrt{\courantNumber}\sqrt{\pi}}  \sin( (k+\tfrac{1}{2})\vartheta -\tfrac{\pi}{4}) k^{-3/2} + \bigO{k^{-5/2}},
\end{equation*}
which is nothing but the estimate for $\coefficientsLaurentStableRoot_{k}$ (equal to $\coefficientsLaurentStableRootAllParities_{2k + 1}$) that we deduce from \eqref{eq:expansionBCloseTwo} together with \eqref{eq:definitionSAsProduct}.

\end{itemize}

\paragraph{Numerical experiments}\label{sec:numExpD1Q2}

\begin{figure}
    \begin{center}
        \includegraphics[width=1\textwidth]{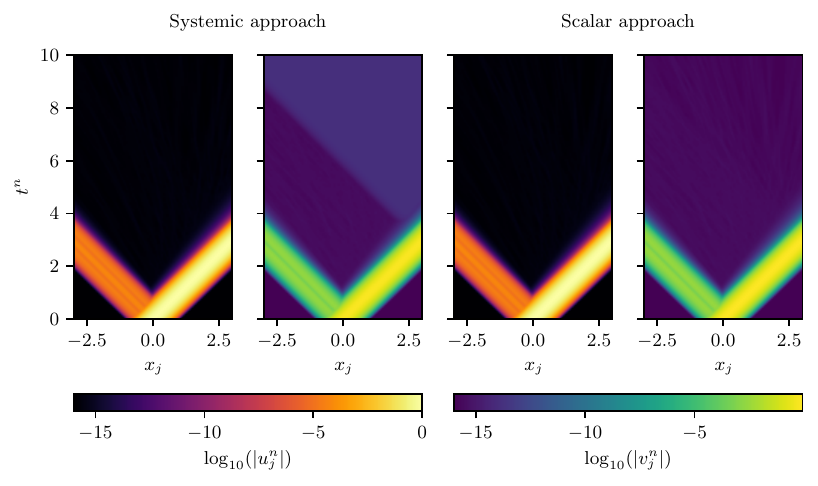}
    \end{center}\caption{\label{fig:D1Q2-sys-vs-scal}Solutions for the \lbmScheme{1}{2} using $\relaxationParameter = 2$, comparing the systemic approach and the scalar approach.}
\end{figure}
\begin{itemize}

\item \strong{Need for accuracy in the computation of the coefficients}\label{sec:accuracyCoeffs}.
We simulate using $\advectionVelocity = 1$, $\latticeVelocity = \tfrac{6}{5}$ and $\numberGridPoints = 1000$ on a domain delimited by $\leftBoundary = -3$ and $\rightBoundary = 3$.
The initial datum is the smooth profile $\conservedMoment^{\circ}(\spaceVariable) = \cos^{10}(\pi\spaceVariable/2)\indicatorFunction{(-1, 1)}(\spaceVariable)$, with $\conservedMoment^{\circ}\in C^9(\reals)$.
We compare the systemic approach where the weights are computed in double precision by \eqref{eq:recurrenceS}, \eqref{eq:recurrecenceRight}--\eqref{eq:recurrenceEigv}, and \eqref{eq:recurrenceLeft}--\eqref{eq:recurrenceEigvleft}, with the scalar approach, where only the computation through \eqref{eq:recurrenceS} is needed.
Results are shown in \Cref{fig:D1Q2-sys-vs-scal} for $\relaxationParameter = 2$, as both approaches yield the same ones for $\relaxationParameter<2$ (simulations not shown here).

\begin{figure}
    \begin{center}
        \includegraphics[width=0.5\textwidth]{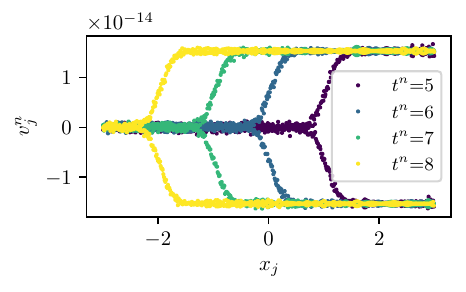}
    \end{center}\caption{\label{fig:D1Q2-progressive-systemic}Non-conserved moment for the \lbmScheme{1}{2} using $\relaxationParameter = 2$ with a systemic approach.}
\end{figure}

We remark---first and foremost---that both approaches ensure that waves, be them physical right-going or spurious left-going ones, exit the domain.
There is a small exception to this, which appears in the case where we employ the \strong{systemic approach}, and concerns the non-conserved moment $\nonConservedMoment$.
By zooming on this, \confer{} \Cref{fig:D1Q2-progressive-systemic}, we see that this is of amplitude $10^{-14}$, so very close to the machine epsilon in double precision.
The pattern reveals a mode that is quite notorious: $(\timeShiftOperator, \fourierShift) = (1, -1)$, the space-checkerboard mode, such that $\unstableRoot(1) = -1$.
This mode propagates to the left at group velocity equal to $-\advectionVelocity$.
This phenomenon can be explained as follows. Although we prescribed that the solution at the right boundary is only brought by the stable eigenvalue-eigenvector, \confer{} \eqref{eq:generalSolutionRightWithoutFreeCoefficient}, round-off errors, probably on the coefficients of the boundary condition,  generate some data with non-zero amplitude along the unstable eigenvector $\transpose{(1, \eigenvectorLetter_{\unstableMarker}(\timeShiftOperator))}$. 
This mode can---according to \Cref{lemma:continuousExtension} and when $\relaxationParameter = 2$ only---be  propagated inside the domain by the bulk scheme only on $\nonConservedMoment$, as $\eigenvectorLetter_{\unstableMarker}(\timeShiftOperator)$ features a pole exactly on $\timeShiftOperator = 1$.
From our recent work \cite{bellotti2025stability}, we know that the scheme can develop mild instabilities stemming from boundaries on $\nonConservedMoment$ when $\relaxationParameter = 2$, due to this fact.

\item \strong{Effectiveness of the transparent boundary conditions and comparison with other boundary conditions}.
\begin{figure}
    \begin{center}
        \includegraphics[width=1\textwidth]{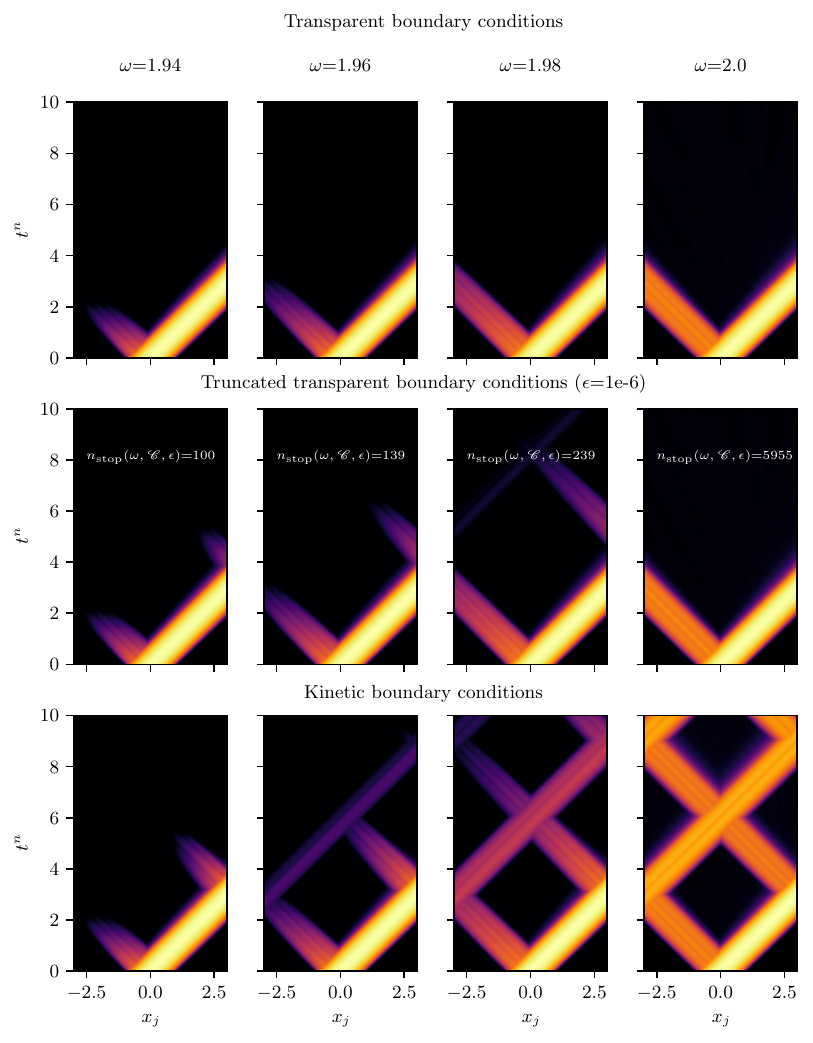}
    \end{center}\caption{\label{fig:D1Q2-comparison-BC}Values of $\log_{10}(|\conservedMoment_{\indexSpace}^{\indexTime}|)$ (same color-scale as \Cref{fig:D1Q2-sys-vs-scal}) for the \lbmScheme{1}{2} endowed with different boundary conditions.}
\end{figure}
From now on, we consider the scalar approach as the procedure of choice.
We compare three boundary conditions, namely the previously introduced transparent boundary conditions, and truncated transparent ones, as well as some kinetic boundary conditions, all described below.

Truncated transparent boundary conditions are devised as follows.
For $\relaxationParameter$ close to two, we have shown that 
\begin{multline*}
    \coefficientsLaurentStableRoot_{\indexTime} \sim \frac{\relaxationParameter-1}{(\courantNumber+1)\relaxationParameter - 2}
    \sqrt{\frac{2}{\pi}} (1-\argumentLegendrePolynomials^2)^{1/4} \sin\bigl ( \bigl ( \indexTime+\tfrac{1}{2}\bigr )\vartheta - \tfrac{\pi}{4}\bigr ) \indexTime^{-3/2} (\relaxationParameter-1)^{\indexTime}\\
    \leq\underbrace{\left | \frac{\relaxationParameter-1}{(\courantNumber+1)\relaxationParameter - 2}
    \sqrt{\frac{2}{\pi}} (1-\argumentLegendrePolynomials^2)^{1/4}  \right |}_{=C} \indexTime^{-3/2} (\relaxationParameter-1)^{\indexTime}.
\end{multline*}
Introducing a tolerance $0\leq \epsilon \ll 1$, we impose that the right-hand side in the previous inequality be equal to $\epsilon$, which commands
\begin{equation*}
    \indexTime = \Bigl (\frac{C}{\epsilon}\Bigr )^{2/3} e^{\tfrac{2}{3}\log(\relaxationParameter - 1)\indexTime}.
\end{equation*}
When $\relaxationParameter\neq 2$, this equation can be solved using the Lambert $W$ function, see \cite{corless1996lambert}, suggesting stopping at 
\begin{equation*}
    \indexTime_{\text{stop}}(\relaxationParameter, \courantNumber, \epsilon) \definitionEquality \Bigl \lceil -\frac{3}{2\log(\relaxationParameter - 1)} W_0\Bigl ( -\frac{2}{3}\log(\relaxationParameter - 1) \Bigl (\frac{C}{\epsilon}\Bigr )^{2/3} \Bigr ) \Bigr \rceil, 
\end{equation*}
where $W_0$ is the principal branch of the Lambert $W$ function.
The growth of $\epsilon\mapsto \indexTime_{\text{stop}}(\relaxationParameter, \courantNumber, \epsilon)$ as $\epsilon\to 0^+$ is very slow, as dictated by the very mild growth of $W_0$ at $+\infty$.
The truncated transparent boundary conditions, taking \eqref{eq:rightBCU} as example, hence read
\begin{equation}\label{eq:truncatedBoundaryCondition}
    \conservedMoment_{\numberGridPoints + 1}^{\indexTime} = \sum_{k = 0}^{\substack{\min(\lfloor (\indexTime - 1)/ 2 \rfloor, \\ \indexTime_{\text{stop}}(\relaxationParameter, \courantNumber, \epsilon))}} \coefficientsLaurentStableRoot_{k} \conservedMoment_{\numberGridPoints}^{\indexTime-2k - 1}.
\end{equation}
We observe that alternative approaches to reduce the cost of transparent boundary conditions, such as the ``sum of exponentials'' technique \cite{cms/1250880098}, could be explored in a future work.

Concerning the kinetic boundary conditions that we test, we employ a kinetic Dirichlet boundary condition at the inflow $\leftBoundary$, see \cite{aregba2025equilibrium, bellotti2025stability}, and a kinetic 1st-order extrapolation, see \cite{bellotti2025consistency, bellotti2025stability}, at the outflow $\rightBoundary$.
They are implemented by using the cells $\spaceGridPoint{0}$ and $\spaceGridPoint{\numberGridPoints + 1}$ as ghosts, where the values of the post-relaxation distribution functions are prepared as 
\begin{equation}\label{eq:kineticBoundaryConditions}
    \distributionFunctionLetter_{+, 0}^{\indexTime, \collided} = 0\qquad \text{and} \qquad \distributionFunctionLetter_{-, \numberGridPoints + 1}^{\indexTime, \collided} = \distributionFunctionLetter_{-, \numberGridPoints}^{\indexTime, \collided}.
\end{equation}

\begin{remark}[Computational cost \strong{versus} accuracy]\label{rem:costVsAccuracy}
    Notice that there is a general trade-off between \strong{computational cost}, which is maximal with our transparent boundary conditions owing to their global-in-time character, and \strong{accuracy} (i.e., absence of reflected waves), which is also maximal in this case.
    Consider that $\numberTimeIterations \propto T/\timeStep$ time-iterations are needed to reach the final time $T>0$.
    Then we can remark the following.
    \begin{itemize}
        \item For transparent boundary conditions \eqref{eq:rightBCU}, the cost of the $\indexTime$-th iteration is of $\bigO{\indexTime}$ sums--multiplications, thus $\bigO{\numberTimeIterations^2}$ to deal with boundary conditions for the whole simulation. If $\numberGridPoints\sim \numberTimeIterations$, this cost is commensurable to that of the inner scheme, which is $\bigO{\numberTimeIterations^2}$.
        This significant toll is the price to pay for perfectly non-reflective boundary conditions for any frequency of impinging waves.
        \item For truncated transparent boundary conditions \eqref{eq:truncatedBoundaryCondition}, the cost of the $\indexTime$-th iteration is of $\bigO{1}$, with a large constant if $\epsilon$ is small and/or coefficients do not decrease fast enough, due to large values of $\indexTime_{\text{stop}}(\relaxationParameter, \courantNumber, \epsilon)$. 
        The cost of boundary conditions is thus $\bigO{\numberTimeIterations}$, possibly with large constants.
        Still, this condition can be of high quality in cases where $\relaxationParameter$ is far apart from two, as coefficients quickly decrease, thus make $\indexTime_{\text{stop}}$ small while essentially retaining the features of perfectly transparent conditions. 
        \item Kinetic boundary conditions such as \eqref{eq:kineticBoundaryConditions} generally use only data at the same/previous iteration, thus are strongly localized in time. 
        Their cost for the whole simulation is therefore $\bigO{\numberTimeIterations}$ with small constants.
        However, their cheaper character comes at the price of large reflected waves.
        Indeed, these boundary conditions are essentially designed using Taylor expansions (see \Cref{sec:foo}), corresponding to non-reflective behaviors only for impinging data of small frequency and corresponding to ``physical'' waves.
    \end{itemize}
\end{remark}

Under the same conditions as \Cref{fig:D1Q2-sys-vs-scal}, the results are collected in \Cref{fig:D1Q2-comparison-BC}.
We point out the following facts.
\begin{itemize}
    \item Only transparent boundary conditions allow all waves to exit the domain. 
    \item For truncated transparent boundary conditions, we see that truncation entails a delay in the outbreak of (small) reflected waves. 
    This happens because these conditions are based on forgetting a remote past through a moving window of fixed size, and the forgotten past is thus more relevant later in time, as the support of the initial datum tardily reaches the boundary.
    \item When $\relaxationParameter<2$, if a damped spurious left-going wave manages to survive (i.e., have a significant amplitude) as $\relaxationParameter\approx 2$ up to reaching the left boundary, reflective boundary conditions may turn it into an undamped physical right-going wave, for instance $(1-\relaxationParameter, 1) \rightsquigarrow (1, 1)$. 
    \item On the other hand, undamped physical right-going waves reaching the right boundary may be turned into damped (except when $\relaxationParameter = 2$) spurious left-going waves, for instance $(1, 1)  \rightsquigarrow (1-\relaxationParameter, 1)$.
    This is not particularly dangerous as long as $\relaxationParameter$ is sufficiently far from two, so that reflected waves are rapidly damped, as they will not live for long enough to reach the other boundary.
    \item When $\relaxationParameter<2$, we see that spurious damped beams pinch before fading away.
    This can be attributed to anti-dissipation, see \Cref{rem:dissipationLowFrequencies}.
\end{itemize}

Tests are repeated using initial data with several packets at different frequencies, for instance 
\begin{equation}\label{eq:manypackets}
    \conservedMoment^{\circ}(\spaceVariable) = \sum_{k=0}^3 \underbrace{\cos^{10}(\tfrac{\pi}{2}(\spaceVariable +3-\tfrac{6}{5}(k+1)))\indicatorFunction{(-1, 1)}(\spaceVariable +3-\tfrac{6}{5}(k+1))}_{\text{envelope}} \underbrace{\cos(5\pi\times 2^k (\spaceVariable +3-\tfrac{6}{5}(k+1)))}_{\text{oscillation}}
\end{equation}
to validate the claim that the devised boundary conditions are transparent \strong{regardless of the frequency} of the incident packet. 
These results are shown in \Cref{app:moreSimulationsD1Q2}.

\item \strong{Transparent inflow boundary conditions}.
\begin{figure}
    \begin{center}
        \includegraphics[width=1\textwidth]{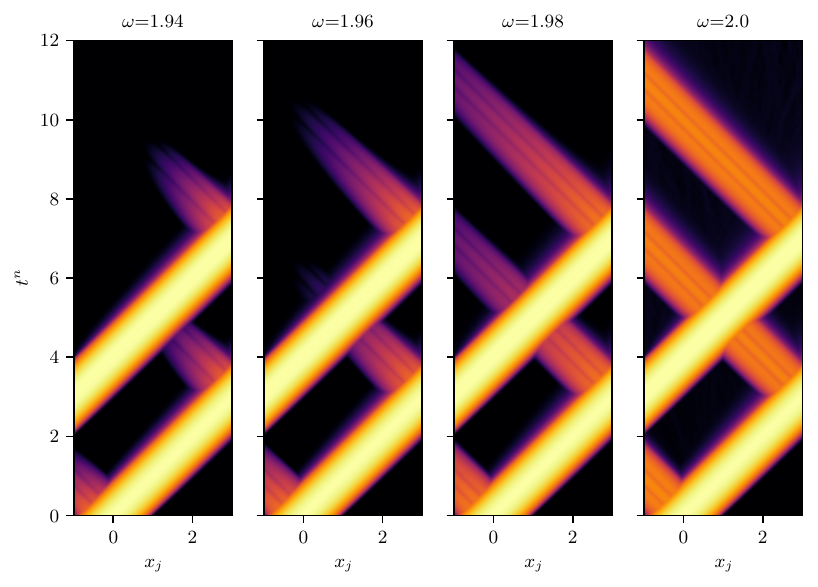}
    \end{center}\caption{\label{fig:D1Q2-inflow}Values of $\log_{10}(|\conservedMoment_{\indexSpace}^{\indexTime}|)$ (same color-scale as \Cref{fig:D1Q2-sys-vs-scal}) for the \lbmScheme{1}{2} endowed a transparent inflow boundary condition.}
\end{figure}
Consider $\advectionVelocity>0$ and a problem different from \eqref{eq:targetTransport1D}, namely
\begin{equation*}
    \begin{dcases}
        \begin{aligned}
            &\partial_{\timeVariable}\conservedMoment(\timeVariable, \spaceVariable) + \advectionVelocity\partial_{\spaceVariable}\conservedMoment(\timeVariable, \spaceVariable)  = 0, \qquad &\timeVariable>0, \quad &\spaceVariable\in (\leftBoundary, \rightBoundary),\\
            &\conservedMoment(0, \spaceVariable)  = \conservedMoment^{\circ}(\spaceVariable), \qquad & &\spaceVariable\in (\leftBoundary, \rightBoundary),\\
            &\conservedMoment(\timeVariable, \leftBoundary)  = \conservedMoment_{\text{in}}(\timeVariable), \qquad &\timeVariable>0. \quad &
        \end{aligned}
    \end{dcases}
\end{equation*}
Now the question is: can we build a numerical boundary condition at $\leftBoundary$ which enforces the inflow datum $\conservedMoment_{\text{in}}$ without reflecting waves?
It turns out that this is possible by a lift in \eqref{eq:solD1Q2Left}.
This boils down to having the right-hand side of \eqref{eq:solD1Q2Left} equal to $\transpose{(\zTransformed{\conservedMoment}_{\indexSpace}(\timeShiftOperator) - \zTransformed{\conservedMoment}_{\text{in}}(\timeShiftOperator), \zTransformed{\nonConservedMoment}(\timeShiftOperator))}$, where $\zTransformed{\conservedMoment}_{\text{in}}(\timeShiftOperator)$ is the $\timeShiftOperator$-transform of the time-discretized boundary datum (we shall use $ \conservedMoment_{\text{in}}(\timeGridPoint{\indexTime})$).
This entails
\begin{equation*}
    \zTransformed{\conservedMoment}_{0}(\timeShiftOperator) = \frac{\stableRoot(\timeShiftOperator)}{\productRoots} (\zTransformed{\conservedMoment}_{1}(\timeShiftOperator) \underbrace{- \zTransformed{\conservedMoment}_{\text{in}}(\timeShiftOperator)}_{\text{unlift}}) \underbrace{+ \zTransformed{\conservedMoment}_{\text{in}}(\timeShiftOperator)}_{\text{lift}}
    \qquad \text{and}\qquad 
    \zTransformed{\nonConservedMoment}_{0}(\timeShiftOperator) = \frac{\stableRoot(\timeShiftOperator)}{\productRoots} \zTransformed{\nonConservedMoment}_{1}(\timeShiftOperator).
\end{equation*}
Inverting the $\timeShiftOperator$-transform, we have 
\begin{equation}\label{eq:DirichletInflowD1Q2}
    \conservedMoment_0^{\indexTime} = \frac{1}{\productRoots} \sum_{k = 0}^{\lfloor (\indexTime - 1)/ 2 \rfloor} \coefficientsLaurentStableRoot_{k} (\conservedMoment_{1}^{\indexTime-2k - 1} - \conservedMoment_{\text{in}}(\timeGridPoint{\indexTime-2k - 1}) ) + \conservedMoment_{\text{in}}(\timeGridPoint{\indexTime}) \qquad 
    \text{and}\qquad 
    \nonConservedMoment_0^{\indexTime} = \frac{1}{\productRoots} \sum_{k = 0}^{\lfloor (\indexTime - 1)/ 2 \rfloor} \coefficientsLaurentStableRoot_{k} \nonConservedMoment_{1}^{\indexTime-2k - 1}.
\end{equation}

On the right boundary, we test with the condition given by \eqref{eq:kineticBoundaryConditions} on purpose, that is to generate reflected waves that shall eventually be incident to the left boundary, and show that \eqref{eq:DirichletInflowD1Q2} manages both to enforce the inflow condition and to prevent reflections.
We consider $\leftBoundary = -1$, $\rightBoundary = 3$, and $\numberGridPoints = 667$, except to what the framework stands as before.
Moreover, we take $\conservedMoment_{\text{in}}(\timeVariable) = \conservedMoment^{\circ}(3-\advectionVelocity\timeVariable)$.
The results in \Cref{fig:D1Q2-inflow} confirm the expected behavior.
The slightly dusty appearance for $\relaxationParameter = 2$ on one side of the beams, the one opposite to the direction of propagation, comes from numerical dispersion.
\end{itemize}

\subsubsection{Transparent boundary conditions for the fourth-order \lbmScheme{1}{3} scheme}\label{sec:BCFourthOrderD1Q3}

As for the \lbmScheme{1}{2} scheme, we present the case $\courantNumber>0$.
The main additional challenge with respect to the  \lbmScheme{1}{2} scheme lays in the computation of the coefficients for the left boundary condition.
Lemmata \ref{lemma:herschD1Q3Fourth} and \ref{lemma:eigenvectorsD1Q3Fourth} below are proved in \cite{bellotti2025stability}.

\paragraph{Spatial roots of the characteristic equation}

\begin{lemma}[Spatial roots of the characteristic equation of the fourth-order \lbmScheme{1}{3} scheme]\label{lemma:herschD1Q3Fourth}
    Let the stability condition \eqref{eq:stabilityConditionFourthOrder} be fulfilled.
    Then, \eqref{eq:charEquationD1Q3FourthOrder} has two roots, one stable $\stableRoot(\timeShiftOperator)\in\unitDisk$ for $\timeShiftOperator\in\neighborhoodInfinity$ and one unstable $\unstableRoot(\timeShiftOperator)$ for $\timeShiftOperator\in\neighborhoodInfinity$.
    Their product $\stableRoot(\timeShiftOperator)\unstableRoot(\timeShiftOperator)$ is given by 
    \begin{equation*}
        \productRoots(\timeShiftOperator) = 
        \frac{(2\courantNumber + 1)(\courantNumber-2)-(2\courantNumber-1)(\courantNumber+2)\timeShiftOperator}{(2\courantNumber - 1)(\courantNumber+2)-(2\courantNumber+1)(\courantNumber-2)\timeShiftOperator}.
    \end{equation*}
    As $\courantNumber>0$, the continuous extension of $\stableRoot$ and $\unstableRoot$ to the unit circle $\unitCircle$ fulfills $\stableRoot(\pm 1) = 1$ and $\unstableRoot(\pm 1) = \mp 1$.
\end{lemma}

\paragraph{Scalar approach}

We adopt a scalar approach with recurrent computation of the coefficients.
Although we do not pursue a systemic approach, we stress the fact that we know that the eigenvectors $\transpose{(1, \eigenvectorLetter_{\stableMarker, \nonConservedMoment}(\timeShiftOperator), \eigenvectorLetter_{\stableMarker, \nonNonConservedMoment}(\timeShiftOperator))}\in\kernel(\timeShiftOperator\identityMatrix{3} - \schemeMatrixBulkFourier(\stableRoot(\timeShiftOperator)))$ and $\transpose{(1, \eigenvectorLetter_{\unstableMarker, \nonConservedMoment}(\timeShiftOperator), \eigenvectorLetter_{\unstableMarker, \nonNonConservedMoment}(\timeShiftOperator))}\in\kernel(\timeShiftOperator\identityMatrix{3} - \schemeMatrixBulkFourier(\unstableRoot(\timeShiftOperator)))$ cannot be continuously extended to $\unitCircle$ for some $\timeShiftOperator$'s.
This has to be taken into account when analyzing numerical results and is made precise by the following result.
\begin{lemma}[Continuous extension of the eigenvectors]\label{lemma:eigenvectorsD1Q3Fourth}
    Let $\courantNumber>0$.
    The stable eigenvector cannot be continuously extended to $\unitCircle$ due to the presence of first-order poles in $\eigenvectorLetter_{\stableMarker, \nonConservedMoment}(\timeShiftOperator)$ and $\eigenvectorLetter_{\stableMarker, \nonNonConservedMoment}(\timeShiftOperator)$ as $\timeShiftOperator\to -1$.

    The unstable eigenvector cannot be continuously extended to $\unitCircle$  either due to the presence of first-order poles in $\eigenvectorLetter_{\unstableMarker, \nonConservedMoment}(\timeShiftOperator)$ as $\timeShiftOperator\to\pm 1$ and in $\eigenvectorLetter_{\unstableMarker, \nonNonConservedMoment}(\timeShiftOperator)$ as $\timeShiftOperator\to -1$.
\end{lemma}

\begin{itemize}

\item \strong{Right boundary}.
We consider the Laurent expansion $\stableRoot(\timeShiftOperator) = \sum_{\indexTime = 1}^{+\infty} \coefficientsLaurentStableRootAllParities_{\indexTime}\timeShiftOperator^{-\indexTime}$, which yields the boundary condition, for $\indexTime\geq 1$
\begin{equation}\label{eq:rightBoundaryD1Q3}
    \phi_{\numberGridPoints + 1}^{\indexTime} = \sum_{k = 1}^{\indexTime}\coefficientsLaurentStableRootAllParities_k \phi_{\numberGridPoints}^{\indexTime-k}, \qquad \phi \in \{ \conservedMoment, \nonConservedMoment, \nonNonConservedMoment \}.
\end{equation}
Inserting the Laurent series of $\stableRoot(\timeShiftOperator)$ into \eqref{eq:charEquationD1Q3FourthOrder} and identifying each term, we find 
\begin{align*}
    \coefficientsLaurentStableRootAllParities_{1} = \tfrac{1}{3}(2\courantNumber - 1)(\courantNumber + 2), \qquad
    \coefficientsLaurentStableRootAllParities_{2} &= - \tfrac{1}{3}(4\courantNumber^2 - 1) \coefficientsLaurentStableRootAllParities_{1}- \tfrac{1}{3}(2\courantNumber + 1)(\courantNumber - 2), \\
    \coefficientsLaurentStableRootAllParities_{3} &= \tfrac{1}{3}(2\courantNumber + 1)(\courantNumber - 2) \coefficientsLaurentStableRootAllParities_1 \coefficientsLaurentStableRootAllParities_1  -\tfrac{1}{3}(4\courantNumber^2 - 1) \coefficientsLaurentStableRootAllParities_2 +\tfrac{1}{3}(4\courantNumber^2 - 1) \coefficientsLaurentStableRootAllParities_1,
\end{align*}
and
\begin{multline*}
    \coefficientsLaurentStableRootAllParities_{\indexTime} = \tfrac{1}{3}(2\courantNumber + 1)(\courantNumber - 2)  \sum_{k = 1}^{\indexTime - 2}\coefficientsLaurentStableRootAllParities_k \coefficientsLaurentStableRootAllParities_{\indexTime -1 - k} - \tfrac{1}{3}(2\courantNumber - 1)(\courantNumber + 2) \sum_{k = 1}^{\indexTime-3}\coefficientsLaurentStableRootAllParities_k \coefficientsLaurentStableRootAllParities_{\indexTime - 2- k} \\- \tfrac{1}{3}(4\courantNumber^2 - 1) (\coefficientsLaurentStableRootAllParities_{\indexTime - 1} - \coefficientsLaurentStableRootAllParities_{\indexTime - 2}) + \coefficientsLaurentStableRootAllParities_{\indexTime - 3}, \quad \indexTime\geq 4.
\end{multline*}

\item \strong{Left boundary}.
For the left boundary, we need to compute 
\begin{equation*}
    \zTransformed{\phi}_0(\timeShiftOperator) = \frac{1}{\unstableRoot(\timeShiftOperator)}\zTransformed{\phi}_1(\timeShiftOperator) = \frac{\stableRoot(\timeShiftOperator)}{\productRoots(\timeShiftOperator)}\zTransformed{\phi}_1(\timeShiftOperator)
\end{equation*}
with $\phi \in \{ \conservedMoment, \nonConservedMoment, \nonNonConservedMoment \}$.
Compared to the \lbmScheme{1}{2}, the product $\productRoots$ is no longer independent of $\timeShiftOperator$.
A first idea would be to exploit the previously found Laurent expansion for $\stableRoot(\timeShiftOperator)$ together with the one of $1/\productRoots(\timeShiftOperator)$, computing each one separately and then combine them by Cauchy product.
Unfortunately, the computation is not viable on floating-point arithmetic as $1/\productRoots(\timeShiftOperator)$ has a pole in $\neighborhoodInfinity$.
More precisely, we find
\begin{align*}
    \frac{1}{\productRoots(\timeShiftOperator)} &= 
    \Biggl ( \frac{(2\courantNumber+1)(\courantNumber-2)}{(2\courantNumber-1)(\courantNumber+2)} - \timeShiftOperator^{-1}\Biggr )
    \Biggl( 1-\frac{(2\courantNumber+1)(\courantNumber-2)}{(2\courantNumber-1)(\courantNumber+2)}\timeShiftOperator^{-1}\Biggr )^{-1}\\
    &= \frac{(2\courantNumber+1)(\courantNumber-2)}{(2\courantNumber-1)(\courantNumber+2)} + 2\, \textnormal{sinh}\Biggl ( \ln \Biggl ( \frac{(2\courantNumber+1)(\courantNumber-2)}{(2\courantNumber-1)(\courantNumber+2)} \Biggr )\Biggr ) \sum_{\indexTime = 1}^{+\infty} \Biggl ( \underbrace{\frac{(2\courantNumber+1)(\courantNumber-2)}{(2\courantNumber-1)(\courantNumber+2)}}_{>1} \Biggr )^{\indexTime}\timeShiftOperator^{-\indexTime},
\end{align*}
which prevents actual computations as the coefficients in the Laurent expansion of $1/\productRoots(\timeShiftOperator)$ geometrically diverge.
A more successful approach consists in finding the Laurent expansion of $\unstableRoot(\timeShiftOperator)^{-1} = \sum_{\indexTime = 1}^{+\infty} \coefficientsLaurentUnstableRootReciprocalAllParities_{\indexTime} \timeShiftOperator^{-\indexTime}$ (which tends to zero as $\timeShiftOperator\to\infty$) as we did for that of $\stableRoot(\timeShiftOperator)$.
Multiplying the characteristic equation \eqref{eq:charEquationD1Q3FourthOrder} by $\fourierShift(\timeShiftOperator)^{-1}$ and proceeding as usual yields 
\begin{align*}
    \coefficientsLaurentUnstableRootReciprocalAllParities_{1}  = \tfrac{1}{3}(2\courantNumber + 1)(\courantNumber - 2), \qquad
    \coefficientsLaurentUnstableRootReciprocalAllParities_{2} &= - \tfrac{1}{3}(4\courantNumber^2 - 1) \coefficientsLaurentUnstableRootReciprocalAllParities_{1}- \tfrac{1}{3}(2\courantNumber - 1)(\courantNumber + 2), \\
    \coefficientsLaurentUnstableRootReciprocalAllParities_{3} &= \tfrac{1}{3}(2\courantNumber - 1)(\courantNumber + 2) \coefficientsLaurentUnstableRootReciprocalAllParities_1 \coefficientsLaurentUnstableRootReciprocalAllParities_1  -\tfrac{1}{3}(4\courantNumber^2 - 1) \coefficientsLaurentUnstableRootReciprocalAllParities_2 +\tfrac{1}{3}(4\courantNumber^2 - 1) \coefficientsLaurentUnstableRootReciprocalAllParities_1, 
\end{align*}
and 
\begin{multline*}
    \coefficientsLaurentUnstableRootReciprocalAllParities_{\indexTime} = \tfrac{1}{3}(2\courantNumber - 1)(\courantNumber + 2)  \sum_{k = 1}^{\indexTime - 2}\coefficientsLaurentUnstableRootReciprocalAllParities_k \coefficientsLaurentUnstableRootReciprocalAllParities_{\indexTime -1 - k} - \tfrac{1}{3}(2\courantNumber + 1)(\courantNumber - 2) \sum_{k = 1}^{\indexTime-3}\coefficientsLaurentUnstableRootReciprocalAllParities_k \coefficientsLaurentUnstableRootReciprocalAllParities_{\indexTime - 2- k} \\
    - \tfrac{1}{3}(4\courantNumber^2 - 1) (\coefficientsLaurentUnstableRootReciprocalAllParities_{\indexTime - 1} - \coefficientsLaurentUnstableRootReciprocalAllParities_{\indexTime - 2}) + \coefficientsLaurentUnstableRootReciprocalAllParities_{\indexTime - 3}, \quad \indexTime\geq 4,
\end{multline*}
so the boundary condition, for $\indexTime\geq 1$
\begin{equation}\label{eq:leftBoundaryD1Q3}
    \phi_{0}^{\indexTime} = \sum_{k = 1}^{\indexTime}\coefficientsLaurentUnstableRootReciprocalAllParities_k \phi_{1}^{\indexTime-k}, \qquad \phi \in \{ \conservedMoment, \nonConservedMoment, \nonNonConservedMoment \}.
\end{equation}
\end{itemize}

\paragraph{Asymptotic expansions}

\begin{figure}
    \begin{center}
        \includegraphics{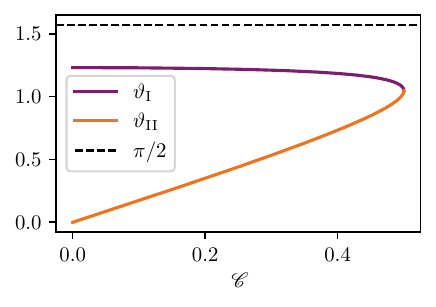}
    \end{center}
    \caption{\label{fig:thetasD1Q3}Plot of $\vartheta_{\textnormal{I}}$ and $\vartheta_{\textnormal{II}}$ given by \eqref{eq:thetaI} and \eqref{eq:thetaII} as functions of $\courantNumber$.}
\end{figure}

We now find the asymptotic trend of $\coefficientsLaurentStableRootAllParities_{\indexTime}$ for large $\indexTime$ ($\coefficientsLaurentUnstableRootReciprocalAllParities_{\indexTime}$ can be investigated analogously).
As the solutions of the characteristic equation feature a square root, we feel that the tame factors will be $\sim\indexTime^{-3/2}$ at leading order.
\begin{proposition}[Asymptotic estimate for $\coefficientsLaurentStableRootAllParities_{\indexTime}$ for the 4th-order D1Q3 scheme]
    For $\indexTime\to +\infty$, the coefficients $\coefficientsLaurentStableRootAllParities_{\indexTime}$ follow the estimate 
    \begin{align}
        \coefficientsLaurentStableRootAllParities_{\indexTime} = 
         -\sqrt{\frac{2}{\pi}} \Bigl (&\sqrt{\sin(\vartheta_{\textnormal{I}})} \textnormal{Re}\bigl (\gamma(e^{i\vartheta_{\textnormal{I}}}) \bigl ( (e^{i\vartheta_{\textnormal{I}}}-e^{i\vartheta_{\textnormal{II}}})(e^{i\vartheta_{\textnormal{I}}}-e^{-i\vartheta_{\textnormal{II}}})\bigr )^{1/2} e^{-i(\indexTime -1/2)\vartheta_{\textnormal{I}} - i\pi/4} \bigr )\nonumber\\
        + &\sqrt{\sin(\vartheta_{\textnormal{II}})} \textnormal{Re}\bigl (  \gamma(e^{i\vartheta_{\textnormal{II}}}) \bigl ( (e^{i\vartheta_{\textnormal{II}}}-e^{i\vartheta_{\textnormal{I}}})(e^{i\vartheta_{\textnormal{II}}}-e^{-i\vartheta_{\textnormal{I}}})\bigr )^{1/2} e^{-i(\indexTime -1/2)\vartheta_{\textnormal{II}} - i\pi/4} \bigr ) \Bigr )\indexTime^{-3/2}+\bigO{\indexTime^{-5/2}},\label{eq:asymptoticsD1Q3}
    \end{align}
    where 
    \begin{align}
        \gamma(\timeShiftOperator) = \frac{ {3\left( {\timeShiftOperator} + 1\right)} }{2((2\courantNumber-1)(\courantNumber+2)\timeShiftOperator - (2\courantNumber+1)(\courantNumber-2))\timeShiftOperator},\quad 
        \vartheta_{\textnormal{I}}&=\arctan\Biggl (\tfrac{\sqrt{4(1-\courantNumber^2)\sqrt{(4\courantNumber^2-1)(\courantNumber^2-1)} - 8\courantNumber^4+13\courantNumber^2 + 4}}{2(1-\courantNumber^2) - \sqrt{(4\courantNumber^2-1)(\courantNumber^2-1)}}\Biggr ),\label{eq:thetaI}\\ 
        \text{and}\qquad \vartheta_{\textnormal{II}}&=\arctan\Biggl (\tfrac{\sqrt{-4(1-\courantNumber^2)\sqrt{(4\courantNumber^2-1)(\courantNumber^2-1)} - 8\courantNumber^4+13\courantNumber^2 + 4}}{2(1-\courantNumber^2) + \sqrt{(4\courantNumber^2-1)(\courantNumber^2-1)}}\Biggr ),\label{eq:thetaII}
    \end{align}
    with $\tfrac{\pi}{2}>\vartheta_{\textnormal{I}}> \vartheta_{\textnormal{II}}>0$.
\end{proposition}
\begin{proof}
    Consider $\stableRoot(\timeShiftOperator^{-1}) = {\mathscr{N}(\timeShiftOperator)}/{\mathscr{D}(\timeShiftOperator)}$ with 
    \begin{align*}
        \mathscr{N}(\timeShiftOperator) &= 3  {\timeShiftOperator}^{3} + {\left(4  {\courantNumber}^{2} - 1\right)} ({\timeShiftOperator}^{2}-\timeShiftOperator)  - 3 +  {\sqrt{3}\left( {\timeShiftOperator} + 1\right)} \sqrt{8  {\left({\courantNumber}^{2} - 1\right)} {\timeShiftOperator}^{3} + 3  {\timeShiftOperator}^{4} - 2  {\left(2  {\courantNumber}^{2} - 5\right)} {\timeShiftOperator}^{2} + 8  {\left({\courantNumber}^{2} - 1\right)} {\timeShiftOperator} + 3}, \\
        \mathscr{D}(\timeShiftOperator) &=2((2\courantNumber-1)(\courantNumber+2)\timeShiftOperator - (2\courantNumber+1)(\courantNumber-2))\timeShiftOperator,
    \end{align*}
    for which one easily sees that $\singularity = 0$ is a removable singularity.
    Conversely, the point $\singularity = (2\courantNumber+1)(\courantNumber-2)/((2\courantNumber-1)(\courantNumber+2))$ is a pole singularity of $\stableRoot(\timeShiftOperator^{-1})$.
    One easily checks that $\singularity > 1$ for $0<\courantNumber<\tfrac{1}{2}$. Therefore, this is not the dominant singularity, as we shall see that the other singularities yet to be considered dwell closer to the origin.
    These are the four branch points in the square root, which are pairwise complex conjugate, and given by 
    \begin{align*}
        \singularity_1 &= \tfrac{2}{3}(1-\courantNumber^2) - \tfrac{1}{3}\sqrt{\tilde{\Pi}(\courantNumber)} + \tfrac{i}{3}\sqrt{4(1-\courantNumber^2)\sqrt{\tilde{\Pi}(\courantNumber)} - 8\courantNumber^4+13\courantNumber^2 + 4} = e^{i\vartheta_{\textnormal{I}}} \in\unitCircle, \\
        \singularity_2 &= \tfrac{2}{3}(1-\courantNumber^2) - \tfrac{1}{3}\sqrt{\tilde{\Pi}(\courantNumber)} - \tfrac{i}{3}\sqrt{4(1-\courantNumber^2)\sqrt{\tilde{\Pi}(\courantNumber)} - 8\courantNumber^4+13\courantNumber^2 + 4} = e^{-i\vartheta_{\textnormal{I}}} \in\unitCircle, \\
        \singularity_3 &= \tfrac{2}{3}(1-\courantNumber^2) + \tfrac{1}{3}\sqrt{\tilde{\Pi}(\courantNumber)} + \tfrac{i}{3}\sqrt{-4(1-\courantNumber^2)\sqrt{\tilde{\Pi}(\courantNumber)} - 8\courantNumber^4+13\courantNumber^2 + 4} = e^{i\vartheta_{\textnormal{II}}} \in\unitCircle, \\
        \singularity_4 &= \tfrac{2}{3}(1-\courantNumber^2) + \tfrac{1}{3}\sqrt{\tilde{\Pi}(\courantNumber)} - \tfrac{i}{3}\sqrt{-4(1-\courantNumber^2)\sqrt{\tilde{\Pi}(\courantNumber)} - 8\courantNumber^4+13\courantNumber^2 + 4} = e^{-i\vartheta_{\textnormal{II}}} \in\unitCircle,
    \end{align*}
    where we have defined $\tilde{\Pi}(\courantNumber)= (4\courantNumber^2-1)(\courantNumber^2-1)>0$, and 
    \begin{equation}\label{eq:argumentD1Q3One}
        \vartheta_{\textnormal{I}}=\arctan\Biggl (\tfrac{\sqrt{4(1-\courantNumber^2)\sqrt{\tilde{\Pi}(\courantNumber)} - 8\courantNumber^4+13\courantNumber^2 + 4}}{2(1-\courantNumber^2) - \sqrt{\tilde{\Pi}(\courantNumber)}}\Biggr ), \qquad 
        \vartheta_{\textnormal{II}}=\arctan\Biggl (\tfrac{\sqrt{-4(1-\courantNumber^2)\sqrt{\tilde{\Pi}(\courantNumber)} - 8\courantNumber^4+13\courantNumber^2 + 4}}{2(1-\courantNumber^2) + \sqrt{\tilde{\Pi}(\courantNumber)}}\Biggr ). 
    \end{equation}
    One can check that in these expressions, the terms under square roots are positive.
    From \Cref{fig:thetasD1Q3}, we understand that $\singularity_3$ and $\singularity_4$ tend to the same point equal to $1$ as $\courantNumber\to 0$, whereas $\singularity_1\to\singularity_3$ and $\singularity_2\to\singularity_4$ as $\courantNumber\to \tfrac{1}{2}$.
    Moreover, the $\singularity$'s stay in the right half-plane for every value of $\courantNumber$.
    In order to study the impact of these singularities, we consider $\psi(\timeShiftOperator) = \gamma(\timeShiftOperator)({\prod_{k = 1}^{4}(\timeShiftOperator-\singularity_k)})^{1/2}$ with $\gamma(\timeShiftOperator) = { {3\left( {\timeShiftOperator} + 1\right)} }/{\mathscr{D}(\timeShiftOperator)}$.
    Note that the coefficients of $\gamma(\timeShiftOperator)$ are real, which entails that $\gamma(\overline{\timeShiftOperator}) = \overline{\gamma(\timeShiftOperator)}$.
    In the vicinity of $\singularity_1$, we have 
    \begin{equation*}
        \psi(\timeShiftOperator)= \sqrt{2}\sqrt{\sin(\vartheta_{\textnormal{I}})} \gamma(\singularity_1) \bigl ( (e^{i\vartheta_{\textnormal{I}}}-e^{i\vartheta_{\textnormal{II}}})(e^{i\vartheta_{\textnormal{I}}}-e^{-i\vartheta_{\textnormal{II}}})\bigr )^{1/2} e^{i\vartheta_{\textnormal{I}}/2 - i\pi/4} \sqrt{1-\frac{\timeShiftOperator}{\singularity_1}} + \bigO{(1-\timeShiftOperator/\singularity_1)^{3/2}},
    \end{equation*}
    whereas close to $\singularity_2$
    \begin{equation*}
        \psi(\timeShiftOperator)= \sqrt{2}\sqrt{\sin(\vartheta_{\textnormal{I}})} \overline{\gamma(\singularity_1) \bigl ( (e^{i\vartheta_{\textnormal{I}}}-e^{i\vartheta_{\textnormal{II}}})(e^{i\vartheta_{\textnormal{I}}}-e^{-i\vartheta_{\textnormal{II}}})\bigr )^{1/2} e^{i\vartheta_{\textnormal{I}}/2 - i\pi/4} } \sqrt{1-\frac{\timeShiftOperator}{\singularity_2}} + \bigO{(1-\timeShiftOperator/\singularity_2)^{3/2}}.
    \end{equation*}
    In the same way, close to $\singularity_3$ and $\singularity_4$, we obtain 
    \begin{align*}
        \psi(\timeShiftOperator)&= \sqrt{2}\sqrt{\sin(\vartheta_{\textnormal{II}})} \gamma(\singularity_3) \bigl ( (e^{i\vartheta_{\textnormal{II}}}-e^{i\vartheta_{\textnormal{I}}})(e^{i\vartheta_{\textnormal{II}}}-e^{-i\vartheta_{\textnormal{I}}})\bigr )^{1/2} e^{i\vartheta_{\textnormal{II}}/2 - i\pi/4} \sqrt{1-\frac{\timeShiftOperator}{\singularity_3}} + \bigO{(1-\timeShiftOperator/\singularity_3)^{3/2}},\\
        \psi(\timeShiftOperator)&= \sqrt{2}\sqrt{\sin(\vartheta_{\textnormal{II}})} \overline{\gamma(\singularity_3) \bigl ( (e^{i\vartheta_{\textnormal{II}}}-e^{i\vartheta_{\textnormal{I}}})(e^{i\vartheta_{\textnormal{II}}}-e^{-i\vartheta_{\textnormal{I}}})\bigr )^{1/2} e^{i\vartheta_{\textnormal{II}}/2 - i\pi/4}} \sqrt{1-\frac{\timeShiftOperator}{\singularity_4}} + \bigO{(1-\timeShiftOperator/\singularity_4)^{3/2}}.
    \end{align*}
    This gives the result, by virtue of \cite[Figure VI.5 and Theorem VI.5]{flajolet2009analytic}.
\end{proof}

\begin{figure}
    \begin{center}
        \includegraphics[width=1\textwidth]{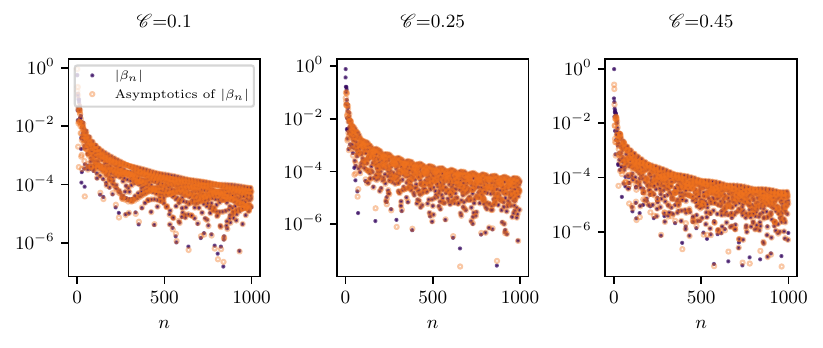}
    \end{center}\caption{\label{fig:D1Q3-4th-asymptotic}Values of $|\coefficientsLaurentStableRootAllParities_{\indexTime}|$ versus their asymptotics given by the leading order of the right-hand side of \eqref{eq:asymptoticsD1Q3}.}
\end{figure}

\Cref{fig:D1Q3-4th-asymptotic} demonstrates the accuracy of the asymptotic expansion \eqref{eq:asymptoticsD1Q3} for different Courant numbers.

\paragraph{Numerical experiments}

\begin{figure}
    \begin{center}
        \includegraphics[width=1\textwidth]{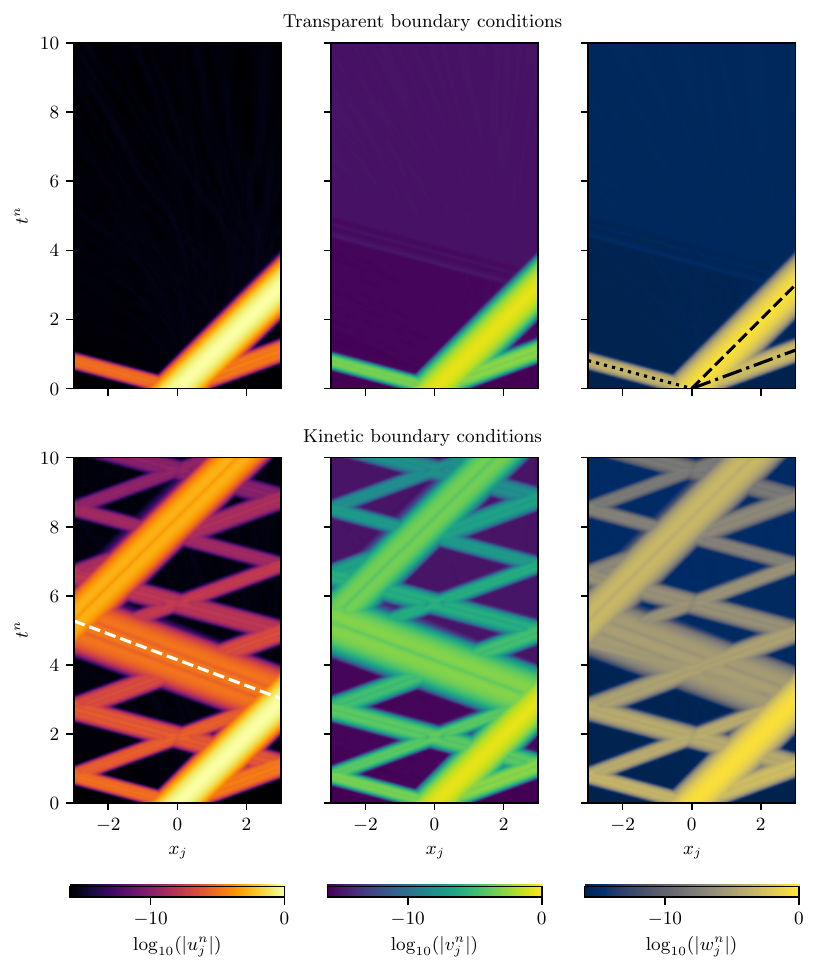}
    \end{center}\caption{\label{fig:D1Q3-4th-order-BC-comparison}Solutions for the fourth-order \lbmScheme{1}{3} endowed with different boundary conditions. Black lines indicate the slope $\latticeVelocity\courantNumber = \advectionVelocity$ for the dashed one, $-\tfrac{\latticeVelocity}{6}(3\courantNumber+\sqrt{24-15\courantNumber^2})$ for the dotted one, and $-\tfrac{\latticeVelocity}{6}(3\courantNumber-\sqrt{24-15\courantNumber^2})$ for the dash-dotted one, \confer{} \eqref{eq:groupVelocitySmoothParasiticD1Q3}, in the $\spaceVariable/\timeVariable$ plane. The white line corresponds to the slope $-3\advectionVelocity/(2\courantNumber^2+1)$, \confer{} \eqref{eq:groupVelocityD1Q3Fourth}.}
\end{figure}

We perform numerical simulations in the same setting as the \lbmScheme{1}{2} scheme, \confer{} \Cref{sec:accuracyCoeffs}, except for the fact of taking $\latticeVelocity = 4$ to ensure stability.
We compare the transparent boundary conditions to the kinetic boundary conditions given by \eqref{eq:kineticBoundaryConditions}, obtaining \Cref{fig:D1Q3-4th-order-BC-comparison}.
\begin{itemize}
    \item For transparent boundary conditions, nothing special has to be said, except the fact that we observe the effect of the lack of continuous extension of the eigenvectors concerning $\nonConservedMoment$ and $\nonNonConservedMoment$, \confer{} \Cref{lemma:eigenvectorsD1Q3Fourth}.
    \item For the kinetic boundary conditions, we observe repeated reflections. 
    The results feature an endless exchange between the two spurious modes, especially around the double root $(-1, 1)$ of the dispersion relation, and which travel at the (rapid) speed prescribed by \eqref{eq:groupVelocitySmoothParasiticD1Q3}.
    Precisely tracking what is the right-going physical wave at the beginning of the simulation, the first reflection is $(1, 1)\rightsquigarrow (1, -1)$, where this time-smooth--space-checkerboard mode travels at the easily computable group velocity
    \begin{equation}\label{eq:groupVelocityD1Q3Fourth}
        \textnormal{V}_{\textnormal{g}}(\eta\timeStep = 0, \frequency\spaceStep = \pm\pi) = -\frac{3\advectionVelocity}{2\courantNumber^2 + 1}.
    \end{equation}
    The next reflection on $\leftBoundary$ gives $(1, -1) \rightsquigarrow (1, 1)$.
\end{itemize}

\subsection{2D setting}\label{sec:2DScalar}

\subsubsection{Target partial differential equation}

We approximate the solution of 
\begin{equation*}
    \begin{dcases}
        \begin{aligned}
            &\partial_{\timeVariable}\conservedMoment(\timeVariable, \vectorial{\spaceVariable}) + \advectionVelocity_{\xLabel}\partial_{\xLabel}\conservedMoment(\timeVariable, \vectorial{\spaceVariable}) + \advectionVelocity_{\yLabel}\partial_{\yLabel}\conservedMoment(\timeVariable, \vectorial{\spaceVariable})  = 0, \qquad &\timeVariable>0, \quad &\vectorial{\spaceVariable}\in\reals^2, \\
            &\conservedMoment(0, \vectorial{\spaceVariable})  = \conservedMoment^{\circ}(\vectorial{\spaceVariable}), \qquad &&\vectorial{\spaceVariable}\in\reals^2.
        \end{aligned}
    \end{dcases}
\end{equation*}
The exact solution is $\conservedMoment(\timeVariable, \vectorial{\spaceVariable}) =  \conservedMoment^{\circ}(\vectorial{\spaceVariable}-\transpose{(\advectionVelocity_{\xLabel}, \advectionVelocity_{\yLabel})}\timeVariable)$.
Numerical simulations must be conducted on the bounded domain $[\leftBoundary, \rightBoundary]\times [\bottomBoundary, \topBoundary]$.

\subsubsection{Space and time discretization}

As space and time discretization are very similar to the 1D setting, we just detail new features.
For the sake of simplicity, we assume that 
\begin{equation*}
    \frac{\topBoundary-\bottomBoundary}{\rightBoundary-\leftBoundary}\in \mathbb{Q}\cap (0, +\infty),
\end{equation*}
so that the interval $[\bottomBoundary, \topBoundary]$ along the $\yLabel$-axis is discretized using $\numberGridPointsOnY + 2$ points, with 
\begin{equation*}
    \numberGridPointsOnY\definitionEquality\frac{\topBoundary-\bottomBoundary}{\rightBoundary-\leftBoundary}\numberGridPoints,
\end{equation*}
that we assume $\numberGridPointsOnY\in\naturals$.
This entails that the space step along the $\yLabel$-axis is also given by $\spaceStep$, and we have $\spaceGridPointOnY{k}\definitionEquality\bottomBoundary + k\spaceStep$, with $k\in\integerInterval{0}{\numberGridPointsOnY + 1}$.
We have Courant numbers along each axis, given by $\courantNumber_{\xLabel}\definitionEquality\advectionVelocity_{\xLabel}/\latticeVelocity$ and $\courantNumber_{\yLabel}\definitionEquality\advectionVelocity_{\yLabel}/\latticeVelocity$.

\subsubsection{Numerical scheme and its properties: link TRT \lbmScheme{2}{5} scheme with magic parameter equal to $1/4$}

This numerical scheme has been introduced in \cite{d2009viscosity}.
More recently \cite{bellotti2024initialisation}, we have rediscovered this approach and have shown one of its peculiar features: a simple dispersion relation that---despite the 2D setting---promises quite explicit stability conditions.
Moreover, in the case of relaxation parameter equal to two, this scheme essentially becomes the 2D leap-frog scheme of \cite{besse2021discrete}, and---as these authors point out---``exhibits some ``dimensional splitting'''' which facilitates the construction of boundary conditions without having to deal with corners. 
Overall, this numerical method shares similarities with the \lbmScheme{1}{2} scheme.

\begin{itemize}
    \item \strong{Relaxation} on the moments: 
    \begin{align*}
        \conservedMoment_{\vectorial{\indexSpace}}^{\indexTime, \collided} = \conservedMoment_{\vectorial{\indexSpace}}^{\indexTime}, \qquad 
        \antiSymmetricMomentX{\vectorial{\indexSpace}}^{\indexTime, \collided} &= (1-\relaxationParameter) \antiSymmetricMomentX{\vectorial{\indexSpace}}^{\indexTime}+\relaxationParameter \overbrace{\courantNumberX \conservedMoment_{\vectorial{\indexSpace}}^{\indexTime}}^{\antiSymmetricMomentXLetter^{\atEquilibrium}(\conservedMoment_{\vectorial{\indexSpace}}^{\indexTime})}, \qquad \symmetricMomentX{\vectorial{\indexSpace}}^{\indexTime, \collided} = (\relaxationParameter-1) \symmetricMomentX{\vectorial{\indexSpace}}^{\indexTime}+(2-\relaxationParameter) \overbrace{\equilibriumCoefficientSymmetricX \conservedMoment_{\vectorial{\indexSpace}}^{\indexTime}}^{\symmetricMomentXLetter^{\atEquilibrium}(\conservedMoment_{\vectorial{\indexSpace}}^{\indexTime})}, \\
        \antiSymmetricMomentY{\vectorial{\indexSpace}}^{\indexTime, \collided} &= (1-\relaxationParameter) \antiSymmetricMomentY{\vectorial{\indexSpace}}^{\indexTime}+\relaxationParameter \underbrace{\courantNumberY \conservedMoment_{\vectorial{\indexSpace}}^{\indexTime}}_{\antiSymmetricMomentYLetter^{\atEquilibrium}(\conservedMoment_{\vectorial{\indexSpace}}^{\indexTime})}, \qquad \symmetricMomentY{\vectorial{\indexSpace}}^{\indexTime, \collided} = (\relaxationParameter-1) \symmetricMomentY{\vectorial{\indexSpace}}^{\indexTime}+(2-\relaxationParameter) \underbrace{\equilibriumCoefficientSymmetricY \conservedMoment_{\vectorial{\indexSpace}}^{\indexTime}}_{\symmetricMomentYLetter^{\atEquilibrium}(\conservedMoment_{\vectorial{\indexSpace}}^{\indexTime})},
    \end{align*}
    for $\vectorial{\indexSpace}\in\integerInterval{0}{\numberGridPoints + 1}\times \integerInterval{0}{\numberGridPointsOnY + 1} \smallsetminus \{(0, 0), (0, \numberGridPointsOnY + 1), (\numberGridPoints + 1, 0), (\numberGridPoints + 1, \numberGridPointsOnY + 1)\}$, i.e. on the inner domain, plus boundary strips parallel to the axes, excluding corners.
    The magic parameter is the product of the H\'enon's parameters of $\nonConservedMoment$ and $\nonNonConservedMoment$ (whatever the considered axis), which reads 
    \begin{equation*}
        \Bigl (\frac{1}{\relaxationParameter}-\frac{1}{2} \Bigr )\times\Bigl (\frac{1}{2-\relaxationParameter}-\frac{1}{2} \Bigr ) = \frac{1}{4}.
    \end{equation*}
    Notice that upon taking $\relaxationParameter = 2$, we also ``artificially'' enforce conservation for the $\nonNonConservedMoment$'s.
    However, note that in this case the equilibrium of $\nonConservedMoment$ depends on $\conservedMoment$, and not on the $\nonNonConservedMoment$'s.
    The real coefficients $\equilibriumCoefficientSymmetricX$ and $\equilibriumCoefficientSymmetricY$ tune dissipative features, and are pointless whenever $\relaxationParameter = 2$.
    \item \strong{Transport} on the distribution functions. After having set
    \begin{equation*}
        \momentMatrix 
        = 
        \begin{pmatrix}
            1 & 1 & 1 & 1 & 1\\
            0 & 1 & -1 & 0 & 0\\
            0 & 1 & 1 & 0 & 0\\
            0 & 0 & 0 & 1 & -1\\
            0 & 0 & 0 & 1 & 1
        \end{pmatrix} 
        \qquad \text{and computed}\qquad 
        \begin{pmatrix}
            \distributionFunctionLetter_{\labZeroVel, \vectorial{\indexSpace}}^{\indexTime, \collided}\\
            \distributionFunctionLetter_{\labPosX, \vectorial{\indexSpace}}^{\indexTime, \collided}\\
            \distributionFunctionLetter_{\labNegX, \vectorial{\indexSpace}}^{\indexTime, \collided}\\
            \distributionFunctionLetter_{\labPosY, \vectorial{\indexSpace}}^{\indexTime, \collided}\\
            \distributionFunctionLetter_{\labNegY, \vectorial{\indexSpace}}^{\indexTime, \collided}
        \end{pmatrix}
        = \momentMatrix^{-1}
        \begin{pmatrix}
            \conservedMoment_{\vectorial{\indexSpace}}^{\indexTime, \collided}\\
            \antiSymmetricMomentX{\vectorial{\indexSpace}}^{\indexTime, \collided}\\
            \symmetricMomentX{\vectorial{\indexSpace}}^{\indexTime, \collided}\\
            \antiSymmetricMomentY{\vectorial{\indexSpace}}^{\indexTime, \collided}\\
            \symmetricMomentY{\vectorial{\indexSpace}}^{\indexTime, \collided}
        \end{pmatrix},
    \end{equation*}
    for $\vectorial{\indexSpace}\in\integerInterval{0}{\numberGridPoints + 1}\times \integerInterval{0}{\numberGridPointsOnY + 1} \smallsetminus \{(0, 0), (0, \numberGridPointsOnY + 1), (\numberGridPoints + 1, 0), (\numberGridPoints + 1, \numberGridPointsOnY + 1)\}$, transport reads 
    \begin{align*}
        \distributionFunctionLetter_{\labZeroVel, \vectorial{\indexSpace}}^{\indexTime + 1} = \distributionFunctionLetter_{\labZeroVel, \vectorial{\indexSpace}}^{\indexTime, \collided}, 
        \qquad 
        \distributionFunctionLetter_{\labPosX, \vectorial{\indexSpace}}^{\indexTime + 1} &= \distributionFunctionLetter_{\labPosX, \vectorial{\indexSpace}-\canonicalBasisVector{\xLabel}}^{\indexTime, \collided}, \qquad 
        \distributionFunctionLetter_{\labNegX, \vectorial{\indexSpace}}^{\indexTime + 1} = \distributionFunctionLetter_{\labNegX, \vectorial{\indexSpace}+\canonicalBasisVector{\xLabel}}^{\indexTime, \collided}, \\
        \distributionFunctionLetter_{\labPosY, \vectorial{\indexSpace}}^{\indexTime + 1} &= \distributionFunctionLetter_{\labPosY, \vectorial{\indexSpace}-\canonicalBasisVector{\yLabel}}^{\indexTime, \collided}, \qquad 
        \distributionFunctionLetter_{\labNegY, \vectorial{\indexSpace}}^{\indexTime + 1} = \distributionFunctionLetter_{\labNegY, \vectorial{\indexSpace}+\canonicalBasisVector{\yLabel}}^{\indexTime, \collided}, \qquad \vectorial{\indexSpace}\in \integerInterval{1}{\numberGridPoints}\times \integerInterval{1}{\numberGridPointsOnY}.
    \end{align*}
    We then compute $\transpose{(\conservedMoment_{\vectorial{\indexSpace}}^{\indexTime+1},
    \antiSymmetricMomentX{\vectorial{\indexSpace}}^{\indexTime+1},
    \symmetricMomentX{\vectorial{\indexSpace}}^{\indexTime+1},
    \antiSymmetricMomentY{\vectorial{\indexSpace}}^{\indexTime+1},
    \symmetricMomentY{\vectorial{\indexSpace}}^{\indexTime+1})} = \momentMatrix \transpose{(\distributionFunctionLetter_{\labZeroVel, \vectorial{\indexSpace}}^{\indexTime+1},
    \distributionFunctionLetter_{\labPosX, \vectorial{\indexSpace}}^{\indexTime+1},
    \distributionFunctionLetter_{\labNegX, \vectorial{\indexSpace}}^{\indexTime+1},
    \distributionFunctionLetter_{\labPosY, \vectorial{\indexSpace}}^{\indexTime+1},
    \distributionFunctionLetter_{\labNegY, \vectorial{\indexSpace}}^{\indexTime+1})}$ for $ \vectorial{\indexSpace}\in \integerInterval{1}{\numberGridPoints}\times \integerInterval{1}{\numberGridPointsOnY}$, hence only on the inner domain.
    \item \strong{Boundary conditions}. Compute $(\conservedMoment_{\vectorial{\indexSpace}}^{\indexTime+1},
    \antiSymmetricMomentX{\vectorial{\indexSpace}}^{\indexTime+1},
    \symmetricMomentX{\vectorial{\indexSpace}}^{\indexTime+1},
    \antiSymmetricMomentY{\vectorial{\indexSpace}}^{\indexTime+1},
    \symmetricMomentY{\vectorial{\indexSpace}}^{\indexTime+1})$ for $\vectorial{\indexSpace}\in \{0, \numberGridPoints + 1\}\times \integerInterval{1}{\numberGridPointsOnY} \cup \integerInterval{1}{\numberGridPoints}\times  \{0, \numberGridPointsOnY + 1\}$ as described in \Cref{sec:bc2D}.
\end{itemize}

In \cite{bellotti2024initialisation}, we have proved that the characteristic equation for the scheme matrix $\schemeMatrixBulkFourier(\fourierShift_{\xLabel}, \fourierShift_{\yLabel})$ (with obvious notations concerning $\fourierShift_{\xLabel}$ and $\fourierShift_{\yLabel}$) reads
\begin{align}
    \determinant(\timeShiftOperator\identityMatrix{5} - \schemeMatrixBulkFourier(\fourierShift_{\xLabel}, \fourierShift_{\yLabel})) = (\timeShiftOperator + (1-\relaxationParameter))(\timeShiftOperator^2 &- (1-\relaxationParameter)^2) \Psi_{2}(\timeShiftOperator) = 0,\label{eq:charEquationD2Q5Magic}\\
    \text{where}\qquad 
    \Psi_{2}(\timeShiftOperator)\definitionEquality\timeShiftOperator^2 + \Bigl ( (\relaxationParameter-2) &- \frac{\relaxationParameter\courantNumberX}{2}(\fourierShift_{\xLabel}^{-1} - \fourierShift_{\xLabel}) + \frac{(\relaxationParameter-2)\equilibriumCoefficientSymmetricX}{2}(\fourierShift_{\xLabel}^{-1} -2 + \fourierShift_{\xLabel}) \nonumber\\
    &- \frac{\relaxationParameter\courantNumberY}{2}(\fourierShift_{\yLabel}^{-1} - \fourierShift_{\yLabel}) + \frac{(\relaxationParameter-2)\equilibriumCoefficientSymmetricY}{2}(\fourierShift_{\yLabel}^{-1} -2 + \fourierShift_{\yLabel})\Bigr )\timeShiftOperator + (1 - \relaxationParameter).\nonumber
\end{align}
We then see that $\Psi_{2}(\timeShiftOperator)$ becomes, when $\relaxationParameter = 2$, the amplification polynomial of a two-dimensional leap-frog scheme treated in \cite{besse2021discrete}.
In this setting, we also have the constant roots $\timeShiftOperator = 1$ (double) and $\timeShiftOperator = -1$ (simple). The former encode that the $\nonNonConservedMoment$'s are (artificially) conserved. 
Yet, these eigenvalues are constant in $\fourierShift$ because the equilibria do not depend on these moments.

\begin{figure}
    \begin{center}
        \includegraphics[width=1\textwidth]{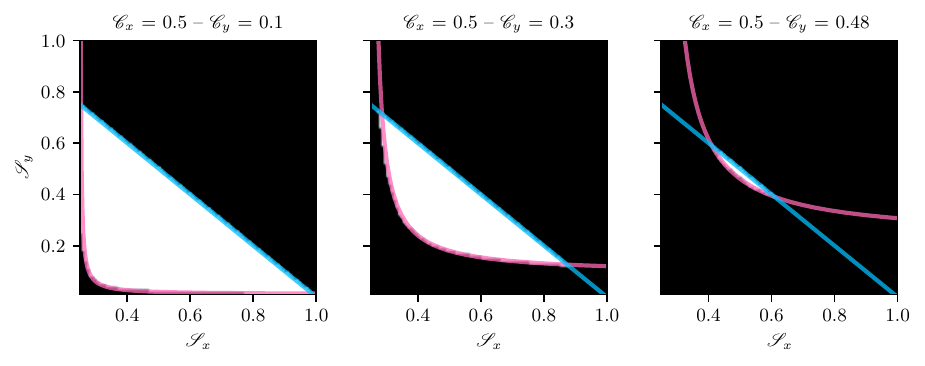}
    \end{center}\caption{\label{fig:stabD2Q5magic}Stability area (white), where \eqref{eq:neuamnnConditionD2Q5MagicNotTwo} holds, for different values of $\courantNumberX$ and $\courantNumberY$, as function of $\equilibriumCoefficientSymmetricX$ and $\equilibriumCoefficientSymmetricY$. This area has been found numerically. The light blue line indicates the constraint \eqref{eq:finalCondition}, whereas the pink one indicates the second constraint in \eqref{eq:positiveDiffusion}.}
\end{figure}

\begin{proposition}[Von Neumann stability of the link TRT \lbmScheme{2}{5} scheme on $\relatives^2$]\label{prop:stabilityD2Q5}
    The link TRT \lbmScheme{2}{5} scheme with magic parameter equal to $\tfrac{1}{4}$ is stable in the sense of von Neumann, namely $\spectrum(\schemeMatrixBulkFourier(e^{i\frequency_{\xLabel}\spaceStep}, e^{i\frequency_{\yLabel}\spaceStep}))\subset \closedUnitDisk$ for $\frequency_{\xLabel}\spaceStep, \frequency_{\yLabel}\spaceStep\in[-\pi, \pi]$, under the following necessary and sufficient conditions.
    \begin{itemize}
        \item For $\relaxationParameter = 2$, the condition\footnote{We know that equality cannot hold to have actual $\ell^2$-stability, due to coinciding roots on $\unitCircle$ midway in the spectrum.} is 
        \begin{equation}\label{eq:stabCondition2DLeapFrof}
            |\courantNumberX| + |\courantNumberY| \leq 1.
        \end{equation}
        \item For $\relaxationParameter \in (0, 2)$, the condition is 
        \begin{multline}\label{eq:neuamnnConditionD2Q5MagicNotTwo}
            -\bigl ( (\courantNumberX^2-\equilibriumCoefficientSymmetricX^2)\mu_{\xLabel}^2 - (\courantNumberX^2 - \equilibriumCoefficientSymmetricX)\bigr )\mu_{\xLabel}^2 -\bigl ( (\courantNumberY^2-\equilibriumCoefficientSymmetricY^2)\mu_{\yLabel}^2 - (\courantNumberY^2 - \equilibriumCoefficientSymmetricY)\bigr )\mu_{\yLabel}^2\\
            +2\Bigl (\equilibriumCoefficientSymmetricX\equilibriumCoefficientSymmetricY \mu_{\xLabel}\mu_{\yLabel} + \courantNumberX\courantNumberY \sqrt{(1-\mu_{\xLabel}^2)(1-\mu_{\yLabel}^2)}\Bigr ) \mu_{\xLabel}\mu_{\yLabel}\leq 0, \qquad \forall (\mu_{\xLabel}, \mu_{\yLabel})\in[-1, 1]^2.
        \end{multline}
    \end{itemize}

    Moreover, when $\relaxationParameter\in(0, 2)$, a necessary condition to fulfill \eqref{eq:neuamnnConditionD2Q5MagicNotTwo} is that \eqref{eq:stabCondition2DLeapFrof} holds.
    Provided that this latter holds, another necessary condition for having \eqref{eq:neuamnnConditionD2Q5MagicNotTwo} reads 
    \begin{equation}\label{eq:1DCondition}
        \courantNumberX^2\leq \equilibriumCoefficientSymmetricX\leq 1\qquad \text{and}\qquad 
        \courantNumberY^2\leq \equilibriumCoefficientSymmetricY\leq 1.
    \end{equation}
    Assuming that \eqref{eq:stabCondition2DLeapFrof}--\eqref{eq:1DCondition} are met, an additional necessary condition for \eqref{eq:neuamnnConditionD2Q5MagicNotTwo} reads 
    \begin{equation}\label{eq:finalCondition}
        \equilibriumCoefficientSymmetricX + \equilibriumCoefficientSymmetricY \leq 1.
    \end{equation}
\end{proposition}
The proof of this result is detailed in \Cref{app:prop:stabilityD2Q5}.
It is interesting to note that these conditions do not depend on $\relaxationParameter$ as long as it belongs to $(0, 2]$, analogously to the \lbmScheme{1}{2} scheme.
Note that \eqref{eq:1DCondition} and \eqref{eq:finalCondition} describe, for given $\courantNumberX$ and $\courantNumberY$ fulfilling \eqref{eq:stabCondition2DLeapFrof}, a right triangle in the $\equilibriumCoefficientSymmetricX\equilibriumCoefficientSymmetricY$-plane with vertexes $(\courantNumberX^2, \courantNumberY^2)$, $(1-\courantNumberY^2, \courantNumberY^2)$, and $(\courantNumberX^2, 1-\courantNumberX^2)$.
Computing the modified equation for this scheme, we obtain 
\begin{equation*}
    \partial_{\timeVariable}\conservedMoment + \advectionVelocity_{\xLabel}\partial_{\xLabel}\conservedMoment + \advectionVelocity_{\yLabel}\partial_{\yLabel}\conservedMoment = \latticeVelocity\spaceStep \Bigl (\frac{1}{\relaxationParameter}-\frac{1}{2}\Bigr ) \Bigl ( \partial_{\xLabel}\bigl ((\equilibriumCoefficientSymmetricX-\courantNumberX^2)\partial_{\xLabel} - \courantNumberX\courantNumberY \partial_{\yLabel}\bigr )  + 
    \partial_{\yLabel}\bigl (-\courantNumberX\courantNumberY \partial_{\xLabel} + (\equilibriumCoefficientSymmetricY-\courantNumberY^2)\partial_{\yLabel} \bigr ) \Bigr )\conservedMoment +\bigO{\spaceStep^2}
\end{equation*}
with diffusion term having the following associated matrix 
\begin{equation*}
    \matricial{D}
    =
    \begin{pmatrix}
        \equilibriumCoefficientSymmetricX-\courantNumberX^2 & -\courantNumberX\courantNumberY\\
        -\courantNumberX\courantNumberY & \equilibriumCoefficientSymmetricY-\courantNumberY^2
    \end{pmatrix}.
\end{equation*}
We then consider the stability criterion introduced in \cite[Chapter 6]{helie2023schema}, which is quite in the spirit of the work of \cite{warming1974modified}, requesting that $\matricial{D}$ be positive definite.
This is equivalent to 
\begin{equation}\label{eq:positiveDiffusion}
    \equilibriumCoefficientSymmetricX>\courantNumberX^2 \qquad \text{and}\qquad \equilibriumCoefficientSymmetricX\equilibriumCoefficientSymmetricY - \equilibriumCoefficientSymmetricX\courantNumberY^2 - \equilibriumCoefficientSymmetricY\courantNumberX^2>0.
\end{equation}
If we relax this criterion by requesting just positive semi-definedness for $\matricial{D}$, we have 
\begin{equation}\label{eq:semipositiveDiffusion}
    \equilibriumCoefficientSymmetricX\geq \courantNumberX^2, \qquad \equilibriumCoefficientSymmetricY\geq\courantNumberY^2 , \qquad \text{and}\qquad \equilibriumCoefficientSymmetricX\equilibriumCoefficientSymmetricY - \equilibriumCoefficientSymmetricX\courantNumberY^2 - \equilibriumCoefficientSymmetricY\courantNumberX^2\geq 0.
\end{equation}
Notice that $\matricial{D}$ is noting but $-1/2$ times the Hessian of the left-hand side of \eqref{eq:neuamnnConditionD2Q5MagicNotTwo} at $\mu_{\xLabel} = \mu_{\yLabel} = 0$ , where a Taylor expansion in this vicinity (i.e., for low frequencies) gives 
\begin{equation*}
    \textnormal{[left-hand side of \eqref{eq:neuamnnConditionD2Q5MagicNotTwo}]}(\mu_{\xLabel}, \mu_{\yLabel}) = -\transpose{(\mu_{\xLabel}, \mu_{\yLabel})} \matricial{D} (\mu_{\xLabel}, \mu_{\yLabel}) + \textnormal{higher order terms},
\end{equation*}
which justifies the criterion on the diffusion matrix from the point of view of \eqref{eq:neuamnnConditionD2Q5MagicNotTwo}.

Despite difficulties in explicitly writing easier conditions for \eqref{eq:neuamnnConditionD2Q5MagicNotTwo} (notice however that \eqref{eq:neuamnnConditionD2Q5MagicNotTwo} is nothing but the claim of \cite[Theorem 2.2.1]{ginzburg2010optimal} and \cite[Equation (23)]{kuzmin2011role}), we can numerically study it, see \Cref{fig:stabD2Q5magic}.
Interestingly, it seems that the stability area is the one enclosed between $\equilibriumCoefficientSymmetricX+\equilibriumCoefficientSymmetricY\leq 1$ and $\equilibriumCoefficientSymmetricX\equilibriumCoefficientSymmetricY - \equilibriumCoefficientSymmetricX\courantNumberY^2 - \equilibriumCoefficientSymmetricY\courantNumberX^2\geq0$.
This conjecture is quite plausible, since we practically see that \eqref{eq:neuamnnConditionD2Q5MagicNotTwo} is violated either on the corners of $[-1, 1]^2$, which gives $\equilibriumCoefficientSymmetricX+\equilibriumCoefficientSymmetricY\leq 1$, or in the neighborhood of $(0, 0)$, where a Taylor expansion of the consistency root of the characteristic equation for small frequencies indeed yields a criterion on the diffusivity, hence $\equilibriumCoefficientSymmetricX\equilibriumCoefficientSymmetricY - \equilibriumCoefficientSymmetricX\courantNumberY^2 - \equilibriumCoefficientSymmetricY\courantNumberX^2\geq0$.

\begin{remark}[Sufficient condition]
    A sufficient stability condition is given in \cite{ginzburg2010optimal, kuzmin2011role}, see their equation (49) (resp., equation (24)).
    To bridge with our previous discussion, our necessary condition \eqref{eq:finalCondition} is one of the three inequalities making up a sufficient stability condition in \cite{ginzburg2010optimal}.
    Assuming that $\courantNumberX, \courantNumberY\neq 0$, the inequalities in \eqref{eq:semipositiveDiffusion} entail the second and third parts of the condition from \cite{ginzburg2010optimal}.
    Again, this must be compared to \Cref{fig:stabD2Q5magic}.
\end{remark}

\begin{remark}[Link with the \lbmScheme{1}{2} scheme]
    Consider a 1D problem along $\xLabel$, hence $\courantNumberY = 0$.
    Moreover, consider a 1D numerical scheme by setting $\equilibriumCoefficientSymmetricY = 0$ as well.
    Taking $\equilibriumCoefficientSymmetricX = 1$, we obtain a (row-wise) \lbmScheme{1}{2} scheme, since with this choice the zero velocity does not play any role.
    For instance, \eqref{eq:neuamnnConditionD2Q5MagicNotTwo} becomes $(\courantNumberX^2-1)(1-\mu_{\xLabel}^2)\mu_{\xLabel}^2\leq 0$ for all $\mu_{\xLabel}\in[-1, 1]$, which is fulfilled under the CFL condition $|\courantNumberX|\leq 1$.
\end{remark}

\subsubsection{Transparent boundary conditions for the link TRT \lbmScheme{2}{5} scheme}\label{sec:bc2D}

\paragraph{Scalar approach}

We again extend the work of Besse \strong{et al.}, which was itself inspired by \cite{engquist1977absorbing}.
We present the way of proceeding on the $\xLabel$ axis (the same on $\yLabel$ is obtained by simply switching labels).
The initial step is to do as if the domain were vertically infinite, working on the slab $[\leftBoundary, \rightBoundary]\times \reals$, and take a Fourier transform in the $\yLabel$-direction.
This boils down to having $\fourierShift_{\yLabel} = e^{i\frequency\spaceStep}$ with $\frequency\spaceStep\in[-\pi, \pi]$ in the non-trivial part (i.e., $\Psi_2$) of the characteristic equation \eqref{eq:charEquationD2Q5Magic}:
\begin{multline}\label{eq:charEqD2Q5Magic}
    \frac{1}{2}\bigl ((\relaxationParameter-2)\equilibriumCoefficientSymmetricX + \relaxationParameter\courantNumberX\bigr )\timeShiftOperator \fourierShift^2 + \Bigl ( \timeShiftOperator^2 + (1-\relaxationParameter) + \bigl ( (\relaxationParameter - 2)(1-\equilibriumCoefficientSymmetricX) + i\relaxationParameter\courantNumberY\sin(\frequency\spaceStep) + (\relaxationParameter-2)\equilibriumCoefficientSymmetricY(\cos(\frequency\spaceStep)-1) \bigr )\timeShiftOperator\Bigr )\fourierShift\\
    +\frac{1}{2}((\relaxationParameter-2)\equilibriumCoefficientSymmetricX - \relaxationParameter\courantNumberX)\timeShiftOperator=0,
\end{multline}
where $\fourierShift$ is a placeholder for $\fourierShift_{\xLabel}$, and must be interpreted as $\fourierShift = \fourierShift(\timeShiftOperator, \frequency\spaceStep)$.
The term in $\sin(\frequency\spaceStep)$ concerns transport in the tangential (i.e. $\yLabel$) direction, whereas the one in $\cos(\frequency\spaceStep)-1$ diffusion in the same direction. 
The leading-order term in \eqref{eq:charEqD2Q5Magic} vanishes when 
\begin{equation*}
    \relaxationParameter = \frac{2\equilibriumCoefficientSymmetricX}{\equilibriumCoefficientSymmetricX + \courantNumberX}.
\end{equation*}
In all other cases, under von Neumann stability conditions, \eqref{eq:charEqD2Q5Magic} has two roots, indicated by $\stableRoot(\timeShiftOperator, \frequency\spaceStep)$ and $\unstableRoot(\timeShiftOperator, \frequency\spaceStep)$, with the first in $\unitDisk$ and the second in $\neighborhoodInfinity$ when $\timeShiftOperator\in\neighborhoodInfinity$.
Their product is independent of $\timeShiftOperator$ and $\frequency\spaceStep$ and given by
\begin{equation*}
    \productRoots_{\xLabel} = \frac{(\equilibriumCoefficientSymmetricX - \courantNumberX)\relaxationParameter-2\equilibriumCoefficientSymmetricX}{(\equilibriumCoefficientSymmetricX + \courantNumberX)\relaxationParameter-2\equilibriumCoefficientSymmetricX}.
\end{equation*}
As seen for the \lbmScheme{1}{2}, this allows to consider only $\stableRoot$ and its Laurent expansion to build boundary conditions both on the right and the left boundary.
To continue, we expand $\fourierShift(\timeShiftOperator, \frequency\spaceStep)$ into powers of small $\frequency\spaceStep$.
More precisely 
\begin{equation}\label{eq:taylorMajda}
    \fourierShift(\timeShiftOperator, \frequency\spaceStep) = \termAtOrder{\fourierShift}{0}(\timeShiftOperator) + 2i\sin(\frequency\spaceStep) \termAtOrder{\fourierShift}{1}(\timeShiftOperator) - 4\sin(\tfrac{1}{2}\frequency\spaceStep)^2\termAtOrder{\fourierShift}{2}(\timeShiftOperator) +\bigO{(\frequency\spaceStep)^3}.
\end{equation}

Into \eqref{eq:charEqD2Q5Magic}, we obtain the equation for $\termAtOrder{\fourierShift}{0}(\timeShiftOperator)$, which reads 
\begin{equation}\label{eq:zerothOrder}
    \frac{1}{2}\bigl ((\relaxationParameter-2)\equilibriumCoefficientSymmetricX + \relaxationParameter\courantNumberX\bigr )\timeShiftOperator (\termAtOrder{\fourierShift}{0}(\timeShiftOperator))^2+ \bigl ( \timeShiftOperator^2 + (1-\relaxationParameter) + (\relaxationParameter - 2)(1-\equilibriumCoefficientSymmetricX) \timeShiftOperator\bigr )\termAtOrder{\fourierShift}{0}(\timeShiftOperator)
    +\frac{1}{2}((\relaxationParameter-2)\equilibriumCoefficientSymmetricX - \relaxationParameter\courantNumberX)\timeShiftOperator=0.
\end{equation}
This is the characteristic equation of a 1D scheme and works under the generally wrong assumption that the solution hitting the boundary is constant along $\yLabel$.
We look for $\termAtOrder{\stableRoot}{0}(\timeShiftOperator) = \sum_{\indexTime = 1}^{+\infty}\termAtOrder{\coefficientsLaurentStableRootAllParities_{\indexTime}}{0}\timeShiftOperator^{-\indexTime}$, which into \eqref{eq:zerothOrder} gives the recurrence
\begin{align*}
    \termAtOrder{\coefficientsLaurentStableRootAllParities_{1}}{0} &= \tfrac{1}{2}(\relaxationParameter\courantNumberX+(2-\relaxationParameter)\equilibriumCoefficientSymmetricX), \qquad 
    \termAtOrder{\coefficientsLaurentStableRootAllParities_{2}}{0} = (2-\relaxationParameter)(1-\equilibriumCoefficientSymmetricX) \termAtOrder{\coefficientsLaurentStableRootAllParities_{1}}{0}, \qquad \text{and}\\
    \termAtOrder{\coefficientsLaurentStableRootAllParities_{\indexTime}}{0} &= (\relaxationParameter-1)\termAtOrder{\coefficientsLaurentStableRootAllParities_{\indexTime - 2}}{0} + (2-\relaxationParameter)(1-\equilibriumCoefficientSymmetricX)\termAtOrder{\coefficientsLaurentStableRootAllParities_{\indexTime-1}}{0} + \tfrac{1}{2}(-\relaxationParameter\courantNumberX + (2-\relaxationParameter)\equilibriumCoefficientSymmetricX)\sum_{k=1}^{\indexTime - 2}\termAtOrder{\coefficientsLaurentStableRootAllParities_{k}}{0}\termAtOrder{\coefficientsLaurentStableRootAllParities_{\indexTime-1-k}}{0}, \quad \indexTime\geq 3.
\end{align*}
The only cases where even terms are identically zero is $\equilibriumCoefficientSymmetricX = 1$, thus we recover the setting of the \lbmScheme{1}{2} scheme, or $\relaxationParameter = 2$ (leap-frog, or \lbmScheme{1}{2} with $\relaxationParameter = 2$).

At order $\bigO{\frequency\spaceStep}$, the equation for $\termAtOrder{\stableRoot}{1}(\timeShiftOperator)$ reads 
\begin{equation*}
    \bigl ( \timeShiftOperator^2 + (1 - \relaxationParameter) + (\relaxationParameter - 2)(1-\equilibriumCoefficientSymmetricX)\timeShiftOperator + ((\relaxationParameter-2)\equilibriumCoefficientSymmetricX+\relaxationParameter\courantNumberX)\timeShiftOperator \termAtOrder{\fourierShift}{0}(\timeShiftOperator) \bigr )\termAtOrder{\fourierShift}{1}(\timeShiftOperator) =- \tfrac{1}{2}\relaxationParameter\courantNumberY \timeShiftOperator \termAtOrder{\fourierShift}{0}(\timeShiftOperator),
\end{equation*}
which becomes \cite[Equation (4.7)]{besse2021discrete} when $\relaxationParameter = 2$.
With $\termAtOrder{\stableRoot}{1}(\timeShiftOperator) = \sum_{\indexTime = 1}^{+\infty}\termAtOrder{\coefficientsLaurentStableRootAllParities_{\indexTime}}{1}\timeShiftOperator^{-\indexTime}$, we obtain
\begin{align*}
    \termAtOrder{\coefficientsLaurentStableRootAllParities_{1}}{1} &= 0, \qquad \termAtOrder{\coefficientsLaurentStableRootAllParities_{2}}{1} = -\tfrac{1}{2}\relaxationParameter\courantNumberY\termAtOrder{\coefficientsLaurentStableRootAllParities_{1}}{0} = -\tfrac{1}{4}\relaxationParameter\courantNumberY(\relaxationParameter\courantNumberX+(2-\relaxationParameter)\equilibriumCoefficientSymmetricX), \qquad \text{and}\\
    \termAtOrder{\coefficientsLaurentStableRootAllParities_{\indexTime}}{1} &= 
    (\relaxationParameter-1)\termAtOrder{\coefficientsLaurentStableRootAllParities_{\indexTime - 2}}{1} + (2-\relaxationParameter)(1-\equilibriumCoefficientSymmetricX)\termAtOrder{\coefficientsLaurentStableRootAllParities_{\indexTime-1}}{1} + (-\relaxationParameter\courantNumberX + (2-\relaxationParameter)\equilibriumCoefficientSymmetricX)\sum_{k=1}^{\indexTime - 2}\termAtOrder{\coefficientsLaurentStableRootAllParities_{k}}{0}\termAtOrder{\coefficientsLaurentStableRootAllParities_{\indexTime-1-k}}{1} - \tfrac{1}{2}\relaxationParameter\courantNumberY \termAtOrder{\coefficientsLaurentStableRootAllParities_{\indexTime - 1}}{0}, \quad \indexTime\geq 3.
\end{align*}

Finally, at order $\bigO{(\frequency\spaceStep)^2}$, the equation for $\termAtOrder{\stableRoot}{2}(\timeShiftOperator)$ is 
\begin{multline*}
    \bigl ( ((\relaxationParameter-2)\equilibriumCoefficientSymmetricX+\relaxationParameter\courantNumberX)\timeShiftOperator\termAtOrder{\fourierShift}{0}(\timeShiftOperator) + \timeShiftOperator^2 + (1-\relaxationParameter)+(\relaxationParameter-2)(1-\equilibriumCoefficientSymmetricX)\timeShiftOperator \bigr )\termAtOrder{\fourierShift}{2}(\timeShiftOperator) = 2((2-\relaxationParameter)\equilibriumCoefficientSymmetricX-\relaxationParameter\courantNumberX)\timeShiftOperator(\termAtOrder{\fourierShift}{1}(\timeShiftOperator))^2 \\
    -
    2\relaxationParameter\courantNumberY\timeShiftOperator \termAtOrder{\fourierShift}{1}(\timeShiftOperator) +\tfrac{1}{2}(2-\relaxationParameter)\equilibriumCoefficientSymmetricY\timeShiftOperator\termAtOrder{\fourierShift}{0}(\timeShiftOperator),
\end{multline*}
so for $\termAtOrder{\stableRoot}{2}(\timeShiftOperator) = \sum_{\indexTime = 1}^{+\infty}\termAtOrder{\coefficientsLaurentStableRootAllParities_{\indexTime}}{2}\timeShiftOperator^{-\indexTime}$, we get
\begin{multline*}
    \termAtOrder{\coefficientsLaurentStableRootAllParities_1}{2} = 0, \qquad 
    \termAtOrder{\coefficientsLaurentStableRootAllParities_2}{2} = \tfrac{1}{2}(2-\relaxationParameter)\equilibriumCoefficientSymmetricY \termAtOrder{\coefficientsLaurentStableRootAllParities_1}{0}, \qquad \text{and}\\
    \termAtOrder{\coefficientsLaurentStableRootAllParities_{\indexTime}}{2} = 
    (\relaxationParameter-1)\termAtOrder{\coefficientsLaurentStableRootAllParities_{\indexTime - 2}}{2} + (2-\relaxationParameter)(1-\equilibriumCoefficientSymmetricX)\termAtOrder{\coefficientsLaurentStableRootAllParities_{\indexTime-1}}{2} + (-\relaxationParameter\courantNumberX + (2-\relaxationParameter)\equilibriumCoefficientSymmetricX)\sum_{k=1}^{\indexTime - 2}\termAtOrder{\coefficientsLaurentStableRootAllParities_{k}}{0}\termAtOrder{\coefficientsLaurentStableRootAllParities_{\indexTime-1-k}}{2} \\
    + 2(-\relaxationParameter\courantNumberX + (2-\relaxationParameter)\equilibriumCoefficientSymmetricX)\sum_{k=1}^{\indexTime - 2}\termAtOrder{\coefficientsLaurentStableRootAllParities_{k}}{1}\termAtOrder{\coefficientsLaurentStableRootAllParities_{\indexTime-1-k}}{1} 
    -2\relaxationParameter\courantNumberY \termAtOrder{\coefficientsLaurentStableRootAllParities_{\indexTime-1}}{1} + \tfrac{1}{2}(2-\relaxationParameter)\equilibriumCoefficientSymmetricY \termAtOrder{\coefficientsLaurentStableRootAllParities_{\indexTime-1}}{0}, \quad \indexTime\geq 3.
\end{multline*}

Therefore, the (approximate) transparent boundary conditions that we propose read, for $\indexTime\geq 1$
\begin{align}
    \phi_{\numberGridPoints + 1, k}^{\indexTime} = &\sum_{m = 1}^{\indexTime} \bigl (\termAtOrder{\coefficientsLaurentStableRootAllParities_{m}}{0} \phi_{\numberGridPoints, k}^{\indexTime - m}  + \termAtOrder{\coefficientsLaurentStableRootAllParities_{m}}{1} (\phi_{\numberGridPoints, k+1}^{\indexTime - m}-\phi_{\numberGridPoints, k-1}^{\indexTime - m}) + \termAtOrder{\coefficientsLaurentStableRootAllParities_{m}}{2} (\phi_{\numberGridPoints, k+1}^{\indexTime - m} - 2 \phi_{\numberGridPoints, k}^{\indexTime - m} +\phi_{\numberGridPoints, k-1}^{\indexTime - m})  \bigr ), \label{eq:BC2DRight}\\
    \phi_{0, k}^{\indexTime} = \frac{1}{\productRoots_{\xLabel}}&\sum_{m = 1}^{\indexTime} \bigl (\termAtOrder{\coefficientsLaurentStableRootAllParities_{m}}{0} \phi_{1, k}^{\indexTime - m}  + \termAtOrder{\coefficientsLaurentStableRootAllParities_{m}}{1} (\phi_{1, k+1}^{\indexTime - m}-\phi_{1, k-1}^{\indexTime - m}) + \termAtOrder{\coefficientsLaurentStableRootAllParities_{m}}{2} (\phi_{1, k+1}^{\indexTime - m} - 2 \phi_{1, k}^{\indexTime - m} +\phi_{1, k-1}^{\indexTime - m})  \bigr ), \quad k \in \integerInterval{1}{\numberGridPointsOnY},\nonumber
\end{align}
for the placeholder $\phi\in\{\conservedMoment, \antiSymmetricMomentXLetter, \symmetricMomentXLetter, \antiSymmetricMomentYLetter, \symmetricMomentYLetter\}$.

Compared to the systemic approach, which would be rather involved to devise in this case, we see that the current scalar approach is significantly more expensive from the standpoint of storage cost in this 2D framework.
Here, five moments need to be stored on the ``frame'' surrounding the boundary, whereas a systemic approach only needs to store one of them, plus five coefficients independent of the considered cell.

\begin{summarybox}{Summary on the construction of transparent boundary conditions in 2D}
    \begin{small}
        For a given boundary:
        \begin{enumerate}
            \item Consider the characteristic equation as the problem were infinite in the direction tangential to the boundary, and take a Fourier transform in this direction, see \eqref{eq:charEqD2Q5Magic}.
            \item Taylor expand quantities for small tangential frequencies \eqref{eq:taylorMajda} up to a desired order and proceed analogously to the 1D case order-by-order.
            \item Invert the tangential Fourier transform up to the desired order to yield the boundary conditions, e.g. \eqref{eq:BC2DRight}.
        \end{enumerate}
    \end{small}
\end{summarybox}

\paragraph{Some asymptotic expansions}

We do not delve into asymptotics for $\termAtOrder{\coefficientsLaurentStableRootAllParities_{\indexTime}}{0}$ and $\termAtOrder{\coefficientsLaurentStableRootAllParities_{\indexTime}}{1}$. 
We conjecture $\termAtOrder{\coefficientsLaurentStableRootAllParities_{\indexTime}}{h}$ for $h=0, 1$ geometrically decrease when $\relaxationParameter\in(0, 2)$ with a tame factor $\sim \indexTime^{-3/2+h}$ at leading order. 
The case $\relaxationParameter = 2$ is detailed in \cite{besse2021discrete}.
We solely study $\termAtOrder{\coefficientsLaurentStableRootAllParities_{\indexTime}}{2}$ for $\relaxationParameter = 2$, for which only the trend of the tame factor $\sim \indexTime^{1/2}$ was conjectured and numerically demonstrated by Besse and collaborators, with a rigorous proof left for future works, which we propose below.
\begin{lemma}
    Let $\relaxationParameter = 2$ and the $\ell^2$-stability condition $|\courantNumberX|+|\courantNumberY|<1$ be fulfilled.
    Let $\vartheta_{\xLabel}=\arccos (1-2\courantNumberX^2)$. 
    Then, the coefficients $\termAtOrder{\coefficientsLaurentStableRootAllParities_{\indexTime}}{2}$ for the $\xLabel$-axis are such that $\termAtOrder{\coefficientsLaurentStableRootAllParities_{\indexTime}}{2} = 0$ if $\indexTime$ is even and follow the asymptotics 
    \begin{equation}\label{eq:asymptMissedBesse}
        \termAtOrder{\coefficientsLaurentStableRootAllParities_{2k + 1}}{2} = \frac{2}{\sqrt{\pi}}\frac{\courantNumberY^2}{\courantNumberX^{1/2}(1-\courantNumberX^2)^{3/4}} \sin\bigl ( \bigl ( k+\tfrac{1}{2}\bigr )\vartheta_{\xLabel} - \tfrac{\pi}{4}\bigr ) k^{1/2} + \bigO{k^{-1/2}}.
    \end{equation}
\end{lemma}
\begin{proof}
    The fact that $\termAtOrder{\coefficientsLaurentStableRootAllParities_{\indexTime}}{2} = 0$ if $\indexTime$ is even can be seen by the recurrence equation.
    From Equation (4.12) in Besse \strong{et al.}, we have that $\termAtOrder{\stableRoot}{2}(\timeShiftOperator^{-1}) = 4\timeShiftOperator^3\courantNumberX\courantNumberY^2 (\timeShiftOperator^4 + 2(2\courantNumberX^2-1)\timeShiftOperator^2 + 1)^{-3/2}$.
    Equation \eqref{eq:asymptMissedBesse} follows from the procedure illustrated in \Cref{sec:backAsympt}.
\end{proof}

\newcommand{\testCaseDescription}[2]{$\langle #1, #2\rangle$}

\paragraph{Numerical experiments}

\begin{figure}
    \begin{center}
        \includegraphics[width=1\textwidth]{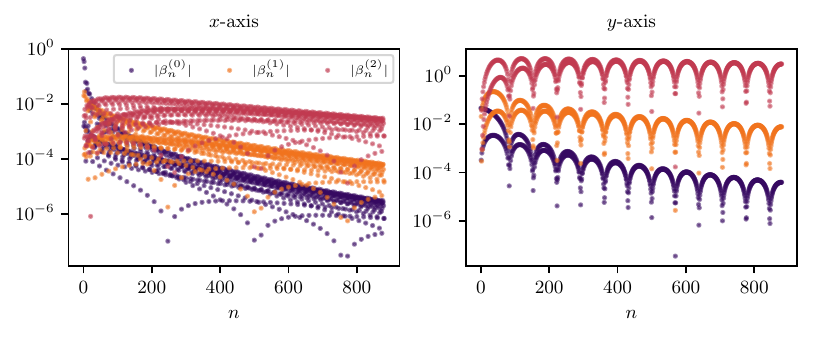}
    \end{center}
    \caption{\label{fig:D2Q5-coeff-asymptotic} Absolute value of the coefficients $\termAtOrder{\coefficientsLaurentStableRootAllParities_{\indexTime}}{0}$, $\termAtOrder{\coefficientsLaurentStableRootAllParities_{\indexTime}}{1}$, and $\termAtOrder{\coefficientsLaurentStableRootAllParities_{\indexTime}}{2}$ for the $\xLabel$-axis (on the left) and the $\yLabel$-axis (on the right).}
\end{figure}

\begin{figure}
    \begin{center}
        \includegraphics[width=1\textwidth]{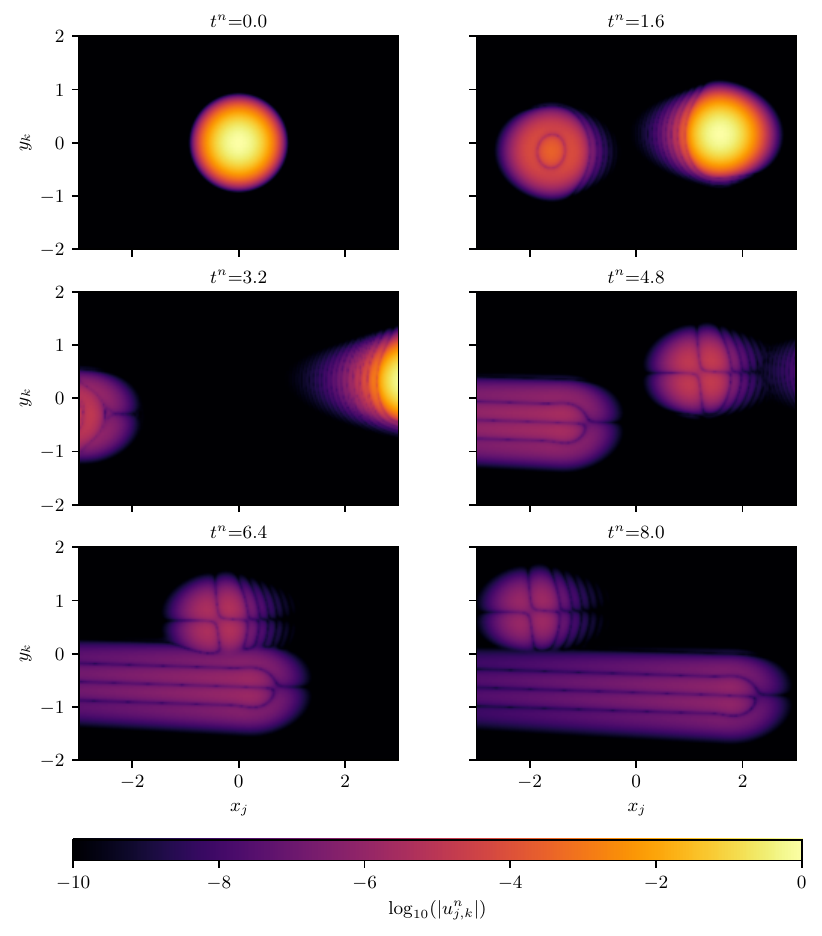}
    \end{center}
    \caption{\label{fig:D2Q5-magic-0-0}Solution for the TRT \lbmScheme{2}{5} with magic parameter equal to $\tfrac{1}{4}$ at $\relaxationParameter = 1.99$ at different time snapshots, using \testCaseDescription{0}{0}.}
\end{figure}

\begin{figure}
    \begin{center}
        \includegraphics[width=1\textwidth]{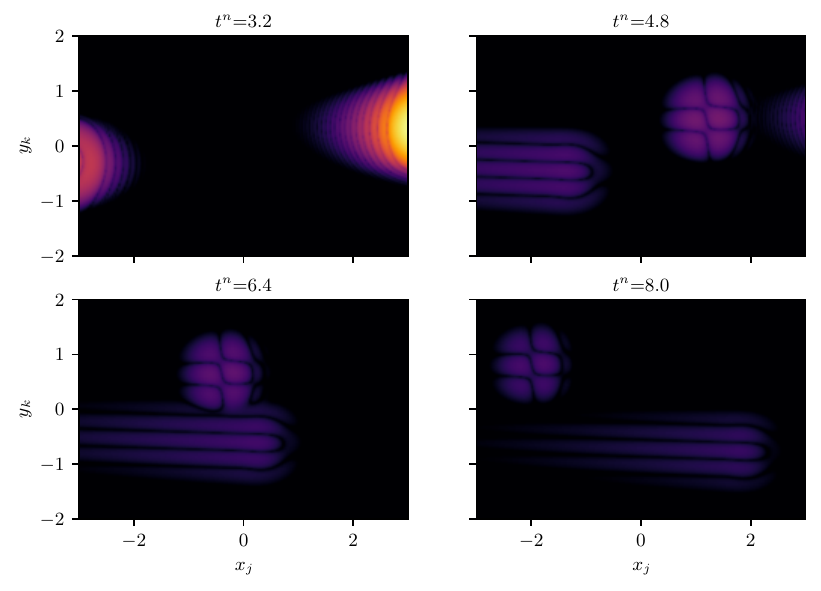}
    \end{center}
    \caption{\label{fig:D2Q5-magic-1-1}Solution for the TRT \lbmScheme{2}{5} with magic parameter equal to $\tfrac{1}{4}$ at $\relaxationParameter = 1.99$ at different time snapshots, using \testCaseDescription{1}{1}. Color-bar shared with \Cref{fig:D2Q5-magic-0-0}.}
\end{figure}

\begin{figure}
    \begin{center}
        \includegraphics[width=1\textwidth]{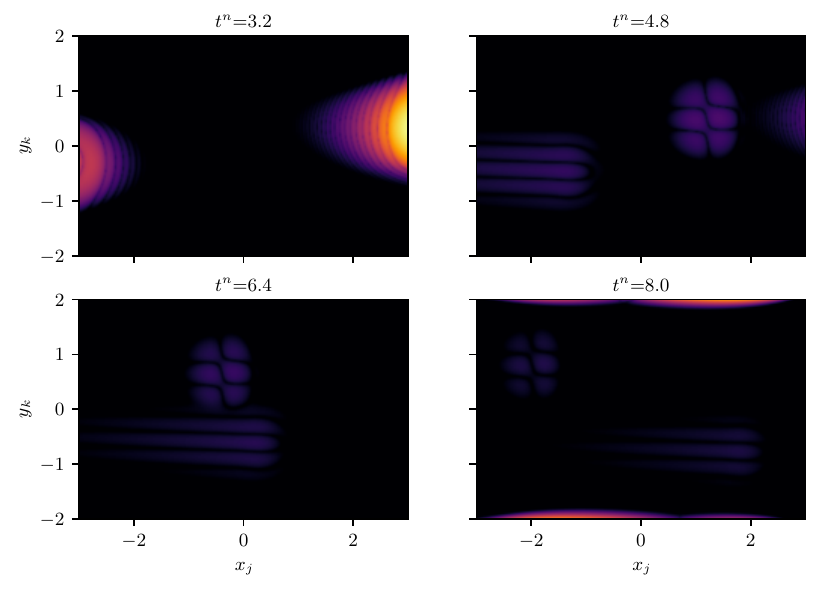}
    \end{center}
    \caption{\label{fig:D2Q5-magic-2-2}Solution for the TRT \lbmScheme{2}{5} with magic parameter equal to $\tfrac{1}{4}$ at $\relaxationParameter = 1.99$ at different time snapshots, using \testCaseDescription{2}{2}. Color-bar shared with \Cref{fig:D2Q5-magic-0-0}.}
\end{figure}

\begin{figure}
    \begin{center}
        \includegraphics[width=1\textwidth]{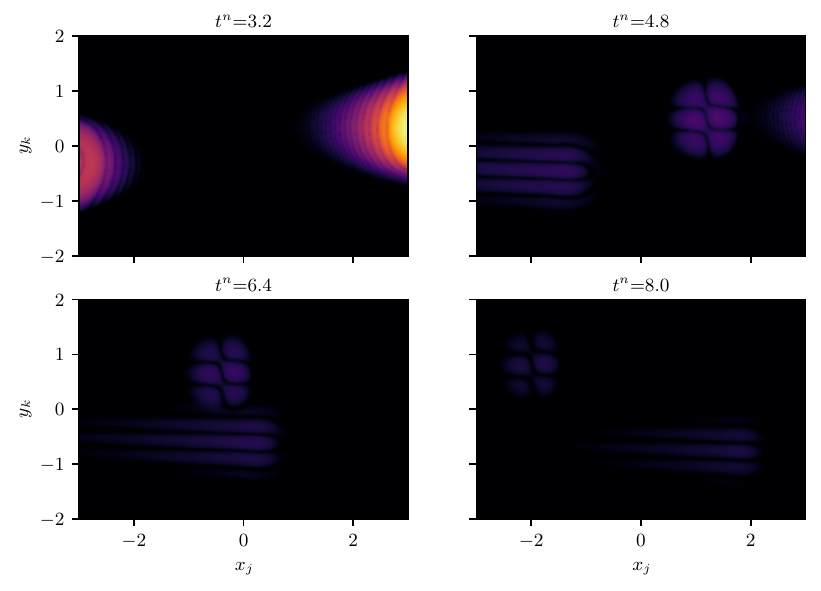}
    \end{center}
    \caption{\label{fig:D2Q5-magic-2-1}Solution for the TRT \lbmScheme{2}{5} with magic parameter equal to $\tfrac{1}{4}$ at $\relaxationParameter = 1.99$ at different time snapshots, using \testCaseDescription{2}{1}. Color-bar shared with \Cref{fig:D2Q5-magic-0-0}.}
\end{figure}

Numerical simulations are conducted using $\advectionVelocity_{\xLabel} = 1$, $\advectionVelocity_{\yLabel} = \tfrac{1}{10}$, $\latticeVelocity = 2(|\advectionVelocity_{\xLabel}|+|\advectionVelocity_{\yLabel}|) =\tfrac{11}{5}$ on a domain delimited by $\leftBoundary = -3$, $\rightBoundary = 3$, $\bottomBoundary = -2$, and $\topBoundary = 2$.
The initial datum is $\conservedMoment^{\circ}(\vectorial{\spaceVariable}) = \cos^{10}(\pi r(\vectorial{\spaceVariable}) /2)\indicatorFunction{[0, 1)}(r(\vectorial{\spaceVariable}))$, where $r(\vectorial{\spaceVariable})=\sqrt{\xLabel^2+\yLabel^2}$.
The mesh features $\numberGridPoints = 300$ and $\numberGridPointsOnY = 200$.
We systematically consider $\relaxationParameter = 1.99$.

We chose this regime featuring $\courantNumberX\gg \courantNumberY$ in such a way that it is dominated by the transport along the $\xLabel$-axis.
This entails that the upper and lower boundaries experience strong ``shear'' (or tangential) movements of the numerical solution.
The absolute value of the coefficients used in the boundary conditions are shown in \Cref{fig:D2Q5-coeff-asymptotic}.
In terms of amplitude, the terms for the $\yLabel$-axis are larger than those for the $\xLabel$-axis, by orders of magnitude. 
The damping of the former is slower than that for the latter.
It is overall quite weak as $\relaxationParameter$ is very close to two.
In terms of oscillations, we see that coefficients for the $\yLabel$-axis oscillate faster than those of $\xLabel$-axis.

In the spirit of the numerical simulations shown in \cite{besse2021discrete}, we switch tangential corrections in the boundary conditions on and off.
The notation \testCaseDescription{o_{\xLabel}}{o_{\yLabel}} indicates that we have set $\termAtOrder{\coefficientsLaurentStableRootAllParities_{\indexTime}}{h} \equiv 0$ for $h>o_{\xLabel}$ for the $\xLabel$-axis, and $\termAtOrder{\coefficientsLaurentStableRootAllParities_{\indexTime}}{h} \equiv 0$ for $h>o_{\yLabel}$ for the $\yLabel$-axis.
On \Cref{fig:D2Q5-magic-0-0}, \ref{fig:D2Q5-magic-1-1}, and \ref{fig:D2Q5-magic-2-2}, we see that---as expected---increasing the number of orders in the tangential correction reduces the amplitude of the reflected waves. 
However, in the case \testCaseDescription{2}{2}, we observe that instabilities attached to the lower and upper boundaries develop. 
They are due to the coupling between $\xLabel$ and $\yLabel$ from the corners and similar to those observed by Besse and collaborators.
The fact that they appear attached to the $\yLabel$ boundaries can be explained by the differences between $\termAtOrder{\coefficientsLaurentStableRootAllParities_{\indexTime}}{2}$ along different axes, see \Cref{fig:D2Q5-coeff-asymptotic}.
Stability is restored taking \testCaseDescription{2}{1}, see \Cref{fig:D2Q5-magic-2-1}.

Compared to the 1D setting, where ``bad'' boundary conditions reflected packets essentially preserving the size of their support, we see a radically different behavior in 2D.
The reflected waves caused by the approximations along $\yLabel$ also remain lodged at the boundary (with a slow damping as $\relaxationParameter$ is close to two).
This is caused by the global-in-time character of boundary conditions, which thus ``feel'' the approximation in $\yLabel$ eventually in time.

Finally, we have tested our transparent boundary conditions \testCaseDescription{0}{0} against the 2D generalization of \eqref{eq:kineticBoundaryConditions}.
Results, not shown here, are better with the former approach by several orders of magnitude.

\section{Systems}\label{sec:systems}

\subsection{Target partial differential equations}

Consider
\begin{equation}\label{eq:nonLinearSystem1D}
    \begin{dcases}
        \begin{aligned}
            &\partial_{\timeVariable}\vectorial{\conservedMoment}(\timeVariable, \spaceVariable) +  \partial_{\spaceVariable}(\vectorial{\fluxLetter}(\vectorial{\conservedMoment}(\timeVariable, \spaceVariable)))  = \vectorial{0}, \qquad &\timeVariable>0, \quad &\spaceVariable\in\reals, \\
            &\vectorial{\conservedMoment}(0, \spaceVariable)  = \vectorial{\conservedMoment}^{\circ}(\spaceVariable), \qquad & &\spaceVariable\in\reals,
        \end{aligned}
    \end{dcases}
\end{equation}
with $\vectorial{\conservedMoment}\in\reals^{\numberConservationLaws}$ and given fluxes $\vectorial{\fluxLetter}:\reals^{\numberConservationLaws}\to \reals^{\numberConservationLaws}$.
Note that, although we investigate the 1D setting for the sake of simplicity, the extension to multi-dimensional problems can be done in the spirit of \Cref{sec:2DScalar}.
Considering a reference state $\referenceState$ independent of space and time, and that $\vectorial{\conservedMoment}(\timeVariable, \spaceVariable) = \referenceState  + \delta \vectorial{\conservedMoment}(\timeVariable, \spaceVariable)$, where the acoustic fluctuation $\delta \vectorial{\conservedMoment}(\timeVariable, \spaceVariable)$ is assumed to be small, we linearize \eqref{eq:nonLinearSystem1D} to yield $\partial_{\timeVariable}\delta\vectorial{\conservedMoment} + \vectorial{\fluxLetter}'(\referenceState)\partial_{\spaceVariable}\delta \vectorial{\conservedMoment} = \vectorial{0}$.
Abusing notation mixing $\vectorial{\conservedMoment}$ and $\delta \vectorial{\conservedMoment}$, we consider the acoustic system
\begin{equation}\label{eq:linearizedSystem1D}
    \begin{dcases}
        \begin{aligned}
            &\partial_{\timeVariable}\vectorial{\conservedMoment}(\timeVariable, \spaceVariable) +  \vectorial{\fluxLetter}'(\referenceState) \partial_{\spaceVariable}\vectorial{\conservedMoment}(\timeVariable, \spaceVariable)  = \vectorial{0}, \qquad &\timeVariable>0, \quad &\spaceVariable\in\reals, \\
            &\vectorial{\conservedMoment}(0, \spaceVariable)  = \vectorial{\conservedMoment}^{\circ}(\spaceVariable), \qquad & &\spaceVariable\in\reals.
        \end{aligned}
    \end{dcases}
\end{equation}
We assume that the system in \eqref{eq:linearizedSystem1D} is hyperbolic, namely that $\vectorial{\fluxLetter}'(\referenceState)$ is diagonalizable with real eigenvalues, denoted $\eigenvalueLinearized_1, \dots, \eigenvalueLinearized_{\numberConservationLaws}$.
The corresponding eigenvectors $\eigenvectorLinearized_1, \dots, \eigenvectorLinearized_{\numberConservationLaws} \in\reals^{\numberConservationLaws}$ of $\vectorial{\fluxLetter}'(\referenceState)$ are gathered column-wise in a non-singular matrix $\matrixEigenvectorsLinearized\definitionEquality[\eigenvectorLinearized_1 \,|\, \cdots \,|\,  \eigenvectorLinearized_{\numberConservationLaws}]\in\linearGroup{\numberConservationLaws}{\reals}$.

The example we consistently present is the shallow water system, where $\vectorial{\conservedMoment} = \transpose{(\height, \height\velocity)}$ and $\vectorial{\fluxLetter}(\vectorial{\conservedMoment}) = \transpose{(\height\velocity, \height\velocity^2 + \tfrac{1}{2}\gravity\height^2)}$.
This entails 
\begin{equation*}
    \vectorial{\fluxLetter}'(\referenceState) = 
    \begin{pmatrix}
        0 & 1\\
        -\referenceStateMarked{\velocity}^2 + \gravity\referenceStateMarked{\height} & 2\referenceStateMarked{\velocity}
    \end{pmatrix}, \qquad 
    \eigenvalueLinearized_1 = \referenceStateMarked{\velocity} - \soundSpeed, \quad 
    \eigenvalueLinearized_2 = \referenceStateMarked{\velocity} + \soundSpeed, \qquad
    \matrixEigenvectorsLinearized = 
    \begin{pmatrix}
        1 & 1 \\
        \referenceStateMarked{\velocity} -\soundSpeed & \referenceStateMarked{\velocity} + \soundSpeed
    \end{pmatrix},
\end{equation*}
with the speed of the sound $\soundSpeed = \sqrt{\gravity \referenceStateMarked{\height}}$, and $\gravity>0$ the gravity constant.

\subsection{Numerical schemes and their properties}

\subsubsection{\lbmScheme{1}{3} scheme}

We first consider a three-velocities ``monolithic'' scheme \cite{van2010study}, which can be designed only for problems with $\numberConservationLaws = 2$ where the first flux is the one of the shallow water system, i.e. equals the second conserved quantity.
We present the scheme for the non-linear problem \eqref{eq:nonLinearSystem1D}.
\begin{itemize}
    \item \strong{Relaxation}.
    \begin{equation*}
        \conservedMoment_{\indexSpace}^{\indexTime, \collided} = \conservedMoment_{\indexSpace}^{\indexTime}, \qquad \nonConservedMoment_{\indexSpace}^{\indexTime, \collided} = \nonConservedMoment_{\indexSpace}^{\indexTime}, \qquad \nonNonConservedMoment_{\indexSpace}^{\indexTime, \collided} = (1-\relaxationParameter)\nonNonConservedMoment_{\indexSpace}^{\indexTime} + \relaxationParameter \fluxLetter_2(\conservedMoment_{\indexSpace}^{\indexTime}, \nonConservedMoment_{\indexSpace}^{\indexTime}), \qquad \indexSpace\in\integerInterval{0}{\numberGridPoints + 1}.
    \end{equation*}
    Here, $\conservedMoment_{\indexSpace}^{\indexTime}\approx \height(\timeGridPoint{\indexTime}, \spaceGridPoint{\indexSpace})$ and $\nonConservedMoment_{\indexSpace}^{\indexTime}\approx \height\velocity(\timeGridPoint{\indexTime}, \spaceGridPoint{\indexSpace})$ in the case of the shallow water system.
    The linearized equation is simulated replacing $\fluxLetter_2(\conservedMoment_{\indexSpace}^{\indexTime}, \nonConservedMoment_{\indexSpace}^{\indexTime})$ by $\transpose{\canonicalBasisVector{2}} \vectorial{\fluxLetter}'(\referenceState) \transpose{(\conservedMoment_{\indexSpace}^{\indexTime}, \nonConservedMoment_{\indexSpace}^{\indexTime})}$.
    \item \strong{Transport}. After having computed $\distributionFunctionLetter_{0, \indexSpace}^{\indexTime, \collided}= \conservedMoment_{\indexSpace}^{\indexTime, \collided} - \tfrac{1}{\latticeVelocity^2}\nonNonConservedMoment_{\indexSpace}^{\indexTime, \collided}$ and $\distributionFunctionLetter_{\pm, \indexSpace}^{\indexTime, \collided}= \tfrac{1}{2\latticeVelocity}(\pm\nonConservedMoment_{\indexSpace}^{\indexTime, \collided} +\tfrac{1}{\latticeVelocity} \nonNonConservedMoment_{\indexSpace}^{\indexTime, \collided} )$ for $\indexSpace\in\integerInterval{0}{\numberGridPoints + 1}$ (remark that in order to obtain the expected flux for the first unknown, the presence of $\latticeVelocity=\spaceStep/\timeStep$ in the previous equalities is essential), we have 
    \begin{equation*}
        \distributionFunctionLetter_{0, \indexSpace}^{\indexTime + 1}= \distributionFunctionLetter_{0, \indexSpace}^{\indexTime, \collided} 
        \qquad \text{and}\qquad
        \distributionFunctionLetter_{\pm, \indexSpace}^{\indexTime + 1}= \distributionFunctionLetter_{\pm, \indexSpace \mp 1}^{\indexTime, \collided}, \qquad \indexSpace\in\integerInterval{1}{\numberGridPoints}.
    \end{equation*}
    Then, we obtain $\conservedMoment_{\indexSpace}^{\indexTime+1} = \distributionFunctionLetter_{0, \indexSpace}^{\indexTime + 1} + \distributionFunctionLetter_{+, \indexSpace}^{\indexTime + 1} +  \distributionFunctionLetter_{-, \indexSpace}^{\indexTime + 1}$, $\nonConservedMoment_{\indexSpace}^{\indexTime+1} = \latticeVelocity(\distributionFunctionLetter_{+, \indexSpace}^{\indexTime + 1} -  \distributionFunctionLetter_{-, \indexSpace}^{\indexTime + 1})$, and $\nonNonConservedMoment_{\indexSpace}^{\indexTime+1} = \latticeVelocity^2(\distributionFunctionLetter_{+, \indexSpace}^{\indexTime + 1} +  \distributionFunctionLetter_{-, \indexSpace}^{\indexTime + 1})$ for $ \indexSpace\in\integerInterval{1}{\numberGridPoints}$.
\end{itemize}

\begin{proposition}[Consistency and von Neumann stability of the \lbmScheme{1}{3} scheme for the linearized shallow water system on $\relatives$]\label{prop:stabilityD1Q3ShallowWater}
     The \lbmScheme{1}{3} scheme for the linearized shallow water system is such that there exist two unique eigenvalues $\timeShiftOperator_{\textnormal{phy}, \pm}(\frequency\spaceStep)$ of $\schemeMatrixBulkFourier(\frequency\spaceStep)$ such that 
     \begin{multline}\label{eq:expPhySWD1Q3}
        \timeShiftOperator_{\textnormal{phy}, \pm} (\frequency\spaceStep)= 
        1 - i(\referenceStateMarked{\velocity} \pm \soundSpeed) \frequency\timeStep
        \\
        +\biggl (
        -\frac{(\referenceStateMarked{\velocity} \pm \soundSpeed)^2}{2}
        -
        \frac{\soundSpeed^4 - \latticeVelocity^2\soundSpeed^2 \pm 3\soundSpeed^3\referenceStateMarked{\velocity}/\referenceStateMarked{\height}^3\mp\latticeVelocity^2\soundSpeed\referenceStateMarked{\velocity}+3\soundSpeed^2\referenceStateMarked{\velocity}^2\pm\soundSpeed\referenceStateMarked{\velocity}^3}{2\latticeVelocity^2\soundSpeed^2}
        \Bigl ( \frac{1}{2}-\frac{1}{\relaxationParameter}\Bigr )
        \Bigr )
        (\frequency\timeStep)^2
        +\bigO{(\frequency\timeStep)^3}.
     \end{multline}
    This entails that the scheme is first-order accurate when $\relaxationParameter\in(0, 2)$ and second-order accurate when $\relaxationParameter = 2$.

    The scheme is stable in the sense of von Neumann, namely $\spectrum(\schemeMatrixBulkFourier(e^{i\frequency\spaceStep}))\subset \closedUnitDisk$ for $\frequency\spaceStep\in[-\pi, \pi]$, if and only if, taking $\relaxationParameter\in (0, 2]$
    \begin{equation}\label{eq:stabConditionD1Q3ShallowWater}
        \latticeVelocity\geq \max(|\referenceStateMarked{\velocity}-\soundSpeed|, |\referenceStateMarked{\velocity}+\soundSpeed|)\qquad \text{and}\qquad |\referenceStateMarked{\velocity}|\leq \soundSpeed.
    \end{equation}
\end{proposition}
The proof of the stability condition can be found in \Cref{app:prop:stabilityD1Q3ShallowWater}, while consistency can be verified with any computer algebra system.
The first stability condition is quite obvious, as it states that the lattice velocity must be larger than the fastest wave.
The second one says that the scheme can only be employed in a subsonic regime.
This latter fact, which is linked with numerical diffusion for $\relaxationParameter \in (0, 2)$, can be understood from the standpoint of Finite Volume methods by considering the relaxation scheme, i.e. $\relaxationParameter = 1$.
In this context, the scheme on $\vectorial{\conservedMoment} = \transpose{(\conservedMoment, \nonConservedMoment)}$ reads
\begin{equation*}
    \vectorial{\conservedMoment}_{\indexSpace}^{\indexTime+1} = \vectorial{\conservedMoment}_{\indexSpace}^{\indexTime}-\tfrac{\timeStep}{\spaceStep} (\tilde{\vectorial{\fluxLetter}}(\vectorial{\conservedMoment}_{\indexSpace}^{\indexTime}, \vectorial{\conservedMoment}_{\indexSpace+1}^{\indexTime}) - \tilde{\vectorial{\fluxLetter}}(\vectorial{\conservedMoment}_{\indexSpace-1}^{\indexTime}, \vectorial{\conservedMoment}_{\indexSpace}^{\indexTime}))
\end{equation*}
with numerical flux function 
\begin{equation*}
    \tilde{\vectorial{\fluxLetter}}(\vectorial{\conservedMoment}_{\ell}, \vectorial{\conservedMoment}_{r}) = \tfrac{1}{2}(\vectorial{\fluxLetter}(\vectorial{\conservedMoment}_{\ell})+\vectorial{\fluxLetter}(\vectorial{\conservedMoment}_{r})) - \tfrac{1}{2}\matricial{D}(\vectorial{\conservedMoment}_{r}-\vectorial{\conservedMoment}_{\ell}) \qquad \text{and}\qquad
    \matricial{D}= 
    \begin{pmatrix}
        (\soundSpeed^2-\referenceStateMarked{\velocity}^2)\tfrac{\timeStep}{\spaceStep} & 2 \referenceStateMarked{\velocity}  \tfrac{\timeStep}{\spaceStep} \\
        0 & \tfrac{\spaceStep}{\timeStep} 
    \end{pmatrix}.
\end{equation*}
This is similar to a Rusanov scheme, see \cite[Chapter 10]{toro2013riemann}, where one uses a matrix $\matricial{D}$ instead of a scalar\footnote{If we had $\matricial{D} = \latticeVelocity\identityMatrix{2}$, we would face the Lax-Friedrichs scheme.}.
The matrix $\matricial{D}$ yields a diffusion term, and requesting that this matrix be positive boils down to considering the subsonic regime.
Stability conditions \eqref{eq:stabConditionD1Q3ShallowWater} were both numerically observed in \cite[Figure 7]{van2010study}, where the authors remarked that the second one comes from constraining eigenvalues in $\closedUnitDisk$ in the vicinity of $\frequency\spaceStep = 0$. 

For this numerical scheme, damping and dissipation can be studied with analogous procedures to the \lbmScheme{1}{2} scheme, see \Cref{sec:D1Q2PresentationAndProp}, although we do not detail them for the sake of room.
We terminate the exposition of the features of this scheme by specifying the propagation pattern of the spurious root of the characteristic equation $\determinant(\timeShiftOperator\identityMatrix{3}-\schemeMatrixBulkFourier(e^{i\frequency\spaceStep})) = 0$ for low frequencies.
Its expansion reads
\begin{equation}\label{eq:spuriouSWD1Q3}
    \timeShiftOperator_{\textnormal{spu}}(\frequency\spaceStep) = \underbrace{(1-\relaxationParameter)}_{\textnormal{damping}} (\underbrace{1 + i2 \referenceStateMarked{\velocity}\frequency\timeStep}_{\textnormal{propagation}} + \bigO{(\frequency\timeStep)^2})
\end{equation}
so the damping is geometrical with common ratio $1-\relaxationParameter$ and propagation at speed $-2 \referenceStateMarked{\velocity}$.

\begin{summarybox}{Summary of the main properties of the \lbmScheme{1}{3} for the linearized shallow water}
    \begin{small}
        \begin{center}
            \begin{tabular}{|ll!{\vrule width 1pt}c|c!{\vrule width 1pt}}
                \cline{3-4}
                \multicolumn{2}{l|}{} & $\relaxationParameter\in(0, 2)$ & \multicolumn{1}{c|}{$\relaxationParameter = 2$}\\
                \cline{1-2}\noalign{\global\arrayrulewidth=1pt}
                \cline{3-4}
                \noalign{\global\arrayrulewidth=.4pt}
                \multicolumn{2}{|l!{\vrule width 1pt}}{Consistency (see \Cref{prop:stabilityD1Q3ShallowWater})} & 1st order & 2nd order \\
                \hline
                \multicolumn{2}{|l!{\vrule width 1pt}}{Von Neumann stability condition (see \Cref{prop:stabilityD1Q3ShallowWater})} & \multicolumn{2}{c!{\vrule width 1pt}}{$\latticeVelocity\geq \max(|\referenceStateMarked{\velocity}-\soundSpeed|, |\referenceStateMarked{\velocity}+\soundSpeed|)\quad \text{and}\quad |\referenceStateMarked{\velocity}|\leq \soundSpeed$} \\
                \hline
                \multirow{2}{*}{Propagation speed at low freq. ($\frequency\spaceStep\approx 0$)} & Physical symb. \eqref{eq:expPhySWD1Q3} & \multicolumn{2}{c!{\vrule width 1pt}}{$\referenceStateMarked{\velocity}\pm\soundSpeed$}\\\cline{2-4}
                                             & Spurious symb. \eqref{eq:spuriouSWD1Q3} & \multicolumn{2}{c!{\vrule width 1pt}}{$-2\referenceStateMarked{\velocity}$}\\
                \cline{1-2}\noalign{\global\arrayrulewidth=1pt}
                \cline{3-4}\noalign{\global\arrayrulewidth=.4pt}
                \end{tabular}
        \end{center}        
    \end{small}
\end{summarybox}

\subsubsection{Vectorial \lbmSchemeVectorial{1}{2}{\numberConservationLaws} scheme}

This numerical scheme has been introduced in \cite{graille2014approximation}, based on ideas close to those of kinetic schemes, see \cite{aregba2000discrete}.
As shown below, the main asset of this scheme is being able to simulate any system regardless of its size $\numberConservationLaws$ and the structure of the fluxes.
\begin{itemize}
    \item \strong{Relaxation}.
    \begin{equation}
        \vectorial{\conservedMoment}_{\indexSpace}^{\indexTime, \collided} = \vectorial{\conservedMoment}_{\indexSpace}^{\indexTime}\qquad \text{and} \qquad \vectorial{\nonConservedMoment}_{\indexSpace}^{\indexTime, \collided} = (1-\omega)\vectorial{\nonConservedMoment}_{\indexSpace}^{\indexTime} + \omega {\vectorial{\fluxLetter}(\vectorial{\conservedMoment}_{\indexSpace}^{\indexTime})}/{\latticeVelocity}, \qquad \indexSpace\in\integerInterval{0}{\numberGridPoints + 1}.
    \end{equation}
    The linearized version replaces $\vectorial{\fluxLetter}(\vectorial{\conservedMoment}_{\indexSpace}^{\indexTime})$ by $\vectorial{\fluxLetter}'(\referenceStateMarked{\vectorial{\conservedMoment}})\vectorial{\conservedMoment}_{\indexSpace}^{\indexTime}$.
    Here, $\vectorial{\conservedMoment}_{\indexSpace}^{\indexTime}\in\reals^{\numberConservationLaws}$ and $\vectorial{\conservedMoment}_{\indexSpace}^{\indexTime}\approx\vectorial{\conservedMoment}(\timeGridPoint{\indexTime}, \spaceGridPoint{\indexSpace})$.
    \item \strong{Transport}. After having computed $\vectorial{\distributionFunctionLetter}_{\pm, \indexSpace}^{\indexTime, \collided} =  \tfrac{1}{2}(\vectorial{\conservedMoment}_{\indexSpace}^{\indexTime, \collided} \pm \vectorial{\nonConservedMoment}_{\indexSpace}^{\indexTime, \collided} )$ for $\indexSpace\in\integerInterval{0}{\numberGridPoints + 1}$, we have 
    \begin{equation}
        \vectorial{\distributionFunctionLetter}_{\pm, \indexSpace}^{\indexTime + 1} = \vectorial{\distributionFunctionLetter}_{\pm, \indexSpace \mp 1}^{\indexTime, \collided}, \qquad \indexSpace\in\integerInterval{1}{\numberGridPoints}.
    \end{equation}
    Then, we obtain $\vectorial{\conservedMoment}_{\indexSpace}^{\indexTime+1} = \vectorial{\distributionFunctionLetter}_{+, \indexSpace}^{\indexTime + 1} +  \vectorial{\distributionFunctionLetter}_{-, \indexSpace}^{\indexTime + 1}$ and $\vectorial{\nonConservedMoment}_{\indexSpace}^{\indexTime+1} = \vectorial{\distributionFunctionLetter}_{+, \indexSpace}^{\indexTime + 1} -  \vectorial{\distributionFunctionLetter}_{-, \indexSpace}^{\indexTime + 1}$ for $ \indexSpace\in\integerInterval{1}{\numberGridPoints}$.
\end{itemize}

We have the following fundamental lemma, which tells us that the \lbmSchemeVectorial{1}{2}{\numberConservationLaws} behaves as a ``superposition'' of $\numberConservationLaws$ \lbmScheme{1}{2} schemes described in \Cref{sec:D1Q2PresentationAndProp}, with its properties inherited at the scale of each wave of hyperbolic system.
\begin{lemma}[Factorization of the characteristic equation for the \lbmSchemeVectorial{1}{2}{\numberConservationLaws} scheme]\label{lemma:splittingCharVectorialD1Q2}
    Introduce the Courant number of each wave: $\courantNumber_{k}\definitionEquality\eigenvalueLinearized_k/\latticeVelocity$ for $k\in\integerInterval{1}{\numberConservationLaws}$.
    For the vectorial \lbmSchemeVectorial{1}{2}{\numberConservationLaws} scheme tackling the linearized problem as above, we have
    \begin{equation*}
        \determinant(\timeShiftOperator\identityMatrix{2\numberConservationLaws} - \schemeMatrixBulkFourier(\fourierShift)) = 
        \prod_{k = 1}^{\numberConservationLaws}\Bigl ( \tfrac{1}{2}((1-\courantNumber_k)\relaxationParameter-2)\timeShiftOperator \fourierShift^{-1} + (\timeShiftOperator^2 + 1-\relaxationParameter) +\tfrac{1}{2}((1+\courantNumber_k)\relaxationParameter-2)\timeShiftOperator \fourierShift \Bigr ),
    \end{equation*}
    where each term in the product features the expression in \eqref{eq:charEquation} with the Courant number of each wave.
\end{lemma}
Similar results can be proved for more involved schemes, for instance based on the link TRT scheme \cite{d2009viscosity}, provided that the equilibria are of the form $(\alpha\identityMatrix{\numberConservationLaws} + \beta \vectorial{\fluxLetter}'(\referenceStateMarked{\vectorial{\conservedMoment}}))\vectorial{\conservedMoment}$ and that every component of the non-conserved moment(s) relaxes with the same relaxation parameter, \confer{} the proof below.
\begin{proof}[Proof of \Cref{lemma:splittingCharVectorialD1Q2}]
    We have 
    \begin{equation*}
        \schemeMatrixBulkFourier(\fourierShift) = 
        \left [\renewcommand{\arraystretch}{1.5}
        \begin{array}{c|c}
            \tfrac{1}{2}(\fourierShift^{-1}+\fourierShift)\identityMatrix{\numberConservationLaws} + \tfrac{\relaxationParameter}{2\latticeVelocity}(\fourierShift^{-1}-\fourierShift)\vectorial{\fluxLetter}'(\referenceState) & \tfrac{1-\relaxationParameter}{2}(\fourierShift^{-1}-\fourierShift)\identityMatrix{\numberConservationLaws}\\
            \hline
            \tfrac{1}{2}(\fourierShift^{-1}-\fourierShift)\identityMatrix{\numberConservationLaws} + \tfrac{\relaxationParameter}{2\latticeVelocity}(\fourierShift^{-1}+\fourierShift)\vectorial{\fluxLetter}'(\referenceState) & \tfrac{1-\relaxationParameter}{2}(\fourierShift^{-1}+\fourierShift)\identityMatrix{\numberConservationLaws}
        \end{array}
        \right ]
    \end{equation*}
    Introducing $\inTheWavesBase{\vectorial{\conservedMoment}} =\matrixEigenvectorsLinearized^{-1} \vectorial{\conservedMoment}$ and $\inTheWavesBase{\vectorial{\nonConservedMoment}} =\matrixEigenvectorsLinearized^{-1} \vectorial{\nonConservedMoment}$, the scheme can be rewritten as
    \begin{equation*}
        \begin{pmatrix}
            \inTheWavesBase{\vectorial{\conservedMoment}}_{\indexSpace}^{\indexTime+1}\\
            \inTheWavesBase{\vectorial{\nonConservedMoment}}_{\indexSpace}^{\indexTime+1}
        \end{pmatrix} = 
        \left [
        \begin{array}{c|c}
            \matrixEigenvectorsLinearized^{-1} & \zeroMatrix{\numberConservationLaws}\\
            \hline
            \zeroMatrix{\numberConservationLaws} & \matrixEigenvectorsLinearized^{-1} 
        \end{array}
        \right ]
        \schemeMatrixBulkFourier(\fourierShift)
        \left [
        \begin{array}{c|c}
            \matrixEigenvectorsLinearized & \zeroMatrix{\numberConservationLaws}\\
            \hline
            \zeroMatrix{\numberConservationLaws} & \matrixEigenvectorsLinearized
        \end{array}
        \right ]
        \begin{pmatrix}
            \inTheWavesBase{\vectorial{\conservedMoment}}_{\indexSpace}^{\indexTime}\\
            \inTheWavesBase{\vectorial{\nonConservedMoment}}_{\indexSpace}^{\indexTime}
        \end{pmatrix}, 
    \end{equation*}
    where the matrix on the right-hand side has the same characteristic polynomial of $\schemeMatrixBulkFourier(\fourierShift)$, as they are similar through a change of basis independent of $\fourierShift$.
    Since  $\matrixEigenvectorsLinearized$ diagonalizes $\vectorial{\fluxLetter}'(\referenceState)$:
    \begin{multline*}
        \left [
        \begin{array}{c|c}
            \matrixEigenvectorsLinearized^{-1} & \zeroMatrix{\numberConservationLaws}\\
            \hline
            \zeroMatrix{\numberConservationLaws} & \matrixEigenvectorsLinearized^{-1} 
        \end{array}
        \right ]
        \schemeMatrixBulkFourier(\fourierShift)
        \left [
        \begin{array}{c|c}
            \matrixEigenvectorsLinearized & \zeroMatrix{\numberConservationLaws}\\
            \hline
            \zeroMatrix{\numberConservationLaws} & \matrixEigenvectorsLinearized
        \end{array}
        \right ] \\
        =
        \left [\renewcommand{\arraystretch}{1.5}
        \begin{array}{c|c}
            \tfrac{1}{2}(\fourierShift^{-1}+\fourierShift)\identityMatrix{\numberConservationLaws} + \tfrac{\relaxationParameter}{2}(\fourierShift^{-1}-\fourierShift)\diagonalMatrix(\courantNumber_1, \dots, \courantNumber_{\numberConservationLaws}) & \tfrac{1-\relaxationParameter}{2}(\fourierShift^{-1}-\fourierShift)\identityMatrix{\numberConservationLaws}\\
            \hline
            \tfrac{1}{2}(\fourierShift^{-1}-\fourierShift)\identityMatrix{\numberConservationLaws} + \tfrac{\relaxationParameter}{2}(\fourierShift^{-1}+\fourierShift)\diagonalMatrix(\courantNumber_1, \dots, \courantNumber_{\numberConservationLaws})  & \tfrac{1-\relaxationParameter}{2}(\fourierShift^{-1}+\fourierShift)\identityMatrix{\numberConservationLaws}
        \end{array}
        \right ].
    \end{multline*}
    Using a permutation matrix (and its inverse) on the previous expression gives a block-diagonal matrix with $2\times2$ diagonal blocks given by 
    \begin{equation*}
        \frac{1}{2}
    \begin{pmatrix}
       (1+\courantNumber_k\relaxationParameter)\fourierShift^{-1} + (1-\courantNumber_k\relaxationParameter)\fourierShift & (1-\relaxationParameter)\fourierShift^{-1} -  (1-\relaxationParameter)\fourierShift\\
       (1+\courantNumber_k\relaxationParameter)\fourierShift^{-1} - (1-\courantNumber_k\relaxationParameter)\fourierShift & (1-\relaxationParameter)\fourierShift^{-1} +  (1-\relaxationParameter)\fourierShift
    \end{pmatrix}
    \end{equation*}
    for $k\in\integerInterval{1}{\numberConservationLaws}$, hence the claim by the fact that the determinant of a block-diagonal matrix is the product of the determinants of each block.
\end{proof}

The stability constraints of this scheme given below have been claimed in \cite{coulette2020vectorial} without proof, and naturally follow from \Cref{prop:consistencyD1Q2} and \Cref{lemma:splittingCharVectorialD1Q2}.

\begin{corollary}[Consistency and stability of the vectorial \lbmSchemeVectorial{1}{2}{\numberConservationLaws} scheme on $\relatives$]
    The vectorial \lbmSchemeVectorial{1}{2}{\numberConservationLaws} scheme tackling the linearized problem as above is such that, for each $k\in\integerInterval{1}{\numberConservationLaws}$, there exists a unique eigenvalue $\timeShiftOperator_{\textnormal{phy}, k}(\frequency\spaceStep)$ of $\schemeMatrixBulkFourier(e^{i\frequency\spaceStep})$ such that 
    \begin{equation}\label{eq:amplificationFactorsD1Q2Vect}
        \timeShiftOperator_{\textnormal{phy}, k}(\frequency\spaceStep) = 1 - i \eigenvalueLinearized_k\xi\timeStep + \Bigl ( -\frac{\courantNumber_k^2}{2} + (1-\courantNumber_k^2)\Bigl (\frac{1}{2}-\frac{1}{\relaxationParameter}\Bigr )\Bigr ) \latticeVelocity^2(\xi\timeStep)^2 + \bigO{(\xi\timeStep)^3}.
    \end{equation}
    This entails that the scheme is first-order accurate when $\relaxationParameter\in(0, 2)$ and second-order accurate when $\relaxationParameter = 2$.

    The scheme is $\ell^2$--stable, namely there exists $C>0$ such that $\sup_{\frequency\spaceStep\in[-\pi, \pi]} |\schemeMatrixBulkFourier(e^{i\frequency\spaceStep})^{\indexTime}|\leq C$ for all $\indexTime\in\naturals$, if and only if 
    \begin{equation*}
        \text{when } \relaxationParameter\in (0, 2), \quad \text{then}\quad \frac{\rho(\vectorial{\fluxLetter}'(\referenceState))}{\latticeVelocity} = \max_{k\in\integerInterval{1}{\numberConservationLaws}}|\courantNumber_k|\leq 1, \qquad\text{or}\qquad
        \text{when }\relaxationParameter = 2,\quad \text{then}\quad \frac{\rho(\vectorial{\fluxLetter}'(\referenceState))}{\latticeVelocity}<1,
    \end{equation*}
    where $\rho$ denotes the spectral radius.
\end{corollary}

Propagation and damping features for physical and spurious modes are---by virtue of \Cref{lemma:splittingCharVectorialD1Q2}---the same as the scalar \lbmScheme{1}{2} scheme wave-by-wave, i.e., replacing $\advectionVelocity$ with $\eigenvalueLinearized_{k}$ and $\courantNumber$ with $\courantNumber_k$ for $k\in\integerInterval{1}{\numberConservationLaws}$ in all claims of \Cref{sec:D1Q2PresentationAndProp}.

\subsection{Transparent boundary conditions for the \lbmScheme{1}{3} scheme}\label{sec:transpD1Q3ShallowWater}

The procedure is exactly the same as the for the fourth-order \lbmScheme{1}{3} with a scalar approach with recurrent computation of the coefficients.
Boundary conditions read \eqref{eq:rightBoundaryD1Q3}--\eqref{eq:leftBoundaryD1Q3} with coefficients given in \Cref{app:sec:transpD1Q3ShallowWater}.

\subsubsection{Numerical experiments}\label{sec:numericalExpD1Q3ShallowWater}

\paragraph{Linearized problem}\label{sec:foo}

\begin{figure}
    \begin{center}
        \includegraphics[width=1\textwidth]{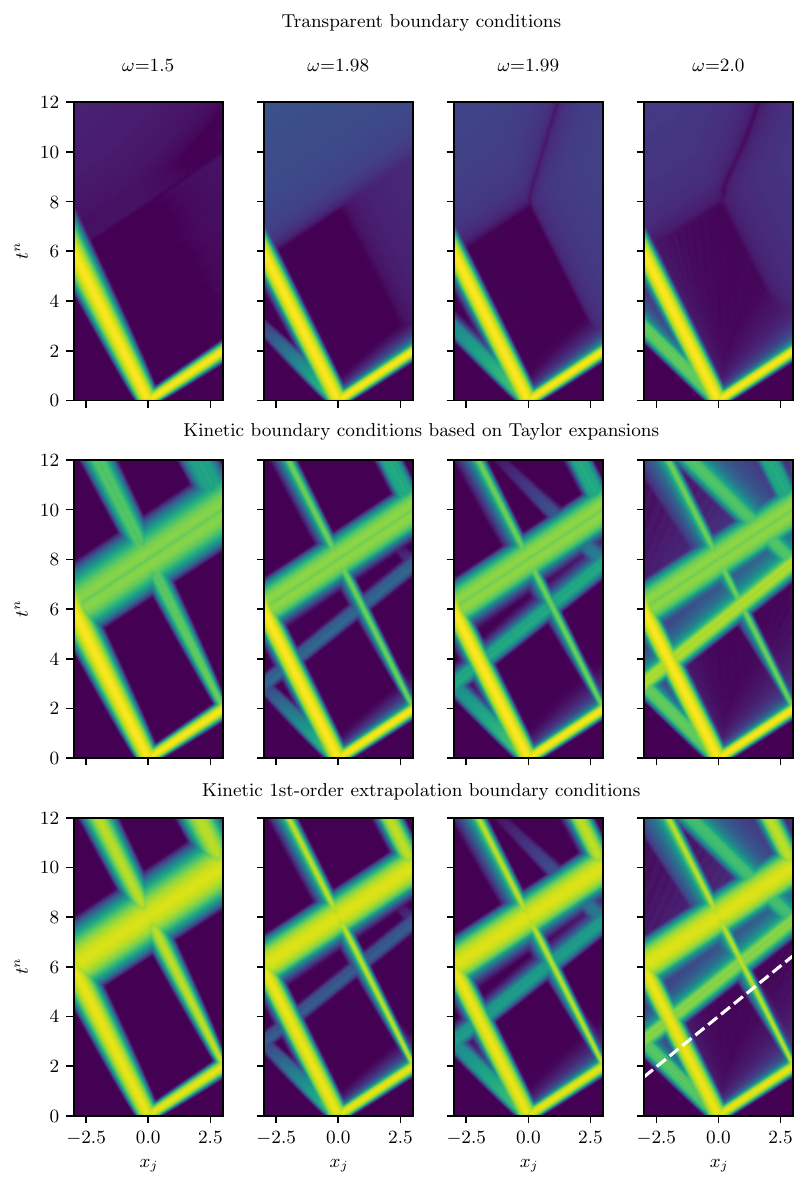}
    \end{center}\caption{\label{fig:D1Q3-shallow-water}Values of $\log_{10}(|\nonConservedMoment_{\indexSpace}^{\indexTime}|)$ (same color-scale as \Cref{fig:D1Q2-sys-vs-scal}) for the \lbmScheme{1}{3} for the linearized shallow water system endowed with different boundary conditions. The white dashed line has slope $\textnormal{V}_{\textnormal{g}}(\eta\timeStep = \pi, \frequency\spaceStep = \pi)$  given by \eqref{eq:groupVelocityInteresting} in the $\spaceVariable/\timeVariable$-plane.}
\end{figure}

We compare with kinetic boundary conditions devised using Taylor expansions in the spirit of \cite[Chapter 12]{bellotti2023numerical}.
The solution of the system can be decomposed on the eigenbasis 
\begin{equation*}
    \begin{pmatrix}
        \height\\
        \height\velocity
    \end{pmatrix}(\timeVariable, \spaceVariable)
    =
    \begin{pmatrix}
        1\\
        \referenceStateMarked{\velocity} - \soundSpeed
    \end{pmatrix}
    \phi_-(\spaceVariable - (\referenceStateMarked{\velocity} - \soundSpeed)\timeVariable)
    +
    \begin{pmatrix}
        1\\
        \referenceStateMarked{\velocity} + \soundSpeed
    \end{pmatrix}
    \phi_+(\spaceVariable - (\referenceStateMarked{\velocity} + \soundSpeed)\timeVariable),
\end{equation*}
hence 
\begin{equation*}
    \begin{pmatrix}
        1 & 1\\
        \referenceStateMarked{\velocity} - \soundSpeed & \referenceStateMarked{\velocity} + \soundSpeed
    \end{pmatrix}
    \begin{pmatrix}
        \phi_-\\
        \phi_+
    \end{pmatrix} = 
    \begin{pmatrix}
        \height\\
        \height\velocity
    \end{pmatrix}, 
    \qquad \text{thus}\qquad
    \phi_{\mp} \propto
    \pm (\referenceStateMarked{\velocity} \pm \soundSpeed)\height \mp \height \velocity.
\end{equation*}
Following the leitmotiv of ``killing'' reflected waves \cite{izquierdo2008characteristic, wissocq2017regularized}, we want $\phi_- = 0$, boiling down to $(\referenceStateMarked{\velocity} + \soundSpeed)\height - \height \velocity = 0$, on the right boundary; and $\phi_+ = 0$, calling for $-(\referenceStateMarked{\velocity} - \soundSpeed)\height + \height \velocity = 0$, on the left boundary.
For the sake of illustration, we try to achieve this using anti-bounce back conditions with source terms
\begin{equation*}
    \distributionFunctionLetter_{+, 0}^{\indexTime, \collided} = -\distributionFunctionLetter_{-, 1}^{\indexTime, \collided} + S_{\ell}^{\indexTime} 
    \qquad \text{and}\qquad 
    \distributionFunctionLetter_{-, \numberGridPoints + 1}^{\indexTime, \collided} = -\distributionFunctionLetter_{+, \numberGridPoints}^{\indexTime, \collided} + S_{r}^{\indexTime}.
\end{equation*}
To devise $S_{\ell}^{\indexTime}$ (we treat $S_{r}^{\indexTime}$ analogously), we follow the procedure that \cite{dubois2008equivalent} proposed to analyze the consistency of the bulk scheme: formal Taylor expansions at $(\timeVariable, \spaceVariable) = (\timeGridPoint{\indexTime}, \leftBoundary)$ of the boundary scheme rewritten on the moments give 
\begin{equation*}
    \latticeVelocity S_{\ell} - \frac{\nonNonConservedMoment^{\atEquilibrium}(\conservedMoment, \nonConservedMoment)}{\latticeVelocity} = \bigO{\spaceStep}.
\end{equation*}
Enforcing that the left-hand side of the previous equation be equal to $(\referenceStateMarked{\velocity} - \soundSpeed)\conservedMoment -  \nonConservedMoment$ (we made a choice concerning the sign in front of this term for empirical stability reasons), we obtain 
\begin{equation*}
    S_{\ell}^{\indexTime} = \frac{\nonNonConservedMoment^{\atEquilibrium}(\conservedMoment_1^{\indexTime}, \nonConservedMoment_1^{\indexTime}) + \latticeVelocity((\referenceStateMarked{\velocity} - \soundSpeed)\conservedMoment_1^{\indexTime} - \nonConservedMoment_1^{\indexTime})}{\latticeVelocity^2}
    \qquad \text{and}\qquad 
    S_{r}^{\indexTime} = \frac{\nonNonConservedMoment^{\atEquilibrium}(\conservedMoment_{\numberGridPoints}^{\indexTime}, \nonConservedMoment_{\numberGridPoints}^{\indexTime}) + \latticeVelocity(-(\referenceStateMarked{\velocity} + \soundSpeed)\conservedMoment_{\numberGridPoints}^{\indexTime} + \nonConservedMoment_{\numberGridPoints}^{\indexTime})}{\latticeVelocity^2}.
\end{equation*}

We also compare to kinetic 1st-order extrapolations $\distributionFunctionLetter_{+, 0}^{\indexTime, \collided} = \distributionFunctionLetter_{+, 1}^{\indexTime, \collided}$ and $\distributionFunctionLetter_{-, \numberGridPoints + 1}^{\indexTime, \collided} = \distributionFunctionLetter_{-, \numberGridPoints}^{\indexTime, \collided}$ at both boundaries.
They give analogous results when replaced by kinetic Dirichlet boundary conditions $\distributionFunctionLetter_{+, 0}^{\indexTime, \collided} = 0$ and $\distributionFunctionLetter_{-, \numberGridPoints + 1}^{\indexTime, \collided} = 0$, not included here.

We simulate using $\referenceStateMarked{\height} = 1$, $\referenceStateMarked{\velocity} = \tfrac{1}{2}$, and $\gravity = 1$.
This yields $\soundSpeed = 1$, $\referenceStateMarked{\velocity} - \soundSpeed = -\tfrac{1}{2}$, and $\referenceStateMarked{\velocity} + \soundSpeed = \tfrac{3}{2}$.
We take $\latticeVelocity = 2$, which ensures stability.
The domain delimited by $\leftBoundary = -3$ and $\rightBoundary = 3$, and discretized using $\numberGridPoints = 1000$ points.
The initial data are $\height^{\circ}(\spaceVariable) = \cos^{10}(\pi\spaceVariable)\indicatorFunction{(-1/2, 1/2)}(\spaceVariable)$ and $\velocity^{\circ}(\spaceVariable) = 0$.

Results are shown in \Cref{fig:D1Q3-shallow-water}.
When reflected, the spurious damped time-oscillating--space-smooth mode $(1-\relaxationParameter, 1)$ travelling at velocity $-2\referenceStateMarked{\velocity} = -1$ is not transformed into the sole physical mode travelling to the right, but becomes a fully checkerboard mode, essentially $(-1, -1)$, travelling to the right at group velocity\footnote{Now very different from the phase velocity.}
\begin{equation}\label{eq:groupVelocityInteresting}
    \textnormal{V}_{\textnormal{g}}(\eta\timeStep = \pi, \frequency\spaceStep = \pi) = \frac{2\latticeVelocity^2 \referenceStateMarked{\velocity}}{\latticeVelocity^2 + \referenceStateMarked{\velocity}^2 - \soundSpeed^2}.
\end{equation}
The fact that the reflected mode has $\timeShiftOperator \in\unitCircle$ (whereas the incoming mode had $|\timeShiftOperator|=|1-\relaxationParameter|<1$) explains why it is appears undamped even with $\relaxationParameter<2$, thus the importance of transparent boundary conditions.
When $\relaxationParameter$ is way apart from two, we observe important numerical diffusion (the beam enlarges with time) for the left-going physical wave, as its Courant number is rather small.
On the other hand, for $\relaxationParameter$ close to two, the right-going physical wave undergoes a significant amount of dispersion.
Finally, notice that kinetic boundary conditions devised through Taylor expansion significantly reduce the magnitude of reflected waves stemming from physical one, consistently with the fact that they have been designed to achieve this result, \confer{} \Cref{rem:costVsAccuracy}. 
Quite the opposite, they amplify the reflected wave arising from the numerical wave, which is not taken into account during the design of these boundary conditions.

\paragraph{Non-linear problem}

\begin{figure}
    \begin{center}
        \includegraphics[width=1\textwidth]{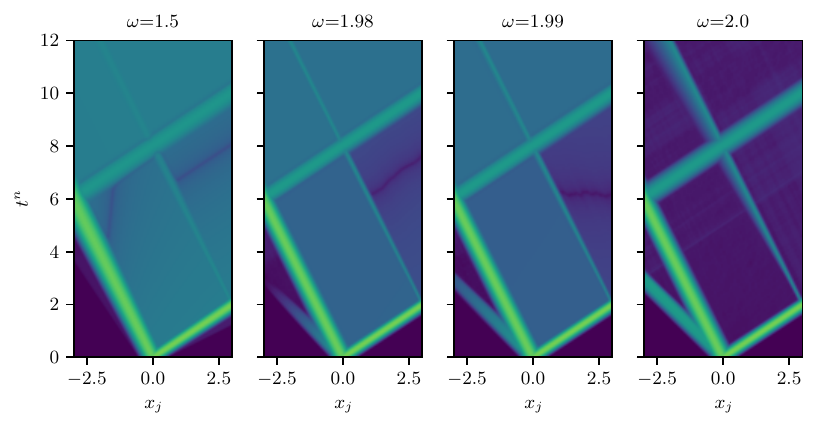}
    \end{center}\caption{\label{fig:D1Q3-shallow-water-non-linear}Values of $\log_{10}(|\nonConservedMoment_{\indexSpace}^{\indexTime} - \referenceStateMarked{\height}\referenceStateMarked{\velocity}|)$ (same color-scale as \Cref{fig:D1Q2-sys-vs-scal}) for the \lbmScheme{1}{3} for the shallow water system.}
\end{figure}

The setting stands as in the linear problem, except that for stability reasons, we take smoother initial data: $\height^{\circ}(\spaceVariable) = \referenceStateMarked{\height} + 10^{-3}e^{-100\spaceVariable^2}$ and $\velocity^{\circ}(\spaceVariable) = \referenceStateMarked{\velocity}$.
Boundary conditions are adapted in a way reminiscent of \eqref{eq:DirichletInflowD1Q2}: 
\begin{align*}
    \conservedMoment_{\numberGridPoints +1}^{\indexTime} = \sum_{k = 1}^{\indexTime}\coefficientsLaurentStableRootAllParities_k (\conservedMoment_{\numberGridPoints}^{\indexTime-k}-\referenceStateMarked{\height}) + \referenceStateMarked{\height} , \qquad 
    \nonConservedMoment_{\numberGridPoints +1}^{\indexTime} &=  \sum_{k = 1}^{\indexTime}\coefficientsLaurentStableRootAllParities_k (\nonConservedMoment_{\numberGridPoints}^{\indexTime-k}-\referenceStateMarked{\height}\referenceStateMarked{\velocity}) + \referenceStateMarked{\height}\referenceStateMarked{\velocity}, 
    \\ 
    \nonNonConservedMoment_{\numberGridPoints +1}^{\indexTime} &=  \sum_{k = 1}^{\indexTime}\coefficientsLaurentStableRootAllParities_k (\nonConservedMoment_{\numberGridPoints}^{\indexTime-k} - (\referenceStateMarked{\height}\referenceStateMarked{\velocity}^2 + \tfrac{1}{2}\gravity\referenceStateMarked{\height}^2 )) + \underbrace{\referenceStateMarked{\height}\referenceStateMarked{\velocity}^2 + \tfrac{1}{2}\gravity\referenceStateMarked{\height}^2}_{=\nonNonConservedMoment^{\atEquilibrium}(\referenceStateMarked{\height}, \referenceStateMarked{\height}\referenceStateMarked{\velocity})},
\end{align*}
and analogously at the left boundary, with the coefficients $\coefficientsLaurentStableRootAllParities_{\indexTime}$ computed on the linearized problem.

We see in \Cref{fig:D1Q3-shallow-water-non-linear} that our procedure is a first-order one: for a perturbation of amplitude $10^{-3}$ going into the boundary, the reflected wave has amplitude $10^{-7}$, which is roughly the square of the original amplitude.

\subsection{Transparent boundary conditions for the vectorial \lbmSchemeVectorial{1}{2}{\numberConservationLaws} scheme}

Although we adopt the scalar approach based on recurrent computations, we note that any of the strategies presented for the scalar \lbmScheme{1}{2} in \Cref{sec:transparentD1Q2} can be adapted---via the change of basis $\matrixEigenvectorsLinearized$ and thanks to \Cref{lemma:splittingCharVectorialD1Q2}---to the \lbmSchemeVectorial{1}{2}{\numberConservationLaws} scheme.
Let $\coefficientsLaurentStableRoot_{\indexTime}|_{\courantNumber = \courantNumber_k}$ denote the sequence given by \eqref{eq:recurrentDefinitionSD1Q2Scalar} where in these expressions, one takes $\courantNumber=\courantNumber_k$, for any $k\in\integerInterval{1}{\numberConservationLaws}$.
The boundary conditions thus read
\begin{align*}
    \vectorial{\phi}_{\numberGridPoints + 1}^{\indexTime} &= \sum_{k = 0}^{\lfloor (\indexTime - 1)/ 2 \rfloor}\matrixEigenvectorsLinearized\,\diagonalMatrix(\coefficientsLaurentStableRoot_{k}|_{\courantNumber = \courantNumber_1}, \dots, \coefficientsLaurentStableRoot_{k}|_{\courantNumber = \courantNumber_{\numberConservationLaws}}) \matrixEigenvectorsLinearized^{-1}\vectorial{\phi}_{\numberGridPoints}^{\indexTime-2k - 1}, \\
    \vectorial{\phi}_{0}^{\indexTime} &= \sum_{k = 0}^{\lfloor (\indexTime - 1)/ 2 \rfloor}\matrixEigenvectorsLinearized\,\diagonalMatrix\Bigl (\frac{\coefficientsLaurentStableRoot_{k}|_{\courantNumber = \courantNumber_1}}{\productRoots_1}, \dots, \frac{\coefficientsLaurentStableRoot_{k}|_{\courantNumber = \courantNumber_{\numberConservationLaws}}}{\productRoots_{\numberConservationLaws}} \Bigr ) \matrixEigenvectorsLinearized^{-1}\vectorial{\phi}_{1}^{\indexTime-2k - 1}, \qquad \vectorial{\phi}\in\{\vectorial{\conservedMoment}, \vectorial{\nonConservedMoment}\},
\end{align*}
with $\productRoots_k = \frac{(1-\courantNumber_k)\relaxationParameter-2}{(1+\courantNumber_k)\relaxationParameter-2}$, and it is understood that $1/\productRoots_k = 0$ if ever $\relaxationParameter = \frac{2}{1+\courantNumber_k}$.

\subsubsection{Numerical experiments}

\begin{figure}
    \begin{center}
        \includegraphics[width=1\textwidth]{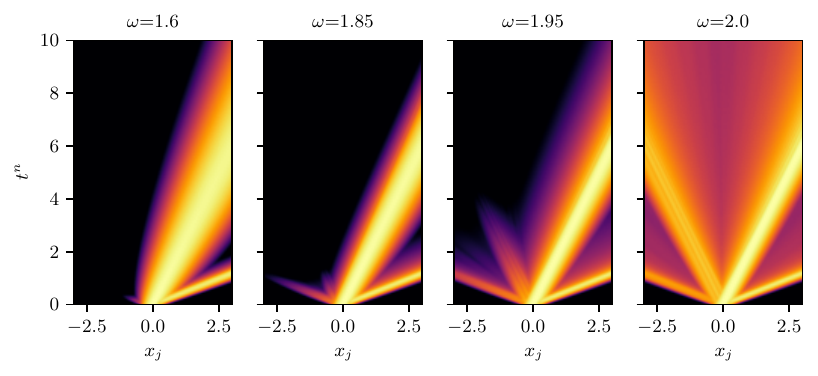}
    \end{center}\caption{\label{fig:D1Q2-vectorial}Values of $\log_{10}(|\conservedMoment_{1, \indexSpace}^{\indexTime}|)$ (same color-scale as \Cref{fig:D1Q2-sys-vs-scal}) for the \lbmSchemeVectorial{1}{2}{2} for the linearized shallow water system.}
\end{figure}

All conditions for the simulation are the same as in \Cref{sec:numericalExpD1Q3ShallowWater} except for those below.
We simulate a supersonic problem using $\referenceStateMarked{\height} = 1$, $\referenceStateMarked{\velocity} = \tfrac{3}{2}$, and $\gravity = 1$.
This yields $\soundSpeed = 1$, $\referenceStateMarked{\velocity} - \soundSpeed = \tfrac{1}{2}$, and $\referenceStateMarked{\velocity} + \soundSpeed = \tfrac{5}{2}$.
We take $\latticeVelocity = 3$, to ensure stability, and utilize $\numberGridPoints = 300$ grid-points.

Results are in \Cref{fig:D1Q2-vectorial}.
As $\latticeVelocity$ must be larger that the fastest wave, and there is a large separation between the fastest and the slowest wave, we see that the slowest wave is strongly diffused (its beam spreads on both sides) because its Courant number is small.
This is clearly manifest from \eqref{eq:amplificationFactorsD1Q2Vect}, since for this wave, the term $(1-\courantNumber^2) = \tfrac{35}{36}\approx 1$, whereas for the fastest wave $(1-\courantNumber^2) = \tfrac{11}{36}$.
For $\relaxationParameter = 1.95$ and $2$, we see that dispersion plays an important role, as beams look blurry on one side, opposite to propagation.
For $\relaxationParameter = 2$, the solution looks doughy due to dispersion and the lack of damping/dissipation.
These facts are particularly visible with coarse grids ($\numberGridPoints = 300 \ll 1000$).

\section*{Acknowledgments}

Constructive and useful comments by two anonymous referees are acknowledged.
The author thanks Irina Ginzburg for having brought relevant results on the stability of the scheme employed in \Cref{sec:2DScalar} to his attention.

\section*{Disclosures and declarations}
The author declares no conflict of interest.

\section*{Data availability statements}

Data sharing is not applicable to this article as no new data were created or analyzed in this
study.

\bibliographystyle{spmpsci}
\bibliography{biblio}

\appendix

\section{Detailed computations to study coinciding symbols for the \lbmScheme{1}{2} scheme}\label{app:D1Q2coincidingsymbols}

Taking $\eqref{eq:systemDoubleRootFirst}-\timeShiftOperator\eqref{eq:systemDoubleRootSecond}$, we obtain $\timeShiftOperator(\frequency\spaceStep)^2 = 1-\relaxationParameter$.
\begin{itemize}
    \item For $\relaxationParameter<1$, $\timeShiftOperator(\frequency\spaceStep) = \pm \sqrt{1-\relaxationParameter}\in\reals$, so issues potentially arise on the real axis. 
    Back into \eqref{eq:systemDoubleRootSecond}, we obtain 
    \begin{equation*}
        \pm 2 \sqrt{1-\relaxationParameter}= (2-\relaxationParameter)\cos(\frequency\spaceStep) - i\relaxationParameter\courantNumber\sin(\frequency\spaceStep),
    \end{equation*}
    which might be fulfilled only by $\frequency\spaceStep = 0, (-1)^{\alpha}\pi$ with $\alpha=0, 1$.
    Let $\frequency\spaceStep=0$, then $\pm 2 \sqrt{1-\relaxationParameter}= 2-\relaxationParameter$. We discard the case with minus sign as the right-hand side is strictly positive. Squaring both sides: $\relaxationParameter^2 = 0$, that we cannot consider.
    Let now $\frequency\spaceStep = (-1)^{\alpha}\pi$. We then consider $\pm 2 \sqrt{1-\relaxationParameter}= -(2-\relaxationParameter)$. For the previously mentioned reason, the minus sign has to be considered and thus we are back to the setting of $\frequency\spaceStep = 0$.
    
    This proves that the roots $\timeShiftOperator_{\textnormal{phy}}(\frequency\spaceStep)$ and $\timeShiftOperator_{\textnormal{spu}}(\frequency\spaceStep)$  are distinct for $\frequency\spaceStep\in[-\pi, \pi]$.

    \item For $\relaxationParameter\geq 1$, we have $\timeShiftOperator(\frequency\spaceStep) = \pm i \sqrt{\relaxationParameter - 1}\in i \reals$, so issues potentially arise on the imaginary axis.
    Into \eqref{eq:systemDoubleRootSecond}:
    \begin{equation*}
        \pm 2 i \sqrt{\relaxationParameter-1}= (2-\relaxationParameter)\cos(\frequency\spaceStep) - i\relaxationParameter\courantNumber\sin(\frequency\spaceStep),
    \end{equation*}
    which calls for consideration of $\frequency\spaceStep = (-1)^{\alpha}\pi/2$ with $\alpha=0, 1$.
    We obtain $\pm 2 \sqrt{\relaxationParameter - 1}= -\relaxationParameter\courantNumber (-1)^{\alpha}$, so the only meaningful relation is $2\sqrt{\relaxationParameter - 1} = \relaxationParameter|\courantNumber|$.
    Squaring both sides and solving in $\relaxationParameter$ gives $\relaxationParameter = \frac{2}{\courantNumber^2} (1-\sqrt{1-\courantNumber^2})$.
\end{itemize}

\section{Expressions of weights in \eqref{eq:recurrecenceRight} and \eqref{eq:recurrenceLeft}}\label{app:expressionWeights}

\begin{multline*}
    \mathscr{V}_{-1}(\indexTime) \\
    = (\omega-1)\Biggl ( 
    \frac{4  {({\argumentLegendrePolynomials}^{2} - {\argumentLegendrePolynomials})} \indexTime^{2} - {(2  {\courantNumber}^{2} {\argumentLegendrePolynomials} \indexTime^{2} - {\courantNumber}^{2} {\argumentLegendrePolynomials} \indexTime - 3  {\courantNumber}^{2} {\argumentLegendrePolynomials})} {\omega}^{2} + 3  {\argumentLegendrePolynomials}^{2} }{2  {({\argumentLegendrePolynomials} - 1)} \indexTime^{2} - {({\courantNumber}^{2} \indexTime^{2} + {\courantNumber}^{2} \indexTime)} {\omega}^{2} - {({\argumentLegendrePolynomials} - 1)} \indexTime - {(2  {({\courantNumber} {\argumentLegendrePolynomials} - {\courantNumber})} \indexTime^{2} - 3  {\courantNumber} {\argumentLegendrePolynomials} - {({\courantNumber} {\argumentLegendrePolynomials} + 2  {\courantNumber})} \indexTime)} {\omega} - 3  {\argumentLegendrePolynomials} + 3}\\
    +
    \frac{- 8  {({\argumentLegendrePolynomials}^{2} - {\argumentLegendrePolynomials})} \indexTime - {(3  {\courantNumber} {\argumentLegendrePolynomials}^{2} + 4  {({\courantNumber} {\argumentLegendrePolynomials}^{2} - {\courantNumber} {\argumentLegendrePolynomials})} \indexTime^{2} + 6  {\courantNumber} {\argumentLegendrePolynomials} - 2  {(4  {\courantNumber} {\argumentLegendrePolynomials}^{2} - {\courantNumber} {\argumentLegendrePolynomials})} \indexTime - 3  {\courantNumber})} {\omega} - 3}{2  {({\argumentLegendrePolynomials} - 1)} \indexTime^{2} - {({\courantNumber}^{2} \indexTime^{2} + {\courantNumber}^{2} \indexTime)} {\omega}^{2} - {({\argumentLegendrePolynomials} - 1)} \indexTime - {(2  {({\courantNumber} {\argumentLegendrePolynomials} - {\courantNumber})} \indexTime^{2} - 3  {\courantNumber} {\argumentLegendrePolynomials} - {({\courantNumber} {\argumentLegendrePolynomials} + 2  {\courantNumber})} \indexTime)} {\omega} - 3  {\argumentLegendrePolynomials} + 3} \Biggr ).
\end{multline*}

\begin{multline*}
    \mathscr{V}_{-2}(\indexTime) =\\ 
    (\relaxationParameter-1)^2\Biggl ( 
    -\frac{2  {({\argumentLegendrePolynomials} - 1)} \indexTime^{2} - {({\courantNumber}^{2} \indexTime^{2} - 2  {\courantNumber}^{2} \indexTime - 3  {\courantNumber}^{2})} {\omega}^{2} - 7  {({\argumentLegendrePolynomials} - 1)} \indexTime}{2  {({\argumentLegendrePolynomials} - 1)} \indexTime^{2} - {({\courantNumber}^{2} \indexTime^{2} + {\courantNumber}^{2} \indexTime)} {\omega}^{2} - {({\argumentLegendrePolynomials} - 1)} \indexTime - {(2  {({\courantNumber} {\argumentLegendrePolynomials} - {\courantNumber})} \indexTime^{2} - 3  {\courantNumber} {\argumentLegendrePolynomials} - {({\courantNumber} {\argumentLegendrePolynomials} + 2  {\courantNumber})} \indexTime)} {\omega} - 3  {\argumentLegendrePolynomials} + 3}\\
    -\frac{- {(2  {({\courantNumber} {\argumentLegendrePolynomials} - {\courantNumber})} \indexTime^{2} + 3  {\courantNumber} {\argumentLegendrePolynomials} - {(7  {\courantNumber} {\argumentLegendrePolynomials} - 4  {\courantNumber})} \indexTime + 6  {\courantNumber})} {\omega} + 3  {\argumentLegendrePolynomials} - 3}{2  {({\argumentLegendrePolynomials} - 1)} \indexTime^{2} - {({\courantNumber}^{2} \indexTime^{2} + {\courantNumber}^{2} \indexTime)} {\omega}^{2} - {({\argumentLegendrePolynomials} - 1)} \indexTime - {(2  {({\courantNumber} {\argumentLegendrePolynomials} - {\courantNumber})} \indexTime^{2} - 3  {\courantNumber} {\argumentLegendrePolynomials} - {({\courantNumber} {\argumentLegendrePolynomials} + 2  {\courantNumber})} \indexTime)} {\omega} - 3  {\argumentLegendrePolynomials} + 3} \Biggr ).
\end{multline*}
The expressions for $\overline{\mathscr{V}}_{-1}(\indexTime)$ and $\overline{\mathscr{V}}_{-2}(\indexTime)$ are the same as the ones for  $\mathscr{V}_{-1}(\indexTime) $ and $\mathscr{V}_{-2}(\indexTime) $ except for the fact that they feature $-\courantNumber$ instead of $\courantNumber$ (not in the definition of $\argumentLegendrePolynomials$).

\section{Proof of \Cref{lemma:rootsRadicand}}\label{app:rootsRadicand}

We study the roots of $\varphi_4(\timeShiftOperator)=  (\relaxationParameter-1)^2\timeShiftOperator^{4} + ((\courantNumber^2-1)\relaxationParameter^2 + 2\relaxationParameter - 2)\timeShiftOperator^{2} + 1$.
Following the notations of \cite[Chapter 4]{strikwerda2004finite}, we obtain $\varphi_4^{*}(\timeShiftOperator) = \timeShiftOperator^4 \overline{\varphi_4(\overline{\timeShiftOperator}^{-1})}=   \timeShiftOperator^4 + ((\courantNumber^2-1)\relaxationParameter^2 + 2\relaxationParameter-2)\timeShiftOperator^{2}  + (\relaxationParameter-1)^2$, hence $\varphi_4(0) = 1$ and $\varphi_4^{*}(0) = (\relaxationParameter-1)^2$, so $\varphi_3(\timeShiftOperator) =\timeShiftOperator^{-1}(\varphi_4^{*}(0)\varphi_4(\timeShiftOperator) - \varphi_4(0)\varphi_4^{*}(\timeShiftOperator)) = \timeShiftOperator((\relaxationParameter-1)^2 - 1)(((\relaxationParameter-1)^2+1)\timeShiftOperator^2 + ((\courantNumber^2-1)\relaxationParameter^2 + 2\relaxationParameter-2))$.
    
As we first want to investigate roots on $\unitCircle$ through \cite[Theorem 4.3.8]{strikwerda2004finite}, we look for $\varphi_3\equiv 0$, which can happen only when $\relaxationParameter = 2$ (our paper does not consider $\relaxationParameter = 0$).
We thus continue with $\relaxationParameter = 2$.
Next, we have to check the roots of $\varphi_4'(\timeShiftOperator) = 4\timeShiftOperator^3 +4 (2\courantNumber^2 - 1)\timeShiftOperator$ and show that they belong to $\closedUnitDisk$.
They are $\timeShiftOperator = 0$ and $\timeShiftOperator = \pm \sqrt{1-2\courantNumber^2}$, which can be either real or complex conjugate. Moreover, $\pm \sqrt{1-2\courantNumber^2}\in\closedUnitDisk$ under \eqref{eq:stabConditionD1Q2}.
This proves that if $\relaxationParameter = 2$ the roots are on $\unitCircle$.
    
We finish on proving that the roots of $\varphi_4(\timeShiftOperator)$ belong to $\neighborhoodInfinity$ when $\relaxationParameter<2$.
To this end, we want to show that here, the zeros of $\psi_4(\timeShiftOperator)=\varphi_4^{*}(\timeShiftOperator)$ are in $\unitDisk$.
We can use \cite[Theorem 4.3.1]{strikwerda2004finite}, checking $|\psi_4(0)| =(\relaxationParameter-1)^2 <|\psi_4^{*}(0)| = 1$, where $\psi_4^{*}(\timeShiftOperator) = \timeShiftOperator^4 \overline{\psi_4(\overline{\timeShiftOperator}^{-1})}$, which boils down to $0< \relaxationParameter< 2$.
We continue with $\psi_3(\timeShiftOperator) = \timeShiftOperator^{-1}(\psi_4^{*}(0)\psi_4(\timeShiftOperator) - \psi_4(0)\psi_4^{*}(\timeShiftOperator))  = -\varphi_3(\timeShiftOperator) = (1-(\relaxationParameter-1)^2) \timeShiftOperator ( (1+(\relaxationParameter-1)^2)\timeShiftOperator^2 + ((\courantNumber^2-1)\relaxationParameter^2 + 2\relaxationParameter-2) ) $, showing that its roots are in $\unitDisk$.
To this end, we can straightly show that the roots of $\psi_2(\timeShiftOperator)= (1+(\relaxationParameter-1)^2)\timeShiftOperator^2 + ((\courantNumber^2-1)\relaxationParameter^2 + 2\relaxationParameter-2) $ are in $\unitDisk$.
We want $|\psi_2(0)|=|(\courantNumber^2-1)\relaxationParameter^2 + 2\relaxationParameter-2|<|\psi_2^*(0)| = 1+(\relaxationParameter-1)^2$, where $\psi_2^{*}(\timeShiftOperator) = \timeShiftOperator^2 \overline{\psi_2(\overline{\timeShiftOperator}^{-1})}$.
\begin{itemize}
    \item If $(\courantNumber^2-1)\relaxationParameter^2 + 2\relaxationParameter-2\geq 0$, we want $(\courantNumber^2-2)\relaxationParameter^2 + 4\relaxationParameter-4< 0$.
    Under \eqref{eq:stabConditionD1Q2}, we have $\courantNumber^2-1\leq -1$, thus $(\courantNumber^2-2)\relaxationParameter^2 + 4\relaxationParameter-4\leq -\relaxationParameter^2 + 4\relaxationParameter-4 $. The right-hand side of the inequality is strictly negative for $\relaxationParameter\neq 2$, concluding this case.
    \item If $(\courantNumber^2-1)\relaxationParameter^2 + 2\relaxationParameter-2< 0$, we want $-\courantNumber^2\relaxationParameter^2<0$, which is trivially true as $\courantNumber, \relaxationParameter\neq 0$.
\end{itemize}
We obtain that $\psi_1(\timeShiftOperator) = \timeShiftOperator^{-1}(\psi_2^{*}(0)\psi_2(\timeShiftOperator) - \psi_2(0)\psi_2^{*}(\timeShiftOperator)) = \alpha\timeShiftOperator$, for some $\alpha\in\reals$: its unique root is $\timeShiftOperator = 0\in\unitDisk$.

\section{Additional numerical simulations with the \lbmScheme{1}{2} scheme}\label{app:moreSimulationsD1Q2}

\begin{figure}
    \begin{center}
        \includegraphics[width=1\textwidth]{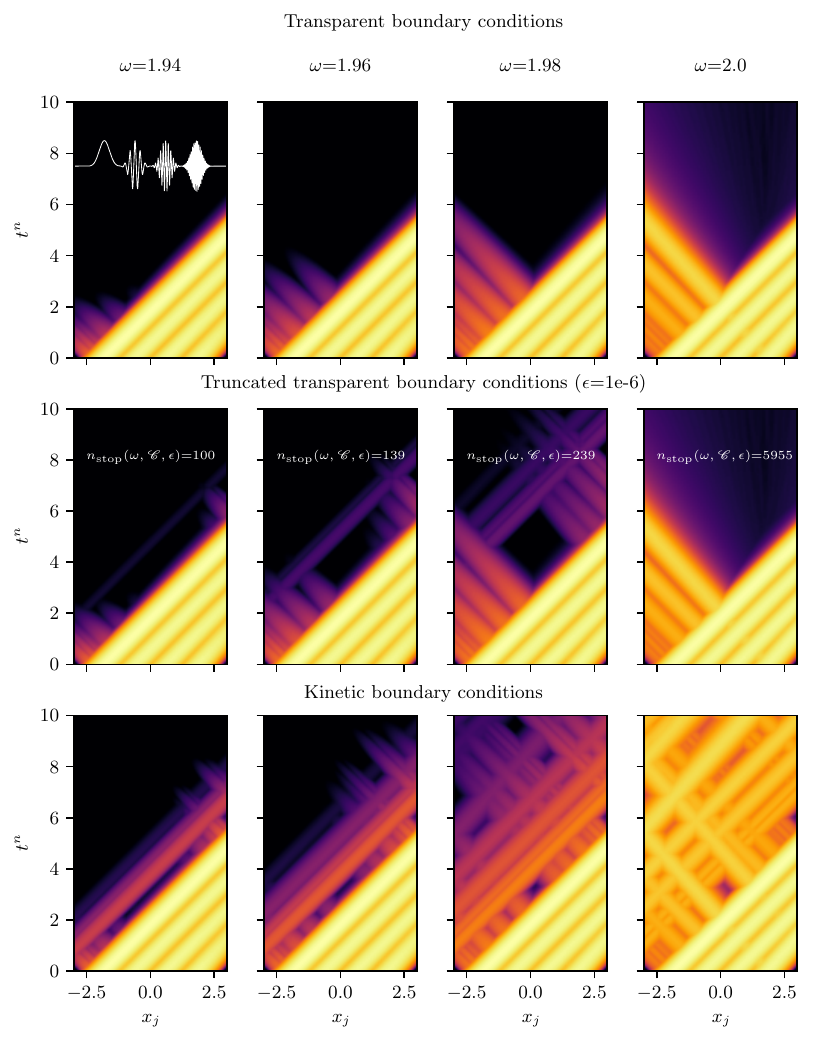}
    \end{center}\caption{\label{fig:D1Q2-comparison-BC-many-packets}Values of $\log_{10}(|\conservedMoment_{\indexSpace}^{\indexTime}|)$ (same color-scale as \Cref{fig:D1Q2-sys-vs-scal}) for the \lbmScheme{1}{2} endowed with different boundary conditions and using \eqref{eq:manypackets}. The white line draws the initial datum.}
\end{figure}

We here report numerical results, see \Cref{fig:D1Q2-comparison-BC-many-packets}, lacking in \Cref{sec:numExpD1Q2}.

\section{Proof of \Cref{prop:stabilityD2Q5}}\label{app:prop:stabilityD2Q5}

We only need to check $\Psi_2(\timeShiftOperator)$, using \cite[Theorem 4.3.7]{strikwerda2004finite}.
    Let $\fourierShift_{\xLabel}=e^{i\frequency_{\xLabel}\spaceStep}$ and $\fourierShift_{\yLabel}=e^{i\frequency_{\yLabel}\spaceStep}$, hence 
    \begin{align*}
        \Psi_{2}(\timeShiftOperator)=\timeShiftOperator^2 + \bigl ( (\relaxationParameter-2) &+ i\relaxationParameter\courantNumberX\sin(\frequency_{\xLabel}\spaceStep) + (\relaxationParameter-2)\equilibriumCoefficientSymmetricX(\cos(\frequency_{\xLabel}\spaceStep)-1) \\
        &+ i\relaxationParameter\courantNumberY\sin(\frequency_{\yLabel}\spaceStep) + (\relaxationParameter-2)\equilibriumCoefficientSymmetricY(\cos(\frequency_{\yLabel}\spaceStep)-1)\bigr )\timeShiftOperator + (1 - \relaxationParameter).
    \end{align*}
    We obtain 
    \begin{align*}
        \Psi_{2}^*(\timeShiftOperator) = \timeShiftOperator^2 \overline{\Psi_{2}(\overline{\timeShiftOperator}^{-1})}= (1 - \relaxationParameter) \timeShiftOperator^2 + \bigl ( (\relaxationParameter-2) &- i\relaxationParameter\courantNumberX\sin(\frequency_{\xLabel}\spaceStep) + (\relaxationParameter-2)\equilibriumCoefficientSymmetricX(\cos(\frequency_{\xLabel}\spaceStep)-1) \\
        &- i\relaxationParameter\courantNumberY\sin(\frequency_{\yLabel}\spaceStep) + (\relaxationParameter-2)\equilibriumCoefficientSymmetricY(\cos(\frequency_{\yLabel}\spaceStep)-1)\bigr )\timeShiftOperator + 1.
    \end{align*}
    A simple computation yields 
    \begin{align*}
        \Psi_1(\timeShiftOperator) &= \timeShiftOperator^{-1}(\Psi_{2}^*(0)\Psi_{2}(\timeShiftOperator) - \Psi_{2}(0)\Psi_{2}^*(\timeShiftOperator)) \\
        &= \relaxationParameter(2-\relaxationParameter) (\timeShiftOperator - 1 - \equilibriumCoefficientSymmetricX(\cos(\frequency_{\xLabel}\spaceStep)-1) - \equilibriumCoefficientSymmetricY(\cos(\frequency_{\yLabel}\spaceStep)-1) + i\courantNumberX\sin(\frequency_{\xLabel}\spaceStep) + i\courantNumberY\sin(\frequency_{\yLabel}\spaceStep) ).
    \end{align*}
    We see that $\Psi_1(\timeShiftOperator)\equiv 0$ if $\relaxationParameter=2$ and not identically zero when $\relaxationParameter\in(0, 2)$.
    \begin{itemize}
        \item Let $\relaxationParameter = 2$. We have to consider 
        \begin{equation*}
            \Psi_{2}'(\timeShiftOperator) = 2 (\timeShiftOperator + i\courantNumberX\sin(\frequency_{\xLabel}\spaceStep) + i\courantNumberY\sin(\frequency_{\yLabel}\spaceStep))
        \end{equation*}
        and show that its only root is in $\closedUnitDisk$ under some condition.
        Its modulus squared is $\courantNumberX^2\sin(\frequency_{\xLabel}\spaceStep)^2 +2\courantNumberX\courantNumberY \sin(\frequency_{\xLabel}\spaceStep)\sin(\frequency_{\yLabel}\spaceStep) + \courantNumberY^2\sin(\frequency_{\yLabel}\spaceStep)^2$.
        Its maximum is clearly reached when $\sin(\frequency_{\xLabel}\spaceStep) = \text{sgn}(\courantNumberX)$ and $\sin(\frequency_{\yLabel}\spaceStep) = \text{sgn}(\courantNumberY)$, which means $\frequency_{\xLabel}\spaceStep = \text{sgn}(\courantNumberX) \tfrac{\pi}{2}$ and $\frequency_{\yLabel}\spaceStep = \text{sgn}(\courantNumberY) \tfrac{\pi}{2}$. For these frequencies, we obtain 
        \begin{equation*}
            |\courantNumberX|^2 +2|\courantNumberX||\courantNumberY| + |\courantNumberY|^2 = (|\courantNumberX| + |\courantNumberY|)^2\leq 1,
        \end{equation*} 
        giving the claim.
        \item Let $\relaxationParameter\in (0, 2)$. The condition $|1-\relaxationParameter| = |\Psi_2(0)|<|\Psi_2^*(0)| = 1$ is fulfilled.
        We have to check that the only root of $\Psi_1(\timeShiftOperator)$ is in $\closedUnitDisk$.
        Let $\mu_{\xLabel}=\sin(\tfrac{1}{2}\frequency_{\xLabel}\spaceStep)\in [-1, 1]$ and $\mu_{\yLabel}=\sin(\tfrac{1}{2}\frequency_{\yLabel}\spaceStep)\in [-1, 1]$.
        Using the identities $\sin(\frequency_{\xLabel}\spaceStep) = 2\mu_{\xLabel}\sqrt{1-\mu_{\xLabel}^2}$ and $\cos(\frequency_{\xLabel}\spaceStep) = 1-2\mu_{\xLabel}^2$, the condition on the root reads
        \begin{multline}\label{eq:stabilityD2Q5LinkOmegaNotTwo}
            -\bigl ( (\courantNumberX^2-\equilibriumCoefficientSymmetricX^2)\mu_{\xLabel}^2 - (\courantNumberX^2 - \equilibriumCoefficientSymmetricX)\bigr )\mu_{\xLabel}^2 -\bigl ( (\courantNumberY^2-\equilibriumCoefficientSymmetricY^2)\mu_{\yLabel}^2 - (\courantNumberY^2 - \equilibriumCoefficientSymmetricY)\bigr )\mu_{\yLabel}^2\\
            +2\Bigl (\equilibriumCoefficientSymmetricX\equilibriumCoefficientSymmetricY \mu_{\xLabel}\mu_{\yLabel} + \courantNumberX\courantNumberY \sqrt{(1-\mu_{\xLabel}^2)(1-\mu_{\yLabel}^2)}\Bigr ) \mu_{\xLabel}\mu_{\yLabel}\leq 0
        \end{multline}
        for every $(\mu_{\xLabel}, \mu_{\yLabel})\in[-1, 1]^2$.

        We can investigate specific values of $(\mu_{\xLabel}, \mu_{\yLabel})\in[-1, 1]^2$ to find explicit necessary conditions.
        Before doing this, remark that we are checking the absolute values of the roots (not their multiplicity), which are continuous functions of the scheme parameters (for instance, $\relaxationParameter$).
        This shows that a necessary condition to have \eqref{eq:stabilityD2Q5LinkOmegaNotTwo} is the stability condition for $\relaxationParameter = 2$, namely $|\courantNumberX| + |\courantNumberY| \leq 1$, which trivially entails $|\courantNumberX|, |\courantNumberY|\leq 1$.

        Let us first pick $\mu_{\yLabel} = 0$ (or symmetrically, $\mu_{\xLabel} = 0$): \eqref{eq:stabilityD2Q5LinkOmegaNotTwo} becomes $(\courantNumberX^2-\equilibriumCoefficientSymmetricX^2)\mu_{\xLabel}^2 - (\courantNumberX^2 - \equilibriumCoefficientSymmetricX)\geq 0$ for $\mu_{\xLabel}\in[-1, 1]$. 
        This is the stability constraint of a 1D problem, \confer{} \cite{bellotti2024initialisation}.
        \begin{itemize}
            \item Let $\courantNumberX^2-\equilibriumCoefficientSymmetricX^2\geq0$. 
            In this case, we have to check that $(\courantNumberX^2-\equilibriumCoefficientSymmetricX^2) - (\courantNumberX^2 - \equilibriumCoefficientSymmetricX) = \equilibriumCoefficientSymmetricX(1-\equilibriumCoefficientSymmetricX)\geq 0$, which is equivalent to $0\leq \equilibriumCoefficientSymmetricX\leq 1$.
            \item Let $\courantNumberX^2-\equilibriumCoefficientSymmetricX^2<0$. We request $ \equilibriumCoefficientSymmetricX\geq \courantNumberX^2$. 
        \end{itemize}
        As we have pointed out, the previous necessary condition entails $\courantNumberX^2\leq 1$, so a necessary condition for \eqref{eq:stabilityD2Q5LinkOmegaNotTwo} is $\courantNumberX^2\leq \equilibriumCoefficientSymmetricX\leq 1$ and $\courantNumberY^2\leq \equilibriumCoefficientSymmetricY\leq 1$.
        This also shows that $\equilibriumCoefficientSymmetricX, \equilibriumCoefficientSymmetricY\geq 0$.

        Secondly, we consider $\mu_{\xLabel} = (-1)^{\alpha_{\xLabel}}$ and $\mu_{\yLabel} = (-1)^{\alpha_{\yLabel}}$ with $\alpha_{\xLabel}, \alpha_{\yLabel}$ being either zero or one.
        Into \eqref{eq:stabilityD2Q5LinkOmegaNotTwo}, we obtain 
        \begin{equation*}
            \equilibriumCoefficientSymmetricX(\equilibriumCoefficientSymmetricX-1) + \equilibriumCoefficientSymmetricY(\equilibriumCoefficientSymmetricY-1) +2 \equilibriumCoefficientSymmetricX\equilibriumCoefficientSymmetricY = (\equilibriumCoefficientSymmetricX+\equilibriumCoefficientSymmetricY)(\equilibriumCoefficientSymmetricX+\equilibriumCoefficientSymmetricY - 1) \leq 0.
        \end{equation*}
        Assuming that the previous necessary condition is met, we have $\equilibriumCoefficientSymmetricX+\equilibriumCoefficientSymmetricY\geq 0$, hence we now request $\equilibriumCoefficientSymmetricX+\equilibriumCoefficientSymmetricY \leq 1$.
    \end{itemize}

\section{Proof of \Cref{prop:stabilityD1Q3ShallowWater}}\label{app:prop:stabilityD1Q3ShallowWater}
    We again proceed following \cite[Chapter 4]{strikwerda2004finite}.
        Set $\varphi_3(\timeShiftOperator) = \determinant(\timeShiftOperator\identityMatrix{3}-\schemeMatrixBulkFourier(e^{i\frequency\spaceStep}))$.
    The polynomial $\varphi_3^*(\timeShiftOperator) =\timeShiftOperator^3 \overline{\varphi_3(\overline{\timeShiftOperator}^{-1})}$ is computed, and yields $\varphi_2(\timeShiftOperator) = \timeShiftOperator^{-1}(\varphi_3^*(0)\varphi_3(\timeShiftOperator)-\varphi_3(0)\varphi_3^*(\timeShiftOperator)) = \relaxationParameter(2-\relaxationParameter) \psi_2(\timeShiftOperator)$,  with $ \psi_2(\timeShiftOperator)\not\equiv 0$ a monic polynomial of degree two.
    
    \begin{itemize}
        \item Let $\relaxationParameter\in (0, 2)$, so that $\varphi_2(\timeShiftOperator)\not\equiv 0$. Requesting $|\varphi_3(0)|<|\varphi_3^*(0)|$ gives $\relaxationParameter\in (0, 2)$, which we assumed for this item.
        Without loss of generality, assume that $\referenceStateMarked{\velocity}\geq 0$.
        
        We compute $\psi_2^*(\timeShiftOperator) = \timeShiftOperator^2 \overline{\psi_2(\overline{\timeShiftOperator}^{-1})}$ and then $\psi_1(\timeShiftOperator) =  \timeShiftOperator^{-1}(\psi_2^*(0)\psi_2(\timeShiftOperator)-\psi_2(0)\psi_2^*(\timeShiftOperator))= (\cos(\frequency\spaceStep) - 1) \eta_1(\timeShiftOperator)$, where $\eta_1$ is a first-order polynomial, which identically vanish only if $\latticeVelocity=\referenceStateMarked{\velocity}+\soundSpeed$ and $\referenceStateMarked{\velocity}=\soundSpeed$. We see that $\psi_1(\timeShiftOperator)\equiv 0$ only when $\frequency\spaceStep = 0$ and when $\latticeVelocity=\referenceStateMarked{\velocity}+\soundSpeed$ and $\referenceStateMarked{\velocity}=\soundSpeed$, which we first study.
        Consider  $\frequency\spaceStep = 0$: we obtain $\psi_2'(\timeShiftOperator)|_{\frequency\spaceStep=0} = 2(\timeShiftOperator-1)$ whose root is in $\closedUnitDisk$.
        Then, consider $\latticeVelocity=\referenceStateMarked{\velocity}+\soundSpeed$ and $\referenceStateMarked{\velocity}=\soundSpeed$, which yields $\psi_2'(\timeShiftOperator) = 2\timeShiftOperator - \cos(\frequency\spaceStep) + i\sin(\frequency\spaceStep) - 1$, whose root is on $\unitCircle$, so a fortiori in $\closedUnitDisk$.
    
        Now consider $\frequency\spaceStep\neq 0$ and that the identities $\latticeVelocity=\referenceStateMarked{\velocity}+\soundSpeed$, $\referenceStateMarked{\velocity}=\soundSpeed$ do not hold. We request that $|\psi_2(0)|^2< |\psi_2^*(0)|^2$, which translates into 
        \begin{equation*}
            (\cos(\frequency\spaceStep) - 1)\Biggl ( {\left(\frac{{\soundSpeed}^{4}}{{\latticeVelocity}^{4}} - \frac{2 {\soundSpeed}^{2}}{{\latticeVelocity}^{2}} - \frac{2 {\soundSpeed}^{2} {\referenceStateMarked{\velocity}}^{2}}{{\latticeVelocity}^{4}} - \frac{2 {\referenceStateMarked{\velocity}}^{2}}{{\latticeVelocity}^{2}} + \frac{{\referenceStateMarked{\velocity}}^{4}}{{\latticeVelocity}^{4}} + 1\right)} \cos(\frequency\spaceStep) - \frac{{\soundSpeed}^{4}}{{\latticeVelocity}^{4}} + \frac{2 {\soundSpeed}^{2} {\referenceStateMarked{\velocity}}^{2}}{{\latticeVelocity}^{4}} - \frac{4 {\referenceStateMarked{\velocity}}^{2}}{{\latticeVelocity}^{2}} - \frac{{\referenceStateMarked{\velocity}}^{4}}{{\latticeVelocity}^{4}} + 1
            \Biggr )<0.
        \end{equation*}
        This culminates in the analysis of 
        \begin{equation*}
            \frac{(\latticeVelocity+\referenceStateMarked{\velocity} + \soundSpeed)(\latticeVelocity-\referenceStateMarked{\velocity} + \soundSpeed)(-\latticeVelocity+\referenceStateMarked{\velocity} + \soundSpeed)(-\latticeVelocity-\referenceStateMarked{\velocity} + \soundSpeed)}{\latticeVelocity^4}\mu - \frac{{\soundSpeed}^{4}}{{\latticeVelocity}^{4}} + \frac{2 {\soundSpeed}^{2} {\referenceStateMarked{\velocity}}^{2}}{{\latticeVelocity}^{4}} - \frac{4 {\referenceStateMarked{\velocity}}^{2}}{{\latticeVelocity}^{2}} - \frac{{\referenceStateMarked{\velocity}}^{4}}{{\latticeVelocity}^{4}} + 1>0
        \end{equation*}
        for $ \mu\in(-1, 1)$.
        The left-hand side is affine in $\mu$, except when $\latticeVelocity = \referenceStateMarked{\velocity} + \soundSpeed$, which by the assumptions before cannot happen.
        When $\latticeVelocity > \referenceStateMarked{\velocity} + \soundSpeed$, one can easily see that the coefficient in front of $\mu$ is positive, so the infimum of the left-hand side is obtained at $\mu = -1$. The inequality is 
        \begin{equation*}
            (\latticeVelocity^2 +  \referenceStateMarked{\velocity}^2 -\soundSpeed^2)(\soundSpeed^2 - \referenceStateMarked{\velocity}^2)\geq 0
        \end{equation*}
        fulfilled if and only if the subsonic condition $|\referenceStateMarked{\velocity}|\leq \soundSpeed$ holds.
        Finally, analogous computations when $\latticeVelocity < \referenceStateMarked{\velocity} + \soundSpeed$ show that the inequality can never be fulfilled in this case.
        
        \item Let $\relaxationParameter = 2$, so $\varphi_2\equiv 0$. 
        Still, we can conclude without further computations, as our criterion is based only on the modulus of the eigenvalues, and eigenvalues (thus their modulii) are continuous in the entries of the matrix. 
        Quite the opposite, multiplicities can change abruptly with parameters.
    \end{itemize}

\section{Expression of the coefficients for \Cref{sec:transpD1Q3ShallowWater}}\label{app:sec:transpD1Q3ShallowWater}

Let 
\begin{align*}
    d_{-1}^2 &= -\frac{\gravity {\referenceStateMarked{\height}} {\relaxationParameter}}{2  {\latticeVelocity}^{2}} + \frac{1}{2}  {\relaxationParameter} - \frac{{\relaxationParameter} {\referenceStateMarked{\velocity}}}{{\latticeVelocity}} + \frac{{\relaxationParameter} {\referenceStateMarked{\velocity}}^{2}}{2  {\latticeVelocity}^{2}} - 1, \qquad
    d_{-1}^1 = -\frac{\gravity {\referenceStateMarked{\height}} {\relaxationParameter}}{2  {\latticeVelocity}^{2}} - \frac{1}{2}  {\relaxationParameter} + \frac{{\relaxationParameter} {\referenceStateMarked{\velocity}}}{{\latticeVelocity}} + \frac{{\relaxationParameter} {\referenceStateMarked{\velocity}}^{2}}{2  {\latticeVelocity}^{2}} + 1,\\
    d_{0}^2 &= \frac{\gravity {\referenceStateMarked{\height}} {\relaxationParameter}}{{\latticeVelocity}^{2}} - \frac{{\relaxationParameter} {\referenceStateMarked{\velocity}}^{2}}{{\latticeVelocity}^{2}} - 1, \qquad 
    d_{0}^1 =\frac{\gravity {\referenceStateMarked{\height}} {\relaxationParameter}}{{\latticeVelocity}^{2}} - {\relaxationParameter} - \frac{{\relaxationParameter} {\referenceStateMarked{\velocity}}^{2}}{{\latticeVelocity}^{2}} + 1, \qquad 
    d_{0}^0 = {\relaxationParameter} - 1, \\
    d_{1}^2 &= -\frac{\gravity {\referenceStateMarked{\height}} {\relaxationParameter}}{2  {\latticeVelocity}^{2}} + \frac{1}{2}  {\relaxationParameter} + \frac{{\relaxationParameter} {\referenceStateMarked{\velocity}}}{{\latticeVelocity}} + \frac{{\relaxationParameter} {\referenceStateMarked{\velocity}}^{2}}{2  {\latticeVelocity}^{2}} - 1, \qquad 
    d_{1}^1 =  -\frac{\gravity {\referenceStateMarked{\height}} {\relaxationParameter}}{2  {\latticeVelocity}^{2}} - \frac{1}{2}  {\relaxationParameter} - \frac{{\relaxationParameter} {\referenceStateMarked{\velocity}}}{{\latticeVelocity}} + \frac{{\relaxationParameter} {\referenceStateMarked{\velocity}}^{2}}{2  {\latticeVelocity}^{2}} + 1.
\end{align*}
Then, we have 
\begin{multline*}
    \coefficientsLaurentStableRootAllParities_1 = -d_{-1}^2, \qquad 
    \coefficientsLaurentStableRootAllParities_2 = -d_{-1}^1 - d_0^2 \coefficientsLaurentStableRootAllParities_1, \qquad
    \coefficientsLaurentStableRootAllParities_3 = -d_0^2 \coefficientsLaurentStableRootAllParities_2 - d_0^1 \coefficientsLaurentStableRootAllParities_1 - d_1^2\coefficientsLaurentStableRootAllParities_1\coefficientsLaurentStableRootAllParities_1, \qquad \text{and}\\ 
    \coefficientsLaurentStableRootAllParities_{\indexTime} = 
    -d_1^2  \sum_{k = 1}^{\indexTime - 2}\coefficientsLaurentStableRootAllParities_k \coefficientsLaurentStableRootAllParities_{\indexTime -1 - k} - d_1^1 \sum_{k = 1}^{\indexTime-3}\coefficientsLaurentStableRootAllParities_k \coefficientsLaurentStableRootAllParities_{\indexTime - 2- k} - d_0^1 \coefficientsLaurentStableRootAllParities_{\indexTime-2} - d_0^0 \coefficientsLaurentStableRootAllParities_{\indexTime-3}, \qquad \indexTime\geq 4,
\end{multline*}
and 
\begin{multline*}
    \coefficientsLaurentUnstableRootReciprocalAllParities_1 = -d_{1}^2, \qquad 
    \coefficientsLaurentUnstableRootReciprocalAllParities_2 = -d_{1}^1 - d_0^2 \coefficientsLaurentUnstableRootReciprocalAllParities_1, \qquad
    \coefficientsLaurentUnstableRootReciprocalAllParities_3 = -d_0^2 \coefficientsLaurentUnstableRootReciprocalAllParities_2 - d_0^1 \coefficientsLaurentUnstableRootReciprocalAllParities_1 - d_{-1}^2\coefficientsLaurentUnstableRootReciprocalAllParities_1\coefficientsLaurentUnstableRootReciprocalAllParities_1, \qquad \text{and}\\ 
    \coefficientsLaurentUnstableRootReciprocalAllParities_{\indexTime} = 
    -d_{-1}^2  \sum_{k = 1}^{\indexTime - 2}\coefficientsLaurentUnstableRootReciprocalAllParities_k \coefficientsLaurentUnstableRootReciprocalAllParities_{\indexTime -1 - k} - d_{-1}^1 \sum_{k = 1}^{\indexTime-3}\coefficientsLaurentUnstableRootReciprocalAllParities_k \coefficientsLaurentUnstableRootReciprocalAllParities_{\indexTime - 2- k} - d_0^1 \coefficientsLaurentUnstableRootReciprocalAllParities_{\indexTime-2} - d_0^0 \coefficientsLaurentUnstableRootReciprocalAllParities_{\indexTime-3}, \qquad \indexTime\geq 4.
\end{multline*}

\end{document}